\RequirePackage{fix-cm}
\documentclass[envcountsect,numbook,natbib]{svjour3_arxiv4}   
\smartqed  
\usepackage{graphicx}
\usepackage{bm}
\graphicspath{{./}}
\usepackage[utf8]{inputenc}
\usepackage{dsfont}
\usepackage{subcaption}
\usepackage{color}
\usepackage[hyphens]{url}
\usepackage{enumerate} 
\usepackage{booktabs}
\usepackage{tabularx}
\usepackage{ltablex}
\usepackage{mathptmx}
\usepackage{amssymb,amsmath,latexsym} 
\usepackage{mathtools}
\usepackage{longtable} 
\usepackage{textcomp}
\usepackage{lscape}
\usepackage{bbm}
\usepackage{eurosym}
\usepackage[font=footnotesize,labelfont=bf]{caption}

\usepackage[usenames,dvipsnames,svgnames]{xcolor}

\usepackage{algorithm}
\usepackage{algpseudocode}
\usepackage[algo2e]{algorithm2e}
\usepackage{mathtools}

\usepackage{listings}

\usepackage{pdfpages}
\usepackage[pdftex]{hyperref}
\usepackage{booktabs}

\usepackage{mathtools}
\mathtoolsset{showonlyrefs=true}
\usepackage{scalerel}

\usepackage{import}
\usepackage{multicol} 
\usepackage{nomencl}

\usepackage[nice]{nicefrac} 

\usepackage{framed} 
\usepackage[12pt]{extsizes}
\usepackage[normalem]{ulem}
\usepackage{geometry}
\usepackage{ifthen}
\usepackage{calrsfs}
\DeclareMathAlphabet{\pazocal}{OMS}{zplm}{m}{n}
\let\mathcal\pazocal

\usepackage[section]{placeins}
\usepackage{lscape}
\usepackage{amsmath,blkarray}
\usepackage{amsmath}               
\usepackage{amssymb}
\usepackage{xcolor}
\usepackage{ae,aecompl}
\usepackage{graphicx}
\usepackage{palatino}
\usepackage[T1]{fontenc}
\usepackage{rotating}
\usepackage{epsf}
\usepackage{setspace}
\usepackage{pdfpages}
\usepackage[utf8]{inputenc}
\usepackage[fulladjust]{marginnote}
\usepackage[]{todonotes} 
\usepackage[acronym,toc]{glossaries}	
\usepackage{hyperref}
\usepackage{float}

\usepackage{natbib} 
\usepackage{amsfonts}
\usepackage{dsfont}     
\usepackage{enumitem}
\usepackage{algorithm}
\usepackage{algpseudocode}

\usepackage{caption}
\usepackage{subcaption}
\usepackage{calrsfs}
\DeclareMathAlphabet{\pazocal}{OMS}{zplm}{m}{n}
\let\mathcal\pazocal

\definecolor{myred}{rgb}{0.8,0,0}  
\definecolor{markierung}{rgb}{0.7,1,0.7} 
\usepackage{soul}

\def \E{\mathbb{E}}
\def \R{\mathbb{R}}               
\def \I{\mathbb{I}}

\def \L{N_l}

\def \Err{D }

\def \M{N_j}

\newcommand{\one}{\mathds{1}}
\definecolor{burntsienna}{rgb}{0.91, 0.45, 0.32}
\definecolor{darksalmon}{rgb}{0.91, 0.59, 0.48}
\definecolor{lightsalmonpink}{rgb}{1.0, 0.6, 0.6}
\definecolor{melon}{rgb}{0.99, 0.74, 0.71}
\definecolor{pink-orange}{rgb}{1.0, 0.6, 0.4}
\definecolor{salmon}{rgb}{1.0, 0.55, 0.41}

\definecolor{xxx}{rgb}{0.85,0.85,0.85}
\definecolor{xx}{rgb}{0.99,0.78,0.68}
\definecolor{myred}{rgb}{0.9,0,0}
\definecolor{mygreen}{rgb}{0.0,0.5,0}
\definecolor{myblue}{rgb}{0,0,1}
\definecolor{lightblue}{rgb}{0.9,0.95,1}
\definecolor{mylila}{rgb}{0.7,0,1}
\definecolor{mygreen}{rgb}{0.0,0.5,0}

\newcommand{\condvar}{Q}

\newcommand{\D}{\displaystyle}

\newcommand{\condmean}{M}  
 
\newcommand{\Var}{\mathrm{Var}}

\newcommand{\argmin}{\mathop{\mathrm{argmin}}}

\newcommand{\card}{\mathop{\mathrm{card}}}
\newcommand{\supp}{\mathop{\mathrm{supp}}}
\newcommand{\Ec}{\mathcal{E}}
\newcommand{\Eq}{\widehat{\mathcal{E}}}
\newcommand{\B}{\mathcal{B}}

\newcommand{\m}{m}
\newcommand{\q}{q}
\newcommand{\z}{z}

\date{\today}

\begin{document}

	\title{Stochastic Optimal Control of an Epidemic Under Partial Information}
	
	\titlerunning{Stochastic Optimal Control of an Epidemic Under Partial Information}        

	\author{Ibrahim Mbouandi Njiasse \and Florent Ouabo Kamkumo   \and Ralf Wunderlich}
	\authorrunning{I.~Mbouandi Njiasse,  F.~Ouabo Kamkumo,  R.~Wunderlich}
	\institute{
		Brandenburg University of Technology Cottbus-Senftenberg, Institute of Mathematics, P.O. Box 101344, 03013 Cottbus, Germany;  \\
		\email{\texttt{Ibrahim.MbouandiNjiasse/Florent.OuaboKamkumo/Ralf.Wunderlich@b-tu.de} }}  
	
	\date{Version of  \today}

	\maketitle

	\begin{abstract}		 
			In this paper, we address a social planner's  optimal control problem  for a partially observable stochastic  epidemic model. The control measures include social distancing, testing, and vaccination. Using a diffusion approximation for the state dynamics of the epidemic, we apply  filtering arguments to transform the partially observable stochastic optimal control problem into an optimal control problem with complete information. This transformed problem is treated as a Markov decision process. The associated Bellman equation is solved numerically using optimal quantization methods for approximating the expectations involved to mitigate the curse of dimensionality.  We implement two approaches,  the first involves  state discretization coupled with linear interpolation of the value function at non-grid points. The second utilizes a parametrization of the value function with educated ansatz functions. Extensive numerical experiments are presented to demonstrate the efficacy of both methods.   
	\end{abstract}

	

	\newpage
	\setcounter{tocdepth}{3}
	\tableofcontents 
	\newpage
	\section{Introduction}
	In order to respond  to an epidemic  efficiently, policy makers have to assess the potential efficacy of various strategies for containing, mitigating, or even eradicating the disease. Such an assessment involves first the modeling of  plausible scenarios for the future course of the epidemic, second the  identification of  preventive  measures and pharmaceutical measures when  available, that  can  be implemented to limit the epidemic spread,  and third the evaluation of  the costs and   impacts  of intervention.  The optimal control of  an epidemic  has been subject to a  massive interest in the recent years, largely due to the  Covid-19 pandemic.  This procedure is of particular  relevance  because of the conflicting interests of the governments  between a safety-oriented approach (limiting the death  due to the disease) and  an economy-focused approach (avoid high economy burden).     The formulation of an optimal control problem for an  epidemic requires the description of the epidemic dynamics. 	Mathematical modeling   plays an important role in this process  by allowing to construct specific models that capture essential  features of the disease of interest.  The compartmental model, first proposed by  \citet{kermack1927contribution}, is the most widely used approach in epidemiology. These models divide the entire  population into  distinct compartments according to the state of health with respect to the disease under study. In that paper,  the resulting dynamics of the compartmental model is described by a system of ordinary differential equations (ODEs).  ODE models are a popular choice among authors  and deterministic optimal control for epidemic are built on top of such ODE models. Furthermore, extensions to  stochastic dynamics are possible  and more realistic, since they allow for the incorporation of  randomness that cannot be predicted during an epidemic.
	
	\paragraph{Literature review on optimal control of epidemics} When facing an epidemic outbreak,  measures that aim at least at containing the spread of the disease have to be  enforced,  since eradication or suppression is both   costly and  unrealistic. Quite often, cure and vaccine are not available, and authorities need to rely on non-pharmaceutical intervention (NPI) such as social distancing, home quarantine, lockdown of non-essential businesses and other institutions. Additionally,  preventive measures such as use of face masks may also be employed. This was pointed out for the recent Covid-19 pandemic by  \citet{kantner2020beyond},   \citet{hellewell2020feasibility}, \citet{ferguson2020report}  who conducted analyses to assess the efficacy  of NPI.  The most widely considered  NPI is social distancing, which allows for a reduction in social interactions and, consequently, the prevention of contact between susceptible and infectious individuals. The optimal level of  social  distancing  and  the timing  have been investigated by several  authors including 
	\citet{miclo2022optimal},  \citet{kruse2020optimal},  \citet{behncke2000optimal},  \citet{nowzari2016analysis}, \citet{federico2021taming}, \citet{alvarez2020simple},   \citet{calvia2024simple} and \citet{charpentier2020covid}. More specifically, \citet{kruse2020optimal}  introduced a parameter that is  controlled  by the planer allowing to adjust the lockdown level. Given that  they considered linear cost  in the number of infected, the resulting optimal control strategy  is of the  bang-bang type.  \citet{miclo2022optimal}  provided an analytic policy for their optimal control problem  with ICU constraints.  \citet{alvarez2020simple}  numerically characterized the optimal lockdown policy for a SIR model.  \citet{charpentier2020covid} considered a control problem with asymptomatic  compartments and provided numerical solution to  the optimal policy. Moreover,  \citet{acemoglu2021optimal} proposed  a targeted lockdown on a more vulnerable  age-grouped model, demonstrating  that  targeted lockdowns  policies outperform indiscriminated policies.  
	
	The main  objective of these NPI is to flatten the infection curve  in order to avoid overwhelming of the healthcare system's capacity. In this regards, \citet{charpentier2020covid}, \citet{miclo2022optimal},  \citet{avram2022optimal}  consider optimal control problems  with ICU constraints. \citet{miclo2022optimal}  observed  that flattening the curve strategy is suboptimal  and suggest the so called  ''filling the box'' strategy  which consists of  four phases. First, the spread of the disease is let freely and then as the ICU limit is approached, a strong lockdown  is applied.  After the lockdown, regulations are gradually lifted maintaining the rate of infection  constants  and in the last phase all regulations are lifted.   \citet{kantner2020beyond}, and \citet{ferguson2020report}  recommended exploring  NPI strategies that goes beyond flattening the curve.  
	
	Beside NPI, vaccination has been considered  as a highly  effective strategy for eradicating, or at the very least containing, an epidemic. The smallpox vaccine has  been successful   in eradication  and the vaccination has reduced death due to measles by $94 \% $.  This explain  why many studies on epidemic control have  focused on  vaccination, see  \citet{hethcote1973optimal}, see also, \citet{ledzewicz2011optimal},  \citet{hu2013optimal}, \citet{laguzet2015global}. In some authors' models, vaccinated  people  are grouped  in  an additional  compartment  (\citet{ishikawa2012stochastic}, \citet{garriga2022optimal}).  Furthermore, \citet{federico2024optimal} examined  an optimal  vaccination  problem  with incomplete immunization power of  the vaccine, analyzing the problem  with the dynamic programming approach.   \citet{hansen2011optimal} investigate the use of isolation and vaccination  in a model  incorporating  limited resources. In some models, the arrival time of the vaccine is random, as observed in \citet{garriga2022optimal} and  \citet{federico2024optimal}. 
	
	When it comes to solution approaches to  epidemic control problems, the Pontryagin maximum principle is widely employed. This requires convexity  conditions on the objective function. Non-convex (in the state dynamics  or in the objective  function) optimization problems   are notoriously difficult to study  with the maximum principle approach.  These convexity conditions restrict the class of  functions that can  be included in the cost functional. However,  \citet{calvia2024simple} noted that the dynamic programming  approach can be profitably applied to the optimal control problems for epidemics with less restrictive conditions on the objective functional. Therefore,  they  analyzed a non-convex epidemic control problem  using  the dynamic programming approach. They  proved a continuity property of the value function  and showed that the value function  is  a solution  of the HJB equation  in the viscosity sense. Similarly,   \citet{federico2021taming,federico2024optimal}   used the dynamic programming technique  and  the viscosity solution concept.

	\paragraph{Problem statement}  In this study, we consider a population which is divided into classes according to the state of health concerning  a given disease. The classes are susceptible $S$, non-detected infected $I^-$, detected infected $I^+$,  non-detected recovered $R^-$, detected recovered $R^+$, and  hospitalized $H$. This model is denoted  by $SI^{\pm}R^{\pm}H$ and transitions between compartments are depicted in Figure  \ref{figcompartmt}. This flowchart is similar to that of   \citet{charpentier2020covid} which accounts for asymptomatic  compartments with unobservable compartment sizes. Such unobserved sizes could also result   from unreported cases of infection during   an epidemic.  \citet{al2020asymptomatic}  and  \citet{day2020covid}  highlight the potentially   high rate of asymptomatic individuals for the Covid-19. A more detailed description of the the model in Figure \ref{figcompartmt} is provided in Section \ref{sect2}. The dynamics of our model are described by a  stochastic process using  the approximation of a continuous time Markov chain (CTMC) by  a diffusion process as presented in     \citet{britton2019stochastic} and in our papers \citet{sem,pif}. Moreover, we assume a situation where the vaccine is already available  and thus the social planner could select three  control measures, namely social distancing, testing to  detect more infectious individuals   and vaccination.
	
	\begin{figure}[h]
		\begin{center}
			\includegraphics[width=5in]{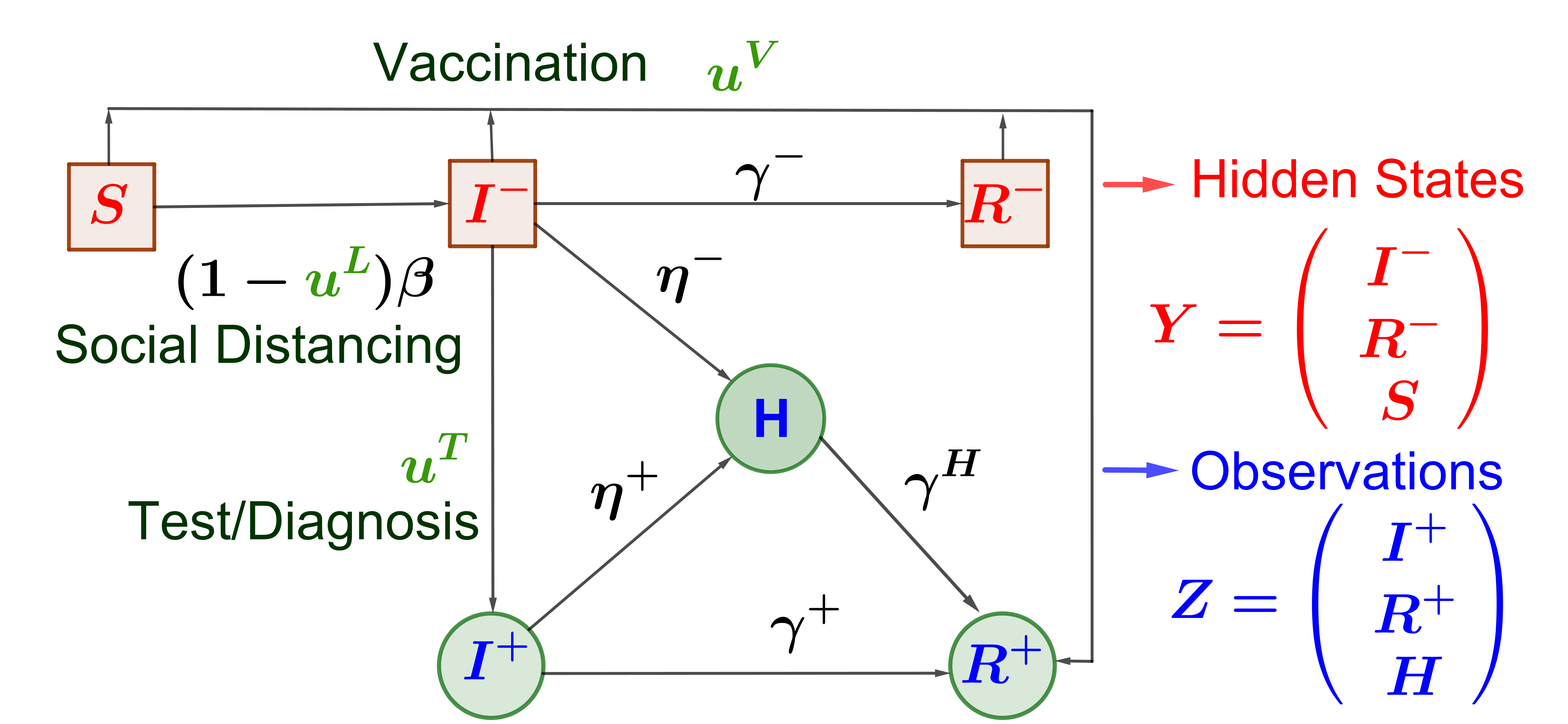}		
		\end{center}
		\caption{Flowchart of the $SI^{\pm}R^{\pm}H$ compartmental model with control variables $u^L,u^T $ and $u^V$. Hidden compartments are in red and observable compartments in blue. }
		\label{figcompartmt}
	\end{figure}
	
	We then formulate the social planner's control problem by constructing realistic costs that account for both the economic and sanitary impacts of intervention measures. These costs include fixed costs that the government must bear as long as a specific control measure is in place. Additionally, the primary component of each cost is linear, and depends on the number of people affected by the particular control measure. The final component is a quadratic penalty applied when a certain threshold is exceeded, making it less likely for strategies that lead to exceeding  this threshold to be  selected.
	The resulting problem is  a stochastic optimal control problem (SOCP) with partially  observed  state process $X_n=(I^-_n, R^-_n, I^+_n , R^+_n, H_n)^\top, n=0, \ldots, N_t$,  where $N_t$ is the time index of the time horizon,  that is also called SOCP under partial information. We employ   the innovations approach of stochastic filtering  to transform the initial SOCP  with  partially observed state process into a corresponding  SOCP  with fully observed  state process by  replacing  the hidden state components  $Y_n=(I^-_n,R^-_n)^\top$ by its extended Kalman filter. The latter is characterized by the approximations of  the  conditional  mean   $\condmean _n$  and conditional covariance   $\condvar_n$. A more detailed examination of these concepts can be found in Section \ref{sectfil}. 
	We then consider the problem as a Markov decision process (MDP)  and derive the associated Bellman equation. 
	
	\paragraph{Our  contribution} The originality of this model lies in the fact that it allows to describe diseases with asymptomatic individuals, even though the number of asymptomatic individuals is not observable, and to address the problem of "flattening the curve" with a modeling of hospital facilities. Decision makers need to ensure that the maximum capacity of hospital equipment is not reached to avoid the ethical triage problem.   In addition to formulating the SOCP of an epidemic with a special feature, namely, the partial information aspect, we provide numerical solutions to the SOCP, which suffers from the curse of dimensionality due to an 8-dimensional state process and a 3-dimensional control. To address this challenge, the Bellman equation is solved numerically through state space discretization. The conditional expectation in the Bellman equation is then approximated by quantization techniques, which involve replacing a continuous random variable with an appropriate  discrete random variable. Two implementations of the backward recursion algorithm are performed. The first uses state space discretization, quantization and linear interpolation of the value function, while the second combines state space discretization, quantization and least squares parametrization of the value function with well-chosen ansatz basis functions.	
	Both implementations yield similar results, providing the optimal decision rule and the associated optimal cost or value function. Simulations of the epidemic path, using the derived optimal control strategies, demonstrate their effectiveness in containing the spread of the epidemic. 
	
	\paragraph{Paper organization}   Section \ref{sect1} details the compartmental model, control measures, and the stochastic dynamics of the compartment sizes. Section \ref{sect2} focuses on the formulation of the SOCP under partial information using the filtering argument. In Section \ref{sect3}, we present numerical methods based on the backward recursion algorithm and optimal quantization. Section \ref{sect4} features extensive numerical experiments performed using both approaches. Finally, appendices  compile some intermediate numerical experiments that were omitted from the main text.

	\section{Model description}\label{sect1}
	
	\subsection{Compartmental model}
	
	Compartmental models are the most powerful tools for monitoring and predicting the evolution of an epidemic.
	Each epidemic must be described by a specific  compartmental model that  captures all the important features of the epidemic. Therefore, 
	we consider the $SI^{\pm}R^{\pm}H$  compartmental epidemic model      in  Figure \ref{figcompartmt},	
	which is an extension of the classical and well-known $SIR$ model.  In this modified model, the total population $N$, that we assume to  be constant throughout the study period, is divided into compartments according to the epidemic status of individuals within the population. Therefore, the susceptible compartment, denoted by $S$, contains all those who could contract the disease and become infectious. To capture the peculiarities of diseases with asymptomatic individuals, the infectious compartment is split into  undetected  infectious $I^-$ and detected infectious  $I^+$. Similarly, the recovered compartment is split into  undetected recovered $R^-$ and  detected recovered  $R^+$. We also include the hospital group $H$, which includes hospitalized individuals and those in special hospital units, such as intensive care units (ICU). This model was inspired by the paper of \citet{charpentier2020covid} and we refer  the reader  to \citet{sem} for a more detailed description of such a model.

	Several transitions with corresponding intensities are considered. The transition rates are the reciprocal of the average sojourn time in a given compartment and allow to get the average number of transitions per unit of time.  We assume that detected infectious persons $I^+$ are well isolated and cannot infect susceptible persons, so that new infections are only due to contacts between a susceptible and a non-detected infectious person. The latter transition occurs with intensity $\beta > 0$. This means that on average $\beta S {I^-}/{N} $ number of people become infectious per unit of time when there are no control measures. Moreover, a non-detected infectious person can be detected with detection rate $\alpha> 0$. If the condition of an infected person becomes severe, he/she  will be hospitalized with rate $\eta^-> 0$ if he is not detected, or at a rate $\eta^+ > 0$ if he is detected.	
	Alternatively, infected individuals recover with intensities  $\gamma^-> 0$ if they are not detected, $\gamma^+> 0$ if they are detected, and $\gamma^H> 0$ if they recover in a hospital. 	Finally, in this model we assume that the vaccine is available. Through  vaccination, susceptible, non-detected infectious or non-detected recovered individuals can be immunized against the disease and move to the class of detected recovered individuals.  In addition, this epidemic flowchart has some non-observable or hidden states $S$, $I^-$, and $R^-$, namely, because we cannot distinguish between  asymptomatic or unreported individuals and susceptible individuals. Hence, we cannot  collect data on the number of people in these compartments at any given time. The remaining compartments will be observable  since  inflows and the outflows of these compartments can mostly  be recorded. In this model, we have made  the assumption  that those who recover from the disease or are vaccinated against it acquire lifelong immunity. At least we assume that the immunity lasts until the time horizon of the study. For this reason, the compartment $R^+$ is an absorbing state. Consequently, this model can describe a disease with lifelong immunity or the early stage of a disease that provides immunity after recovery that lasts longer than the time horizon of the study.

	\subsection{Control measures}
	The management of an epidemic can be summarized as  a  problem of cost-optimal containment of the  epidemic through an appropriate mix of measures. Decision-makers  need to know on which levers they can act to mitigate the spread of the disease within the population. For this epidemic model, we are investigating  some potential measures   that  will be divided into three types: social distancing measures or lock-down, detection  or testing measures, and vaccination strategy.

	\subsubsection{ Social distancing or lock-down} Historically, quarantine has been employed to isolate infected individuals and prevent them from contaminating those who  have  not yet been infected. However, this type  of measure is more effective in the early phase of an epidemic and becomes more difficult to implement on a larger scale. Furthermore, the economic cost of  such a policy may be too high. Instead, for a large-scale epidemic,  we can consider social distancing measures   with different levels of restrictions. The level of social distancing will be quantified   with the control variable $u^L$ which  ranges from $0$ to $1$, with $0$    corresponding to the absence of social distancing,  while the value $1$  represents  a strict  and perfect lock-down (which is not achievable in reality). These control variables involve other NPI measures such as  wearing masks in public places and mobility restrictions. The implementation of the control variable $u^L$ will result in a reduction of the transmission rate from $\beta$ to $(1-u^L) \beta$.		
	\\
	
	\subsubsection{   Detection or testing}  Accurate determination of the actual number of infectious individuals is of critical relevance in the assessment of the current state of the epidemic and its future evolution. A good testing strategy allows for the identification, tracking, and isolation of infected individuals if necessary. The testing control variable may range from a minimum testing rate $u^T_{min} \geq 0 $ when almost no specific testing policy is applied, to the maximum possible testing strategy $u^T_{max}$ that policymakers can afford. The lower limit  $u^T_{min}$ may eventually be zero if no detection measures are applied. In the model, it is assumed that the expected  number of detections per unit of time is proportional to the testing rate $u^T$. In other words, the greater the number of tests performed, the greater the proportion of non-detected infected individuals who are detected.

	\subsubsection{   Vaccination} The most effective method for  eradicating  an epidemic is through a large-scale  vaccination campaign. In the fight for  total eradication of the disease, the government has as one of its main objectives, to make a vaccine available to the population for prevention. When designing a vaccination strategy, several questions  arise such as  the quantity of vaccine to be produced and also  the vaccination coverage to be established. Unfortunately, planning and executing these strategies have  financial costs. The decision maker must consider the problem of controlling the rate of individuals to be vaccinated at a given time. Before proceeding, we introduce $u^V \geq 0$, a control variable that controls the average intensity of individuals who receive the vaccine per unit of time. In  addition, the upper limit  of the vaccination intensity that the policymakers can afford to have in place  is denoted by  $u^V_{max}$.

	\subsection{Stochastic dynamics of the compartmental model}
	To describe the dynamic 	of the above  compartmental model in Figure \ref{figcompartmt}, we adopt the continuous-time Markov chain  $(CTMC)$ approach  as described in \citet{anderson2011continuous}; \citet{britton2019stochastic}; \citet{guy2015approximation} and \citet{sem} . This approach was initially developed to model chemical reaction networks. The dynamics of compartment sizes are  described by  counting processes which are first replaced by  a Poisson process with state dependent intensities. Then, considering the relative subpopulation sizes for $N$ large, the Markov jump process is approximated by a diffusion process, for which the error between both infinitesimal generators is of order $o({1}/{N})$, see \citet{guy2015approximation}.  By applying this method, we derive the following diffusion approximation for the $SI^{\pm}R^{\pm}H$ model. It should be noted that for the sake of simplified  notation, we  omitted the time argument of each component of the state variable in the drift and diffusion coefficients. These coefficients may depend explicitly on the time variable  through the infections rate  that can  be considered time dependent  $\beta= \beta(t)$. 	
	If we denote the state variable  $X=(Y^\top,Z^\top)^\top$, with $Y=(  I^{-},R^{-},S)^\top$ and $Z=(I^{+},R^{+},H)^\top$,
	which components represent  absolute  subpopulation size  in each  compartment, then it  satisfies the following  stochastic differential equation (SDE), where the component $S(t)=Y_3(t)$ is omitted and  can be recovered by normalization $Y_3=N-(Y_1+Y_2+Z_1+Z_2+Z_3)$ because we assume that the population size remains constant.

	{\footnotesize 
		\begin{eqnarray*}\label{DA}
			d\left(\begin{array}{c}
				Y_{1}\\
				Y_{2}\\	[0.4ex]
				Z_{1}\\
				Z_{2}\\
				Z_{3}\\
			\end{array} \right) \hspace{-0.2cm}
			&=&\hspace{-0.2cm}\left(\begin{array}{c}		
				\frac{ (1-u^L)\beta Y_{1}}{N} (N-(Y_1+Y_2+Z_1+Z_2+Z_3))-(\gamma^{-} + \eta^-+u^T+u^V)Y_{1} \\
				\gamma^{-}Y_{1}-u^V Y_{2}\\ [0.4ex]
				u^T Y_{1}-(\gamma^{+}+\eta_{}^{+})Z_{1}\\
				\gamma^{+}Z_{1}+  \gamma^{H}Z_3 + u^V( N-(Z_1+Z_2+Z_3))\\
				\eta^{-}Y_{1}+\eta_{}^{+}Z_{1}- \gamma^{H}Z_{3}							
			\end{array} \right)dt + \left(\begin{array}{cc}
				\overline{\sigma}& \overline{g}\\
				0&\overline{\ell}
			\end{array}\right) d W					
		\end{eqnarray*}
	}
	
	such that when the state is split into hidden state $Y$  and observation $Z$, the diffusion approximation can be expressed as follows:
	\begin{align}
		dY(t)&=\overline{f}(t,Y(t),Z(t),u(t))dt+ \overline{\sigma} (t,Y(t),Z(t),u(t))dW^{1}(t) +\overline{g}(t,Y(t),Z(t),u(t)) dW^{2}(t),& 
		\\[0.5ex]
		dZ(t)&= (\overline{h}_0(t,Z(t),u(t))+\overline{h_1}(t,Z(t),u(t))Y(t))dt + \overline{\ell}(t,Y(t),Z(t),u(t))dW^{2}(t).& 
	\end{align}	
	The drift matrices are defined as:\\
	~\\
	\begin{tabular}{l}
		$ \overline{f}(t,Y,Z,u)= \left(\begin{array}{c}		
			\frac{ (1-u^L)\beta Y_{1}}{N} (N-(Y_1+Y_2+Z_1+Z_2+Z_3))-(\gamma^{-} + \eta^-+u^T+u^V)Y_{1} \\
			\gamma^{-}Y_{1}-u^V Y_{2}\\ 				
		\end{array} \right),$\\[4ex]
		$\overline{h}_0(t,Z,u)~~= \left(\begin{array}{c}		
			-(\gamma^{+}+\eta_{}^{+})Z_{1}\\
			\gamma^{+}Z_{1}+  \gamma^{H}Z_3 + u^V( N-(Z_1+Z_2+Z_3))\\
			\eta_{}^{+}Z_{1}- \gamma^{H}Z_{3}		
		\end{array} \right),$\\ [4ex]
		$\overline{h_1}(t,Z,u)~~= \left(\begin{array}{cc}		
			u^T &0\\
			0&0\\
			\eta^{-}&0							
		\end{array} \right). $
	\end{tabular}

	The  matrices in the diffusion coefficient are defined as \\[.1ex]
	
	\begin{tabular}{lcl}		
		$	\overline{\sigma} (t,Y,Z,u)$&=& $\left( \begin{array}{cc}
			\sqrt{\frac{(1-u^L)\beta Y_{1}}{N}(N-(Y_1+Y_2+Z_1+Z_2+Z_3))} & -\sqrt{\gamma^{-}Y_{1}} \\
			0& \sqrt{\gamma^{-}Y_{1}}\\
		\end{array} \right),$ \\[4ex]					
		$\overline{g}(t,Y,Z,u)$&$=$&				
		$ \left( \begin{array}{*{8}{c}}
			-\sqrt{u^{T}Y_1}&0&0&-\sqrt{u^V Y_1}&0&-\sqrt{\eta^-Y_1}&0&0 \\
			0&0&0&0&- \sqrt{ u^V Y_2}&0&0&0
		\end{array} \right),$				
	\end{tabular}

	~\\
	
	\begin{tabular}{l}				
		
		\hspace{-.4cm}$	\overline{\ell} (t,Y,Z,u)$ \hspace{.8cm}$=$\\	[2ex]									
		\hspace{-.6cm}\footnotesize{$\left( \begin{array}{c@{\hspace*{0.2em}}c@{\hspace*{0.1em}}c@{\hspace*{0.2em}}c@{\hspace*{0.2em}}c@{\hspace*{0.1em}}c@{\hspace*{0.1em}}c@{\hspace*{0.2em}}c@{\hspace*{0.2em}}}
				\sqrt{u^{T}Y_1}&-\sqrt{\gamma^{+}Z_{1}}&0&0&0&0&-\sqrt{\eta^+ Z_1}&0\\						
				0&\sqrt{\gamma^+ Z_1}&\sqrt{u^V(N-(Y_1+Y_2+Z_1+Z_2+Z_3))}&\sqrt{u^V Y_1}&\sqrt{u^V Y_2}&0&0&\sqrt{\gamma^H Z_{3}}\\
				0&0&0&0&0&\sqrt{\eta^- Y_1}&\sqrt{\eta^{+}Z_{1}}&-\sqrt{\gamma^H Z_{3}}	\end{array} \right),$ }
	\end{tabular}
	
	~\\
	and\\

	$dW=\left(	dW_{1}, \ldots,	dW_{10}\right)^\top ,$  for each transition, an independent Brownian motion is associated.   	
	
	\vspace*{0.25cm}

	In the above SDE,  we  have split the state vector  into non-observable components $Y$   and observable components $Z$. We should also note that the independent components of the  $10$-dimensional Brownian motion  which depicts the noise of transitions between compartments are split into  Brownian motion that appear in the observation $Z$, $dW^1=\left(\begin{array}{*{2}{c@{\hspace*{0.35em}}}}
		dW_{1},&
		dW_{2}			\end{array} \right)^\top $  	   and  those which  are only involved in the hidden state   $dW^2=\left(dW_{3}, \ldots,	dW_{10}
	\right)^\top $. 
	
	Functions $\overline{f},~ \overline{\sigma},~\overline{g}$ and $ \overline{\ell}$  are clearly nonlinear in the hidden state $Y$.
	This nonlinearity  prevents  the use of the standard Kalman filter approach when it comes to filtering problem.  
	Since we aim to  solve the optimal control problem which is numerically untractable  in continuous-time for high-dimensional state,  we perform a time-discretization to mitigate  the curse of dimensionality. Moreover,  real-world  data for  an epidemic  are also   collected  in discrete time. Therefore, the Euler-Maruyama scheme, with time discretization $t_n= n \Delta t$, $n=0,\ldots,N_t$, is employed to derive the discrete dynamics
	\begin{align} \label{DDA}
		Y_{n+1}&=Y_n+ f(n,Y_n,Z_n,u_n)+ \sigma (n,Y_n,Z_n,u_n)\B ^{1}_{n+1} +g(n,Y_n,Z_n,u_n) \B^{2}_{n+1},& 
		\\[0.5ex]
		Z_{n+1}&= Z_n  + h_0(n,Z_n,u_n) + h_1(n,Z_n,u_n) Y_n+  \ell(n,Y_n,Z_n,u_n)\B^{2}_{n+1},& 
	\end{align}			
	\noindent where  $(\B^1_n)$ and $(\B^1_n)$   are independent sequences of i.i.d. $ \mathcal{N}(0,\I)$ random vectors. Note that $\I$ is the identity matrix with the corresponding dimension. In this discrete dynamics, terms with $\Delta t$  are incorporated into functions $ f, ~ \sigma,~ g ,~ \ell,~ h_0, \text{and}~ h_1 $ as we have for instance,  $f(n,\cdots)=\overline{f}(t_n,\cdots) \Delta t$, $\sigma(n,\cdots)=\overline{\sigma}(t_n
	,\cdots)\sqrt{\Delta t}$ and similarly for the other coefficients. The time discretization is subject to discretization errors that can destroy the normalization. The loss of this property can lead to negative values for the components of $Y$ and $Z$, or to values greater than $N$. To deal with this issue, it is often necessary for the numerical implementation to set minimum and maximum values for the compartment sizes.

	\subsection{Partial information and filtering} \label{sectfil}
	
	A  particular  feature  of our problem  is that the  state process $X$		is  not  fully  observable. We will  cope with  this  issue  using  the  projection to the observable filtration.	
	Let us  consider  our state process $X=(Y^\top_n,Z^\top_n)^\top, ~n=0,\ldots,N_t$,  described by the  discrete dynamic \eqref{DDA} mentioned above.  The filtering problem  consists in sequentially estimating  the hidden state $Y_n$  given the observation process $(Z_j),~j=0,\ldots,n$ up to time $n$.  This filtering problem is  non-standard, first, because the signal drift   functions  $ f $  is non-linear in the  hidden  state $Y$ and second, because  the diffusion coefficients $\sigma,~ g ,~ \ell$ depend on the hidden signal  $Y$.
	To solve  this non-linear filtering problem, we use the extended Kalman filter approach. This allows us to derive an approximation to  the optimal filter which is characterized by the conditional mean $\condmean_n : =\mathbb{E}\left( Y_n\vert \mathcal{F}^Z_n \right)$ and the  associated estimation error that is the conditional covariance matrix \[\condvar_n := \Var(Y_n|\mathcal{F}^Z_n) =\mathbb{E}[(Y_n-\condmean_n)(Y_n-\condmean_n)^\top \vert \mathcal{F}^Z_n ].\] 
	
	Using  the Kalman filter for conditional Gaussian sequences, see   \citet[Theorem 13.4]{liptser2013statistics}, and the extended Kalman filter procedure, the dynamics of  $\condmean$ and $\condvar$   are   given by the following proposition.
	
	\begin{proposition}
		The extended Kalman filter approximation of  the    conditional mean 	$\condmean_{n} $ and the conditional covariance  $ \condvar_{n}$  of the conditional distribution of $Y_n$ given  $\mathcal{F}^Z_n$  defined by equations \eqref{DDA}
		solve the following recursions  driven by the observations
		
		\begin{eqnarray}
			\condmean_{n+1}\hspace*{-0.05cm}&=& \hspace*{-0.05cm}  \condmean_n+f(n, \condmean_{n}, Z_n ) ~+ ~\left[ g\ell^\top+f_1 \condvar_{n} h_1^\top \right]\left[\ell\ell^\top + h_1 \condvar_{n} h_1^\top\right]^{+} \nonumber  \big[ Z_{n+1}- \big( h_0+h_1 \condmean_{n} \big) \big] \\
			{ \condvar}_{n+1}\hspace*{-0.05cm}&=& \hspace*{-0.05cm} -\left[ g\ell^\top+
			f_1 \condvar_{n} h_1^\top \right]\nonumber \left[\ell\ell^\top + h_1 \condvar_{n} h_1^\top\right]^{+}  \left[ g\ell^\top+ f_1 \condvar_{n} h_1^\top  \right]^\top 
			+f_1 \condvar_{n}f_1^\top+ \sigma\sigma^\top \label{eqvar}
		\end{eqnarray}			
		where the function $f_1=\mathbb{I}+ \dfrac{\partial f}{\partial y}$  with $\mathbb{I}$ the identity matrix of the dimension of $Y_n$, and functions $\sigma,~ g,~\ell$ have as argument $(n,\condmean_n,Z_n)$.
		In addition, the initial information $\sigma$-algebra  $\mathcal{F} ^ I_0$ is such that the conditional distribution of $Y_0$ given $\mathcal{F} ^ I_0$  is $\mathcal{N}(\condmean_0,\condvar_0)$.	
		Further, $[A]^+$ denotes the pseudoinverse of the matrix $A$.
		
	\end{proposition}
	
	\begin{proof}
		This result is a direct application of the Extended Kalman filter approach  through linearization of the discrete dynamic with respect to the hidden variable $Y$ and applying  kalman filter for conditional Gaussian sequences, in \citet[Theorem 13.4]{liptser2013statistics}. More details about the derivation of this result can be found in \citet{pif} \qed 
	\end{proof}
	
	As we can see, the recursion for $\condvar$  is a quadratic recursion, often referred to  as the Riccati equation in the filtering context. The recursion for  $\condmean$ is driven by the difference between the actually  observed value $Z_{n+1}$ and the expectation of the observation 
	given $\mathcal{F}_n^{Z}$. The latter difference drives the so-called innovation process
	$(\Ec_n) $ which has the form
	\begin{equation*}
		\Ec_{n+1} = ([\ell\ell^\top + h_1 \condvar_{n} h_1^\top]^+)^{1/2}\big(Z_{n+1}-(h_0+h_1 M_{n} )\big).
	\end{equation*}

	It has been shown in \citet[Theorem 13.5 ]{liptser2013statistics} that the sequence 	$(\Ec_n) $ is an  i.i.d.   sequence of   $\mathcal{N}(0,\mathbb{I})$ random variables  with  respect to the observable  $\sigma$-algebra   $\mathcal{F}_n^Z=\mathcal{F}_n^{\Ec} \vee \mathcal{F}_0^I$, where $\mathcal{F}_0^I$ is the initial information $\sigma$-algebra.
	
	Therefore, the  processes $\condmean $ and $Z$  can be expressed in terms of the innovation process $(\Ec_n) $ and they  are as follows:		
	\begin{eqnarray}
		\condmean_{n+1}\hspace*{-0cm}&=& \hspace*{-0cm}  \condmean_n+f(n, \condmean_{n}, Z_n ) ~+ ~\left[ g\ell^\top+f_1 \condvar_{n} h_1^\top \right] \Big(\left[\ell\ell^\top + h_1 \condvar_{n} h_1^\top\right]^{+}\Big)^{\frac{1}{2}} \Ec_{n+1}, \label{eqmean}\\ [.5ex]		
		Z_{n+1}&=&h_0+h_1 M_{n} +\left[\ell\ell^\top + h_1 \condvar_{n} h_1^\top \right]^{1/2}\Ec_{n+1}. \label{eqobs}  
	\end{eqnarray}
	
	The dynamics of $\condmean$ and $Z$ are deduced directly from the definition of the innovation sequence. This new form of the dynamics of the conditional mean $\condmean$ and the observation $Z$  will be essential to transform the optimal control problem with a partially observed state process into one with a fully observed state process.

	\section{Formulation of the optimal control problem}\label{sect2}
	
	The optimization problem we wish to address is that of a social planner attempting to influence the future course of an epidemic by implementing certain measures.    The implementation of  these measures  requires both  human and financial resources. To assess the performance of these measures, decision-makers must define metrics that depend not only on the implementation cost, but also on the future impact of these measures on the number of infected individuals.
	\subsection{Objective  function }
	The aim of this section is to model the performance criterion of our optimization problem.  The idea is to have a performance criterion that is as realistic as possible. Unlike many papers in the literature that consider quadratic cost mostly because it is mathematically convenient, we want to have costs that 
	depend on the number of individuals affected by a certain counter-measure. In reality, this dependence can be described well by linear functions, if the numbers do not exceed certain capacity limits. Beyond these limits, penalties are added, which can be described by increasing and convex functions.  For example, to be able to account for limited capacities for testing, vaccination as well as the limited capacity of hospital equipment,		
	quadratic penalties can be considered, to make values of control leading to reaching capacity limits less attractive. Another important aspect of epidemic management is the consideration of fixed costs that the social planner must bear as soon as he decides to implement control measures. One such cost is the cost of equipment, facilities and staff required for testing and vaccination, which does not depend directly on the number of individuals tested or vaccinated per unit of time. However, these costs must be supported regardless of the exact level of testing and vaccination given that this level is within a certain range. In our formulation, costs functional will primarily  be described under full information,  and then costs under   partial information will be obtained by taking the conditional expectation w.r.t. the observation $\sigma$-algebra.    
	\subsubsection{Running costs}
	
	It follows that  our running costs should 
	take into account  the various costs related to the treatment of infected individuals, social costs resulting from  the application of  containment measures, detection measures and vaccination of people.	Our main goal  is to reduce   the number of infected individuals as much as possible  over time while simultaneously minimizing  the costs of implementing all the social distancing measures, hospitalization costs, testing and vaccination costs, as well as the economic impact of interventions.
	
	To  define  our running costs,   we consider  functions  of the following   form depending on 	
	the number $x$ of ``affected'' individuals (per unit of time), and  
	penalties if a certain  capacity threshold $\overline x$ is exceeded
	\begin{equation}
		C_k:~\mathbb{R}_+^2~ \rightarrow~ \mathbb{R}_+,~~
		(x,\overline{x})~ \mapsto  ~ C_k (x,\overline{x})
	\end{equation} 
	
	It is reasonable  that these functions be increasing and convex with respect to the variable $x$.  Examples of such functions include those of the form
	$$C_k (x,\overline{x})= 
	\begin{cases}
		\overline{
			a}_k \one_{\{x >0\}} +a_k x, & x\le \overline x,\\[0.5ex]
		\overline{
			a}_k \one_{\{x >0\}} +a_k x+ b_k (x-\overline{x})^2, & x> \overline x,
	\end{cases} $$
	where $\overline{
		a}_k$,  $a_k$   and $b_k$  are positive   constants, with $\overline{
		a}_k\one_{\{x >0\}}$ which stands for the fixed cost,  $a_k$  is the cost per affected person by unit of time and $b_k$ can be viewed as the penalty weight. The main feature of the fixed costs are to  vanish if the control value is zero, otherwise they are constant  and therefore do not  depend on the size of $u_k$ implemented.   Extreme cases for this cost function  are  a purely linear cost  if the threshold $\overline{x}$  is infinity, a purely quadratic  cost if  the threshold  $\overline{x}$  is 0, and no fixed cost, $\overline{a}_k=0$. An illustration of the typical curve for such a cost  function is given in Figure \ref{figcost}. 
	
	\begin{figure}[h]
		\begin{center}
			\includegraphics[width=3.5in]{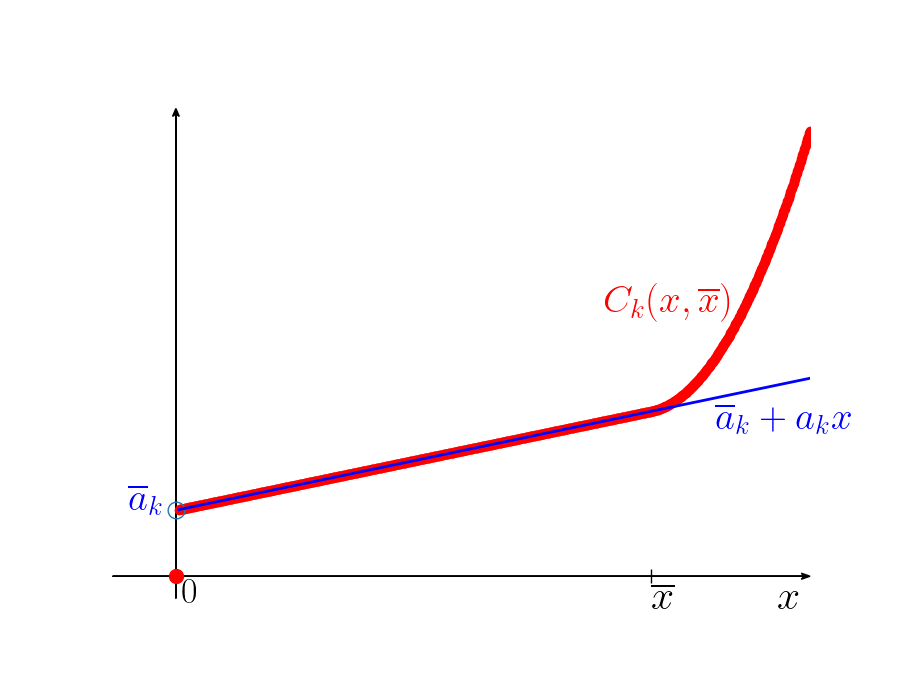}		
		\end{center}
		\caption{Typical curve of running costs functions featuring the fixed cost jump for $x>0$, linear cost in the interval $(0,\overline{x})$ and linear plus quadratic cost for $x> \overline{x}$.   }
		\label{figcost}
	\end{figure}

	\begin{enumerate}
		\item 	\textbf{Economic costs:}
		At a given time $t$,  individuals who are primarily concerned by social distancing measures are susceptible individuals. Besides asymptomatic people, namely  
		non-detected  infectious $(I^{-})$ and non-detected  recovered $(R^{-})$ are concerned as well since their infection status is not known. Furthermore,  since the social distancing measures usually  apply  even to recovered individuals,  we define those who are directly  affected by restrictive measures at time $n$  as 
		$X^{\text{work}}_n=S_n+I^{-}_n+R^{-}_n+R^{+}_n$. Using the normalization,  $X^{\text{work}}_n=N-I^+_n-H_n$ becomes observable and  represents the labor force within the population which are not  in quarantine ($I^+$).  
		When restrictive measures with strength $u^L>0$  are implemented, the proportion     
		$u^L X^{\text{work}}$  of individuals is unable to fulfill their role in the social economy as a labor force. Consequently,  the loss of  productivity due to   restriction measures  will depend on the number of  people  concerned by  the lock-down. In the literature, as in  \citet{charpentier2020covid,  federico2024optimal,federico2021taming} a quadratic function or more generally, a convex function  is commonly used to describe the socio-economic cost because of its mathematical tractability.
		Therefore, we define  the  social distancing   running cost at time $n$ by  
		\begin{equation*}
			C_L \left(u^L_n X^{ \text{Work}}_n,0\right),
		\end{equation*}
		
		where the threshold is taken to be  $0$   which allows   to  penalize  any  additional number of  people affected  by  restrictive measures. \\
		
		\item \textbf{Detection costs:}
		Decision-makers must  design   detection   policies  for   individuals who are infected and ignore their status, and eventually for those who have recovered from the disease but also ignore their status. In this regard, they must also   plan a testing strategy  that  will provide more information about  the current state  of  the epidemic  and  reduce some uncertainties  about  the   current  and future course  of  the   epidemic. Those who  can be  tested  at a given time $n$ are susceptible $S$, non-detected infectious $I^-$  and non-detected recovered $R^-$.	Then,  the number of  individuals  who can be  tested is    $X^{\text{Test}}=I^{-}+R^{-}+S$.
		Hence, the running  cost for the  testing control variable  $u^T$  is  as follows :
		
		\begin{equation*}
			C_T \left(u^T_n X^{\text{Test}}_n,\overline{ \text{x}}^{\text{Test}}\right),
		\end{equation*}
		
		where $\overline{ \text{x}}^{\text{Test}}>0$  represents  the maximum capacity for testing that can be reached. Note that in a model with constant population size,  as in our case, it holds   that  $X^{\text{Test}}= N-(I^{+}+R^{+}+ H)$ which makes $X^{\text{Test}}$ observable.
		
		\item  \textbf{Vaccination costs:}			
		Large-scale vaccination programs require significant human and financial resources. We assume that this vaccination cost depends on the number of vaccinated persons per unit of time, which is a portion of $X^{\text{Test}}$.
		Therefore, the running  cost corresponding to the vaccination strategy $u^V$ is given as 
		
		\begin{equation*}
			C_V  \left(u^V_n X^{\text{Test}}_n,\overline{ \text{x}}^{\text{Vacc}} \right).
		\end{equation*}
		
		Similarly  to the detection   cost,  we assume  a vaccination threshold denoted by   $\overline{ \text{x}}^{\text{Vacc}}>0$  that represents the maximum capacity for vaccination during a time period of length $\Delta t$ that could not be exceeded. 
		
		\item \textbf{Hospitalization costs:} Due to the epidemic, the hospital is subject to more patients  and  some equipment are used more intensively. This implies additional cost that  we can model as 
		
		\begin{equation*}
			C_H\left(H_n,\overline{ \text{x}}^{H}\right),		  
		\end{equation*}
		
		where  $\overline{ \text{x}}^{H}>0$  is the hospital capacity that can not be exceeded. Note that the fixed cost associated to the hospital cost that is related to ICU equipment, applies as long as there are infected people within the population.
		
		\item \textbf{Infection Costs or Infection Penalties:}			
		It is reasonable to assume that  during an epidemic the   government expenses and the social burden of dealing with the disease are in some sense positively correlated with the number of people infected. For instance,	the treatment and medical  care  of patients in hospitals and in specialized hospital centers like  ICU requires  considerable resources.  Further, more infections increase the risk of social tensions.		The following running costs model infection costs due to non-detected   and detected infected respectively. \\
		\begin{equation*}
			C^-_I \left(I^-_n,\overline{ \text{x}}^{I^-}\right),~~C^+_I\left(I^+_n,\overline{ 	\text{x}}^{I^+}\right),		  
		\end{equation*}
		
		where $\overline{\text{x}}^{I^-}>0$   and $\overline{\text{x}}^{I^+}>0$  are respectively the maximum number of non-detected and detected  infected   that could be tolerated and effectively  handled  with existing equipment.		 It  is  clear  that these    capacities  are  limited and  can not  be  significantly extended  in  a  short  period  of time.  We can imagine that the social  cost of the  non-detected infected  is much higher  than the  cost of the detected infected because non-detected infectious are likely to cause more harm to the society since detected infected are in quarantine and do no longer spread the disease. Note that the fixed cost associated with this infection cost is zero  because this cost is not directly related to a control variable that is activated. 				
		It should be highlighted that these infection costs  could  play a significant role in flattening  the curve of the epidemic, as  high numbers of infected  individuals are penalized by these costs.		
	\end{enumerate}

	We are now in a  position to write down  the overall running cost functional of  our  control problem. This will  be denoted for the full information framework by $\Psi^F$ and obtained by summing up previously mentioned costs.  Therefore, in a period of time of length $\Delta t$  the one-period cost writes:
	
	\begin{eqnarray*}
		\Psi^F(X_n,u_n) &= \left[ C_L(u^L_n X_n^{ \text{Work}},0) + 	C_T(u^T_n X^{\text{Test}}_n,\overline{ \text{x}}^{\text{Test}})+ C_V(u^V_n X^{\text{Test}}_n,\overline{ \text{x}}^{\text{Vacc}}) +C_H\left(H_n,\overline{ \text{x}}^{H}\right) \right. \\
		&\hspace{-6cm}+  \left. C^-_I(I^-_n,\overline{ \text{x}}^{I^-})+	C^+_I(I^+_n,\overline{ 	\text{x}}^{I^+})\right] \Delta t,
	\end{eqnarray*}
	
	\noindent where the control is $u_n= (u^L_n,u^T_n, u^V_n)^\top$.

	\subsubsection{Terminal costs}
	
	At the horizon time $T=N_t \Delta t$,  policymakers do no longer have to bear the costs associated with social distancing, testing and vaccination. Nevertheless,  they must continue to care  for  the remaining infected people in the population,  since the end of the control measure does not necessarily mean that the disease is not present in the population anymore. In this regard, we consider  the terminal cost to be  a scaled value of the infection costs and hospital costs at the  terminal time which  is denoted by $\Phi^F$ and given by 
	
	\begin{equation*}
		\Phi^F(X_{N_t}) =  C^{I^-}_{N_t}(I^-_{N_t},\overline{ \text{x}}^{I^-}_{N_t})+	C^{I^+}_{N_t}(I^+_{N_t},\overline{ 	\text{x}}^{I^+}_{N_t})+ C^H_{N_t}\left(H_{N_t},\overline{ \text{x}}^{H}\right),
	\end{equation*}
	
	\noindent 	where  $\overline{ 	\text{x}}^{I^-}_{N_t}>0$  and $\overline{ 	\text{x}}^{I^+}_{N_t}>0$  are numbers of non-detected  and detected infected  that can be tolerated at terminal time  respectively.

	\subsubsection{Performance criterion}
	
	Given  an initial  state  $X_0=x$  and   an admissible  control process $u$, the expected aggregated cost to be minimized  is defined  by 
	
	\begin{align*} \mathcal{J}^F(x,u)&= \E\Big[\sum_{n=0}^{N_t-1}  \Psi^F(X_n ,u_n)+ \Phi^F(X_{N_t}) \Big| X_0 = x\Big].
	\end{align*}
	
	The optimization  problem is to find the minimizer of  this expected aggregated  cost in  the  set of controls $u$  which satisfies some conditions that  will be specified later. A crucial observation  we can make about  this performance criterion  is that it depends not only  on the observable component $Z$ of the state process  $X=(Y^\top,Z^\top)^\top$, but also  on the hidden component $Y$. This is a non-standard objective function,  since decisions or  controls need to be  based only  on  observable quantities. Consequently, this optimal control problem is  an optimal control problem with partially observed  state process. As a result, not only  the state process is not adapted to the filtration generated by the observation process, but also the objective function depends on the non-adapted process $Y$. This does not fit to the setting of a  standard control problem. To overcome this problem, we will take advantage of the filtering  machinery  that  allows us   to transform this optimal control problem into  an equivalent  control problem with adapted and   fully observed state process.

	\subsection{Optimal control problem with partial   information }
	We are now in a position to define the associated optimal control  problem  with partial  information, i.e.,  a control problem for which the state process as well as the performance criterion are adapted with respect to  the observation filtration $\mathbb{F}=(\mathcal{F}^Z_n)_{n\geq 0}$.  In view of the introduction of  the new controlled state process, the performance criterion is expressed  in terms of the conditional mean $\condmean$ and the  conditional  variance $\condvar$ in the following proposition. 	
	\begin{proposition}\label{costprop}
		Based on the observable  filtration $\mathbb{F}$,  the new performance criterion\\ $	\mathcal{J}(x_0,u)=\D  \E_{m_0, q_0 ,z}[\mathcal{J}^F(X_0,u)] $ with $Z_0=z, x_0=(m_0, q_0 ,z)$, is the conditional  expectation of $\mathcal{J}^F$ given the initial information   and  can be  written as follows,
		\begin{equation}
			\mathcal{J}_0(x_0,u)=\D  \E_{m_0, q_0 ,z}\Big[\sum_{n=0}^{N_t-1} \Psi(\condmean_n,\condvar_n, Z_n,u_n)      + \Phi(\condmean_{N_t},\condvar_{N_t}, Z_{N_t})  	\Big],\\
		\end{equation}
		
		\noindent	where $\Psi$  and $\Phi$  are two Borel-measurable functions,  defined for $x=(m^\top, q^\top ,z^\top)^\top $,\\$m=(m^1,m^2)^\top$, $q=(q^1,q^2,q^{12})^\top$, $z=(z^1,z^2,z^3)^\top$  and $u_n=\nu=(\nu^L,\nu^T,\nu^V)^\top$ by 
		\begin{align*}
			\Psi(m, q ,z,\nu ) =& \Big [ C_L(\nu^L x^{ \text{Work}},0) + 	C_T(\nu^T x^{\text{Test}},\overline{ \text{x}}^{\text{Test}})+ C_V(\nu^V x^{\text{Test}},\overline{ \text{x}}^{\text{Vacc}}) \\& +C_H\left(z^3,\overline{ \text{x}}^{H}\right)  +	C^+_I(z^1,\overline{ 	\text{x}}^{I^+}) \\
			&+  	a_{I^-} m^1+ b_{I^-}  \int_{\overline{ \text{x}}^{I^-}}^{\infty} (y-\overline{ \text{x}}^{I^-})^2  \dfrac{1}{\sqrt{2 \pi q^1}} \exp \Big(- \frac{1}{2} \dfrac{(y-m^1)^2}{q^1}\Big) dy \Big ] \Delta t,\\	
			\Phi(m, q ,z)  = & a_{I^-,N_t} m^1+ b_{I^-,N_t}  \int_{\overline{ \text{x}}^{I^-}}^{\infty} (y-\overline{ \text{x}}^{I^-})^2  \dfrac{1}{\sqrt{2 \pi q^1}} \exp \Big(- \frac{1}{2} \dfrac{(y-m^1)^2}{q^1}\Big) dy\\
			& +	C^{I^+}_{N_t}(z^1,\overline{ 	\text{x}}^{I^+}_{N_t})+ C^H_{N_t}\left(z^3,\overline{ \text{x}}^{H}\right).
		\end{align*}
		
		\noindent with  $x^{ \text{Work}}=N-z^1$, $x^{\text{Test}}=N-z^1-z^2$.
	\end{proposition}
	
	The proof of this proposition can be found in Appendix \ref{proof}. Note that with some abuse of notation, the symmetric matrix $\condvar$ is considered as a vector of its entries, $\condvar=(\condvar^1,\condvar^2,\condvar^{12})^\top$, similarly, $q=(q^1,q^2,q^{12})^\top$ . A simpler form of the integral appearing in the Proposition \ref{costprop}  can be obtained in terms of the cumulative distribution function of the standard normal distribution. This form is given in the following Lemma \ref{lem_integral} proved in Appendix \ref{prooflem} and is quite helpful for numerical purposes. 
	
	\begin{lemma}\label{lem_integral}
		The integral  $ \displaystyle \int_{\overline{ \text{x}}^{I^-}}^{\infty} (y-\overline{ \text{x}}^{I^-})^2  \dfrac{1}{\sqrt{2 \pi q^1}} \exp \Big(- \frac{1}{2} \dfrac{(y-m^1)^2}{q^1}\Big) dy $  can be expressed  as 
		
		$\big((m^1-\overline{ \text{x}}^{I^-})^2 +q^1 \big) \Phi_{\mathcal{N}}\Big( \dfrac{m^1-\overline{ \text{x}}^{I^-}}{\sqrt{q^1}}\Big ) 
		+ (m^1-\overline{ \text{x}}^{I^-}) \sqrt{\dfrac{q^1}{2 \pi}}
		\exp \Big(- \frac{1}{2} \dfrac{(m^1-\overline{ \text{x}}^{I^-})^2}{q^1}\Big)  ,$
		
		where  $\Phi_{\mathcal{N}}$  is the cumulative  distribution function of  the standard normal distribution. 
		
	\end{lemma}

	We have therefore transformed  the control problem with partially observed state process 	$X=(Y^\top,Z^\top)^\top$ into a corresponding  control problem with fully  observed state process 	$X^P=(\condmean^\top,\condvar^\top,Z^\top)^\top$. This is done by replacing  the hidden  state $Y$  by the   filter $(\condmean^\top,\condvar^\top)^\top$  and rewriting the dynamics  of $\condmean$ and $Z$ in terms of the innovation process $\Ec$. 
	The new controlled state process is 
	$X^P=(\condmean^\top,\condvar^\top,Z^\top)^\top$
	that has  dynamics generated by the following transition operators,
	\begin{equation} \label{stateXP}
		\begin{array}{rll}
			\condmean_{n+1}&=  \mathcal{T}^M (n,M_n,Q_n,Z_n,u_n,\Ec_{n+1}), \\[0.5ex]			
			\condvar_{n+1} &=  \mathcal{T}^Q (n,M_n,Q_n,Z_n,u_n,\Ec_{n+1}) ,	\\[0.5ex]		
			Z_{n+1}&= \mathcal{T}^Z (n,M_n,Q_n,Z_n,u_n,\Ec_{n+1}),
		\end{array}
	\end{equation}
	
	\noindent with the transitions operators having following  structures:   
	
	\begin{equation} \label{stateXP2}
		\begin{array}{rll}			
			\mathcal{T}^M (n,y,q,z,\nu,\varepsilon) &=f_M(n,y,q,z,\nu) + ~g_M(n,y,q,z,\nu)\,\varepsilon\\[0.5ex]	
			\mathcal{T}^Q (n,y,q,z,\nu,\varepsilon)&= f_Q(n,y,q,z,\nu) \\[0.5ex]		
			\mathcal{T}^Z (n,y,q,z,\nu,\varepsilon)&=
			f_Z(n,y,q,z,\nu) +~g_Z(n,y,q,z,\nu)\,\varepsilon
		\end{array}
	\end{equation}
	and where $f_M,~f_Q,~f_Z, ~g_M,~g_Z$ are continuous and measurable functions, with the precise expressions  that  can be derived from equations \eqref{eqvar},\ref{eqmean}, \ref{eqobs}. More details are provided in Appendix \ref{transfunct}. 
	This new controlled state process is adapted to  the observable filtration 	$\mathbb{F}= (\mathcal{F}^Z_n)_{n\geq 0}$. It is also important to note firstly that the processes $\condmean$  and $Z$  are driven by  the same noise. Secondly,   the dynamic of the conditional variance $\condvar$ does not include  any  noise term. These two facts lead to a control problem with a degenerated controlled stochastic state process  $X^P$.  
	
	Our optimal control problem  can now be treated  as a  MDP   with state process  $X^P=(\condmean^\top,\condvar^\top,Z^\top)^\top$  taking values in a state space $\mathcal{X}$ which is a subset of $\mathbb{R}^d$, where $d$  is the number of relevant 
	variables in $X^P$. For the epidemic model considered, we have $d=8$ since the conditional mean $\condmean$ is two-dimensional, the conditional covariance matrix has three relevant entries because it is symmetric and finally the observation $Z$  is three-dimensional. The dynamics of the state process $X^P$ which is adapted to the observable filtration $\mathbb{F}^Z$, is described by a transition operator  $\mathcal{T}=\left( (\mathcal{T}^M)^\top ,(\mathcal{T}^Q)^\top,(\mathcal{T}^Z)^\top \right)^\top$ with a Gaussian transition kernel, and  equation \eqref{stateXP} is written as follows
	$X_{n+1}^P = \mathcal{T}(n,X_n^P,u_n,\Ec_{n+1})$.
	
	\smallskip
	
	\noindent\textbf{Admissible controls.} We denote by  $\mathcal{A}$ the class of admissible controls,
	consisting of  control processes $u$ of the   Markov type $u_n=\widetilde u(n,X^P_n)$, ($\widetilde u$ being a measurable function)   and taking values  in the set of feasible control values $\mathcal{U}=  [0,1]\times \mathbb{R}_+^2$. 
	
	\begin{tabular}[t]{lcr}
		$\mathcal{A}$& $=$ &$\big\{\!\!			
		(u_n)_{n=0,\ldots,N_t-1} ~|~  \text{ Markov control } u_n=\widetilde u(n,X^P_n),$ \\[0.5ex]
		&&$\text{control constraints } u_n\in \mathcal{U}=  [0,1]\times \mathbb{R}_+^2
		\big\}$.			
	\end{tabular}

	\smallskip
	
	\noindent \textbf{Performance criterion.} Given an admissible  strategy $u=(u_0,\ldots,u_{N_t-1})$, its performance will be evaluated at time $n$,	 for $n=0,\ldots, N_t,~~ x=(m,q,z)$ by the expected aggregated cost 		of the problem starting  at time $n$ at the state $X^P_n=x$
	\begin{align}
		J(n,x,u)=	\E\bigg[\sum_{k=n}^{N_t-1} \Psi(X^P_k ,u_k)+ \Phi(X^P_{N_t}) \Big| {X^P_n=x}\bigg]. \label{objfunctjn}
	\end{align}
	
	\smallskip
	
	\noindent \textbf{Optimization problem.}		The objective will  be to  look for the strategy $u^*\in \mathcal{A} $ that minimizes the performance criterion.  Hence, we have that 
	
	\begin{align}
		J(n,x,u^*) =&V(n,x) :=\inf\limits_{u\in \mathcal{A}}  J(n,x,u).  \label{optpbn}
	\end{align} 
	
	In the case that a strategy $u^*$ does exist, it   is called the optimal strategy   and the associated cost functional defines the value function $V$.
	
	\smallskip
	
	\noindent	\textbf{	Dynamic programming approach.} To solve this control problem with   complete information, we use the dynamic programming approach, which is based on  the Bellman principle presented in \citet{bauerle2011markov}.  
	This approach consists of embedding the initial optimization problem with objective function $\mathcal{J}_0$  into a family of optimization problems with the objective function $J$ in equation \eqref{objfunctjn} and initial time
	$n = 0,\cdots,N_t-1$ and initial state $X^P_n = x \in \mathcal{X}$. 
	This approach 
	leads to the following necessary optimality condition called Bellman equation or dynamic
	programming equation (DPE).
	
	\begin{theorem}			
		{(Bellman Equation / Dynamic Programming Equation)}\\
		The value function satisfies 
		\begin{align}		\label{DPE}
			V(n,x)&=\inf_{\nu \in  \mathcal{U}}\Big\{\Psi(x,\nu)+ 
			\E_{n,x} \big[V(n+1,\mathcal{T}(n,x,\nu,\Ec_{n+1})\big]\Big\}, ~n=0,1,\ldots,N_t-1 \\ \label{BE} 
			V(N_t,x)&=\Phi(x) ,     \qquad \text{(terminal condition)}.
		\end{align}
		Here, the transition operator $\mathcal{T}$, with Gaussian transition kernel generated by  $\Ec_{n+1}$  is given by \eqref{stateXP}.
	\end{theorem}
	For all $n=0,1,\ldots,N_t-1$, the corresponding  candidate for the  optimal strategy is $u^*_n=\widetilde{u}^*(n,X^{P,u^*}_n)$
	with the
	optimal decision rule given by
	
	$$\widetilde{u}^*(n,x)=\argmin_{\nu \in  \mathcal{U}}\Big\{\Psi(x,\nu)+ 
	\E_{n,x} \big[V(n+1,\mathcal{T}(n,x,\nu,\Ec_{n+1}))\big]\Big\}.$$

	The dynamic programming equation \eqref{DPE} can be solved using the following backward recursion algorithm \eqref{algBR} which starts at the terminal time $N_t$  and recursively computes at each time step the value function and the associated  optimal decision rule. Note  that for known $x$  and $\nu$, the conditional expectation $\E_{n,x}$ becomes an unconditional expectation with respect to the random vector  $\Ec_{n+1}$.

	\begin{algorithm}
		\caption{Backward recursion algorithm}\label{algBR}
			\vspace{.3cm}	
			\KwResult{ Find the value function $V $ and the optimal decision rule  $\widetilde{u}^*$}
			\vspace{.3cm}
			\KwSty{	\textbf{Step 1:}} Compute for all $x \in \mathcal{X}$,
			$$V(N_t,x) = \Phi(x)$$
			
			\vspace{.2cm}
			\KwSty{	\textbf{Step 2:} } 
			
			\Indp	\For {$n=N_t- 1,\ldots,1,0 $,}{  compute for all $x \in \mathcal{X}$ \;
				
				$$V(n,x)=\inf_{\nu \in  \mathcal{U}}\Big\{\Psi(x,\nu)+ 
				\E \big[V(n+1,\mathcal{T}(n,x,\nu,\Ec))\big]\Big\}$$\;
				
				Compute a minimizer $u^*_n$  of  the above point-wise  optimization problem, given by\;
				
				$$\widetilde{u}^*(n,x)=\argmin_{\nu \in  \mathcal{U}}\Big\{\Psi(x,\nu)+ 
				\E\big[V(n+1,\mathcal{T}(n,x,\nu,\Ec))\big]\Big\}$$ 
			}
		\end{algorithm}
		
		The implementation of the backward recursion algorithm faces certain challenges. The computation of the conditional expectation in the Bellman equation at each time step $n$ for all state $x \in \mathcal{X}$ is required. Unfortunately, there is no closed-form expression for this conditional expectation and its approximation becomes computationally intractable if the dimension of the state space is high,  the so-called curse of dimensionality. Moreover, the dimension of the control variable has a huge impact on the computational  time of the infimum in the Bellman equation.  For our optimization problem under partial information, the state $X^P$ is a 
		$8$-dimensional vector and control  variable $u$ is a $3$-dimensional vector.
		To address these challenges, in the next section, we discretize the state space and the set of feasible values of  control to form an approximate MDP for a controlled finite-state Markov chain. Subsequently, we will
		approximate  the conditional expectation\\ $\E\big[V(n+1,\mathcal{T}(n,x,\nu,\Ec))\big]$ using the quantization techniques.
		
		\section{Numerical methods}\label{sect3}

		\subsection{Backward recursion  with optimal quantization}
		
		In order to implement the backward recursion algorithm for solving our optimal control problem which lacks a known closed-form solution, we  discretize the state space  and the control space. Furthermore, to mitigate the curse of dimensionality, the  conditional expectation appearing in the Bellman equation will be approximated using quantization techniques for  the Gaussian transition  kernel. Specifically, the Gaussian random vector is replaced by  an appropriated  discrete and finite random random vector. We select  the quantizer such that the approximation error is small and     computational tractability is maintained. 
		
		%
		%
		
		\subsubsection{Discretization of state space and feasible control set   }
		
		The controlled state process contains as  components  the observation process $(Z_n)$, the conditional mean process $(\condmean_n)$  and the conditional covariance process $(\condvar_n)$. For the observation $Z=\left( I^+, R^+,H \right)^\top$, each component takes formally  values in  $[0,N]$, if we neglect the discretization errors. Unlike the compartment $R^+$,  which can actually take values from $0$ to $N$, the compartment $I^+$ and $H$, in application do not   reach the value $N$. Therefore, it is  reasonable to assume maximum values, denoted as  $I^+_{max}, ~H_{max} < N$   for the number of non-detected infected  individuals  and hospitalized individuals, respectively. Analogously, we will consider that the conditional mean components $\condmean^1$ and $\condmean^2$ take values in $[0,\condmean^1_{max}]$  and $[0,\condmean^2_{max}]$ for some suitable $\condmean^1_{max},\condmean^2_{max} < N$ , respectively. For the covariance matrix $\condvar$, to ensure that all possible matrices obtained are positive semi-definite, we discretize  the diagonal entries $\condvar^{1}, \condvar^{2}$ and the correlation  coefficient $\rho=\dfrac{\condvar^{12}}{\sqrt{\condvar^{1} \condvar^{2}}}$  instead of the non-diagonal entry $\condvar^{12}$.  Hence, $\rho$ takes values in $[-1,1]$ and we assume $\condvar^{1} \in [0, \condvar^{1}_{max}]$ and $\condvar^{2} \in [0, \condvar^{2}_{max}]$ for some suitable $\condvar^1_{max},\condvar^2_{max} > 0$, which we calibrated to simulated  path of the filter processes.

		Let $\widetilde{N}_1$ and $\widetilde{N}_2$ be the number of grid points in the $ \condmean^1$ and $\condmean^2$ direction  respectively, let  $0=\widetilde{\condmean}^1_1<  \ldots< \widetilde{\condmean}^1_{\widetilde{N}_1}= \condmean^1_{max}$,  $0=\widetilde{\condmean}^2_1<  \ldots< \widetilde{\condmean}^2_{\widetilde{N}_2}= \condmean^2_{max}$ be finitely many grid points in the  $ \condmean^1$ and $\condmean^2$ direction  respectively, and  let $\widetilde{N}_3,~ \widetilde{N}_4~ \text{and}~ \widetilde{N}_5$ be the number of discretization grid points in the direction $\condvar^{1},~\condvar^{2}, ~ \text{and }  ~ \rho$ respectively, in addition, let  $0=\widetilde{\condvar}^{1}_1<  \ldots< \widetilde{\condvar}^{1}_{\widetilde{N}_3}= \condvar^{1}_{max}$,  $0=\widetilde{\condvar}^{2}_1<  \ldots< \widetilde{\condvar}^{2}_{\widetilde{N}_4}= \condvar^{2}_{max}$ and $ \widetilde{\rho}_1<\ldots< \widetilde{\rho}_{\widetilde{N}_5}$  be finitely many grid points in the direction
		$\condvar^{1},~\condvar^{2}, ~ \text{and }  ~ \rho$ respectively. Finally,
		let $\widetilde{N}_6, \widetilde{N}_7,~ \text{and} ~ \widetilde{N}_8$  be the number of discretization grid points in the direction $Z^1,Z^2 ~ \text{and} ~ Z^3$ respectively, let $0= \widetilde{Z}^1_1<  \ldots< \widetilde{Z}^1_{\widetilde{N}_6}= \widetilde{Z}^1_{max}$,  $0=\widetilde{Z}^2_1< \ldots< \widetilde{Z}^2_{\widetilde{N}_7}= Z^2_{max}$ and $0= \widetilde{Z}^3_1<  \ldots< \widetilde{Z}^3_{\widetilde{N}_8}= Z^3_{max}$  be finitely many grid points in the direction $Z^1,Z^2 ~ \text{and} ~ Z^3$ respectively. 		
		\noindent Then, the discretized state  space is given by $  \mathcal{X}^D= \mathcal{\condmean}^1\times \mathcal{\condmean}^2 \times \mathcal{\condvar}^1 \times \mathcal{\condvar}^2 \times \mathcal{R} \times \mathcal{Z}^1 \times \mathcal{Z}^2 \times \mathcal{Z}^3  $, where 
		
		\smallskip
		
		\begin{tabular}{lclllcll}				
			$\mathcal{\condmean}^j$& $=$&$\left\{  \widetilde\condmean^j_1, \ldots, \widetilde\condmean^j_{\widetilde{N}_j}\right\}$,& j=1,2,&
			$\mathcal{\condvar}^j$& $=$&$\left\{  \widetilde\condvar^j_1, \ldots, \widetilde\condvar^j_{\widetilde{N}_j}\right\}$,& j=1,2,\\
			[1ex]			
			$\mathcal{R}$& $=$&$\left\{  \widetilde\rho_1, \ldots, \widetilde\rho_{\widetilde{N}_5}\right\}$,& &
			$\mathcal{Z}^j$& $=$&$\left\{  \widetilde Z^j_1, \ldots, \widetilde Z^j_{\widetilde{N}_j}\right\}$, & j=1,2,3.\\
		\end{tabular}
		
		\smallskip
		
		Let us introduce the  set of  multi-indices $\mathcal{N}= \prod_{k=1}^{8} \mathcal{N}_k$  where $\mathcal{N}_k =\{1,2,\ldots,\widetilde{N}_k \},$ for  $k=1,\cdots,8$. For a multi-index $\widetilde m=(\widetilde m_1,\widetilde m_2,\ldots,\widetilde m_8) \in \mathcal{N}$, we denote by $x_{\widetilde m}$ the grid point $$(\widetilde \condmean^1_{ \widetilde m_1}, \widetilde \condmean^2_{ \widetilde m_2}, \widetilde \condvar^1_{\widetilde m_3}, \widetilde \condvar^2_{\widetilde m_4}, \widetilde \rho_{\widetilde m_5},\widetilde Z^1_{\widetilde m_6},\widetilde Z^2_{\widetilde m_7}, \widetilde Z^3_{\widetilde m_8}).$$
		
		The set of feasible control values $\mathcal{U}=[0,1]\times \mathbb{R}^2_+$ is also discretized into the discrete set $$\mathcal{U}^D= \mathcal{U}^L \times \mathcal{U}^{T} \times \mathcal{U}^{V},$$
		with $\mathcal{U}^L=\{ \widetilde u^L_1,\ldots,\widetilde u^L_{\widetilde{N}^L} \} \subset [0,1]$, $\mathcal{U}^{T}=\{ \widetilde u^{T}_1,\ldots,\widetilde u^{T}_{\widetilde{N}^{T}} \} \subset \mathbb{R}_+$   and $\mathcal{U}^{V}=\{ \widetilde u^{V}_1,\ldots,\widetilde u^{V}_{\widetilde{N}^{V}} \} \subset  \mathbb{R}_+$.
		
		This state space and feasible control set  discretization allows to approximate the solution of  the given MDP with state space $\mathcal{X}$ by evaluating the value function  and the optimal policy  only on the grid points of the  discrete  state space $ \mathcal{X}^D$, and interpolate between grid points. The resulting MDP becomes more tractable, especially enabling us to explore discrete set of control strategies.
		
		\subsubsection{Optimal quantization}
		Quantization techniques have been applied to reduce the computational burden in several domains among which we could mention signal transmission, clustering with the "k-means" algorithm, numerical integration to  compute the expectation  and many more. Quantization  plays a crucial role due to his practical advantages and computational efficiency.
		Vector quantization is an approach to approximate the distribution of a  continuous random vector $\Ec$ with values in $\mathbb{R}^d$ by a discrete random vector $\Eq$ taking preferably finitely many  values in the  set  $\Gamma=\big\{\xi_1,\ldots,\xi_{\L}\big\}, ~~\xi_l  \in  \mathbb{R}^d, i=1,\cdots,\L $   with associated weights 
		$\left\{w_1,\cdots,w_{\L}\right\}   $.   The possible values $\big\{\xi_1,\ldots,\xi_{\L}\big\}$ of the discrete random vector   are called quantizers or quantization grids.    To have a good approximation, the discrete random vector must be constructed appropriately. The approximation $\Eq$ should therefore satisfy some specified properties. For instance, if we require that all the weights $w_i$  are equal, we obtain the equiprobable quantization. If we construct a  quantization of the vector $\Ec \in \mathbb{R}^d $  by taking the Cartesian product of the quantizers of its $d$ components, we would get the so-called product quantization. For more details on quantization based on  Kantorovich, Wassertein, and Kolmogorov distance between the distribution of $\Ec$ and $\Eq$, we refer to \citet{pflug2011approximations}.  The most  classical criterion in the literature, on which    we focus,  is to assess the goodness of an approximation $\Eq$ through the $L^p$-quantization error. Hence, if the random vector $\Ec$ possesses a moment of order $p\geq 1$, i.e. $\Ec  \in L^p$ then the $L^p$-quantization error of $\Eq$  taking values in $\Gamma$   is defined as 
		
		$$ \Err_{\L,p}(\Ec,\Eq)=\left(\mathbb{E} [\delta(\Ec,\Gamma)^p]\right)^{1/p},$$
		
		\noindent  where the distance between a vector $\xi$ and a set $A$ is given by $\delta (\xi, A)=\min_{a \in A} \vert  \xi -a\vert$ where  $\vert \cdot \vert$   is the Euclidean norm. 			
		
		Often, to a quantizer set $\Gamma=\big\{\xi_1,\cdots,\xi_{\L}\big\} $, one can assign  a partition\\ $\left\{\left(C_i\right)_{1 \leq i\leq \L}\right\}$ of $\mathbb{R}^d$  with 
		
		$$ C_i = \left\{ y \in \mathbb{R}^d; \vert y-\xi_i\vert\leq   \min_{1 \leq l \leq \L} \vert y-\xi_l \vert  \right\} .$$
		
		This partition is  the Voronoi partition  and allows  writing the approximation $\Eq$  as the closest neighbor projection of $\Ec$ into the set $\Gamma$. Weights $w_i$  can then be viewed as the measure of $C_i$ with respect to the distribution of $\Ec$. Hence, if $\mu_{\Ec} $ denotes the distribution of the random vector $\Ec$, then the $L^p$- quantization error can be expressed as

		\[\Err_{\L,p} (\Ec, \Eq)~=~\left(\int_{\mathbb{R}^d} \min_{1 \leq l \leq \L} \vert y-x_l \vert^p \mu_{\Ec}(dy) \right)^{1/p}
		~=~\left( \sum_{l=1}^{\L}\int_{C_l}  \vert y-x_l \vert^p \mu_{\Ec}(dy) \right)^{1/p}.			 
		\]
		
		An $L^p$-optimal  quantization   $\Eq^*$ of level $\L$  is a quantization taking values in  $\Gamma^*=\big\{\xi^*_1,\cdots,\xi^*_{\L}\big\}   $   with associated probabilities 
		$\big\{w^*_1,\cdots,w^*_{\L}\big\}   $  such that 
		
		\begin{equation*}
			\Err_{\L,p} (\Ec, \Eq^*)= \inf_{ \Eq \in \mathcal{L}_{\L}} \Err_{\L,p} (\Ec, \Eq),	 
		\end{equation*}
		
		\noindent where $\mathcal{L}_{\L}$ is the set of discrete random vectors taking at most $\L$ possible values.
		
		The quantization method can be linked in some sense  to the Monte-Carlo approximation which also provides a finite sample realization of the random  variable with equal weight for each realization. However, unlike the Monte Carlo method which is a stochastic method,  because it gives a different  set of possible values each time, the quantization method is a deterministic method. Thanks to this deterministic aspect, the  quantization method allows to achieve a  prescribed level of accuracy by increasing the level $\L$ as we will  see later with  Zador's theorem.   Another advantage of the quantization approach is that the optimal quantization will always provide a higher accuracy compared to the Monte Carlo method sample   with the same level $\L$. However, the main drawback of this method is the computational effort to find an optimal quantizer. Nevertheless, this can be done offline such that the quantizers and the weights are stored before the actual computation.

		Many theoretical questions emerge regarding optimal quantizer, for instance does the optimal quantizer  always exist? If it exists,  does  the optimal quantizer provide a good   approximation of the probability distribution, and can we get an  error estimate? Also,  how  do we compute the optimal quantizer numerically? 
		
		The existence of the optimal quantization was proven by \citet{pages1998space},  \citet[Theorem 4.12]{graf2007foundations}  for $ \mathbb{R}^d$, and in
		\citet{graf2007optimal} for any Banach space.  The quality of the optimal quantization approximation have been addressed in several results  in the literature. Let us  start  by the strictly decreasing property of the optimal quantizer as  the level $\L$ tends to infinity which is part of  \citet[Theorem 4.12]{graf2007foundations}.
		
		\begin{theorem}(The mapping $\L \mapsto \Err_{\L,p} (\Ec, \Eq_{\L}^*) $ is strictly decreasing)
			For every random variable $\Ec$ with distribution $\mu_{\Ec}  \in L^p(\mathbb{R}^d)$ such that $\card(\supp(\mu_{\Ec})) \geq \L $, one has 
			
			$$ \Err_{\L,p} (\Ec, \Eq_{\L}^*) < \Err_{\L-1,p} (\Ec, \Eq_{\L-1}^*),  ~~~ \text{for} ~ \L \geq 2.$$
		\end{theorem} 
		
		This result shows that  as  the quantization level $\L$ increases, the optimal quantization error gets closer to $0$.  In addition, the following Zador's theorems    provide an   upper bound for the optimal quantization error.
		
		\begin{theorem}(Non-asymptotic Zador's theorem)
			Let $\eta >0$, then  for every random vector $\Ec$ with distribution $\mu_{\Ec} \in L^{p+\eta}\left(\mathbb{R}^d\right)$ and for every quantization level $\L$, there exists a constant $C_{d,p,\eta}>0$ which depends only on $d,p ~\text{and}~\eta$  such that 
			\begin{equation*}
				\Err_{\L,p} (\Ec, \Eq_{\L}^*)  \leq  C_{d,p,\eta} \sigma_{p+\eta}(\mu_{\Ec}) \L^{-1/d},
			\end{equation*}

			\noindent where for $ r >0,~ \sigma_r(\mu_{\Ec})= \min_{a \in \mathbb{R}^d} \left[\int_{\mathbb{R}^d}\vert y-a \vert^r \mu_{\Ec} (dy)  \right]^{1/r}$.
			
		\end{theorem} 
		
		This non-asymptotic Zador's theorem which was proven by \citet[Theorem 5.2]{pages2018numerical} shows that the optimal quantization error is of order $\L^{-1/d}$. Moreover, as the  level $\L$ approaches  infinity,  the second Zador's theorem provides an upper bound for the quantization error (see \citet{graf2007foundations}).
		
		\begin{theorem}(Zador's theorem)
			Let  $\Ec$ be  random vector with distribution $\mu_{\Ec} $ which has the following Lebesgue decomposition with respect to the Lebesgue measure $\lambda_d$,  $\mu_{\Ec}=\mu_a+\mu_s=h \lambda_d+ \mu_s,$ where $\mu_a$ is the absolutely continuous part with density $h$,  and $\mu_s$ is the singular part of $\mu_{\Ec}$. Let $\eta >0$ such that 
			$\mu_{\Ec} \in L^{p+\eta}\left(\mathbb{R}^d\right)$. Then,  there exists a constant $C_{d,p} >0$ which only  depends  on $d ~\text{and}~p$  such that 
			\begin{equation*}
				\lim_{\L \to +\infty}  \L^{1/d}	\Err_{\L,p} (\Ec, \Eq_{\L}^*)  \leq  C_{d,p} \left[\int_{\mathbb{R}^d}h^{\frac{d}{d+p}} d \lambda_d  \right]^{\frac{1}{p}+\frac{1}{d}}.
			\end{equation*}
			
		\end{theorem}
		
		After these theoretical results on the convergence of the quantization error to zero, we can turn to the practical question of computing the optimal quantizer.
		This issue is addressed in the literature  for the case $p=2$, and  the $L^2$-optimal  quantizer  is called the optimal quadratic quantizer. Many algorithms based on iterative methods have been employed to  compute optimal quantizers. These algorithms exploit either the differentiability of the square of the quadratic  quantization error  and design a  zero search algorithm, or the stationarity property of the  quadratic  optimal quantizer described below.
		
		\begin{theorem}
			Let $\Ec$ be  a random vector with distribution  $ \mu_{\Ec} \in L^2(\R^d)$ such that\\ $ \card(\supp(\mu_{\Ec})) \geq \L $. If the norm on $\R^d$ is the Euclidean norm, then
			any quadratic optimal quantizer $\xi^*= (\xi^*_ 1, \ldots, \xi^*_{\L})$ of level $\L$ is stationary in the sense
			\begin{equation*}
				\E \left[\Ec \vert \Eq^*\right]= \Eq^*,
			\end{equation*}
			with this equality valid for every Voronoi partition.
		\end{theorem}
		For the proof we refer to  \citet[Proposition 5.1]{pages2018numerical}. Classical methods  for computing optimal quantizers  are the Competitive Learning Vector Quantization (CLVQ), the Lloyd I algorithm and  the randomized
		version of Lloyd’s algorithm which is more efficient than Lloyd I for dimensions greater than 3. More details on these algorithms can be found in  \citet{pages2018numerical}. For the standard multivariate normal distribution,  Pages provides a repository  for optimal quantizers, weights and error estimate (see \url{http://www.quantize.maths-fi.com/gaussian_database}). 
		
		\subsubsection{Application to the Bellman equation}
		
		We now want to apply  the optimal quantization method  to the backward recursion algorithm by  approximating  the expectation appearing in the  Bellman equation. More generally, for a random vector $\Ec$  with distribution $\mu_{\Ec} \in L^p(\R^d)$, let $\Eq^*$  be the optimal quantization of level $\L$ of $\Ec$,  for a given function $V$, if  $\Eq^*$  is a good approximation of $\Ec$, then  it is reasonable  to consider as a good approximation of $\E(V(\Ec))$  the following cubature formulas
		
		\begin{equation*}
			\E \left[V(\Eq^*)\right]= \sum_{l=1}^{\L} V(\xi^*_l) \mu_{\Ec}  (C_l).
		\end{equation*}
		
		For this cubature formulas, error  bounds have been provided by  \citet{pages2018numerical}  for certain classes of functions $V$. Especially, if $V$  is a Lipschitz continuous function  with Lipschitz constant $K_V$,   one can show that 
		
		\begin{equation*}
			\E \left[ \vert V(\Ec)- V(\Eq^*) \vert \right] \leq  K_V \Vert \Ec- \Eq^* \Vert _1 \leq K_V \Vert \Ec- \Eq^* \Vert _p, \text{ for all } p \geq 1.
		\end{equation*} 
		
		Furthermore, in the case where  $V$  is differentiable with  a Lipschitz continuous gradient $\nabla V$  with Lipschitz constant $K_{\nabla V}$, we have (see  \citet[Prop 5.2]{pages2018numerical})
		
		\begin{equation*}
			\E \left[ \vert V(\Ec)- V(\Eq^*) \vert \right]  \leq \dfrac{1}{2}K_{ \nabla V} \Vert \Ec- \Eq^* \Vert _2.
		\end{equation*}

		Solving the MDP with discrete state space $\mathcal{X}^D$  using the backward recursion algorithm consists  of computing an approximation of the value  function $V$ and the optimal decision rule $\widetilde{u}^*$  for $n=0,1,\cdots, N_t-1$ and all $x_i \in \mathcal{X}^D$.  For a fixed grid point  $x_i$  and $n\in \{   0,1,\cdots, N_t-1 \}$, given the value function $V(n+1,x)$ for the next time point, and all grid points $x \in \mathcal{X}^D$ (and hence for all $x \in \mathcal{X}$  by interpolation), the backward recursion requires the computation of  $\E_{n,x_i} \big[V(n+1,\mathcal{T} (n,x_i, \nu, \Ec_{n+1}))\big]$,		
		\noindent where $\mathcal{T}$  is the transition operator  and $\Ec_{n+1}$ a multivariate standard normally distributed random variable. The approximation of this conditional expectation using the quantization approach is obtained by replacing  the Gaussian distribution of $\Ec_{n+1}$  by an optimal quantizer $\Eq$   	with values in 
		$\Gamma_{\L}=\{ \xi^*_1, \ldots, \xi^*_{\L} \} $ 
		and associated  probabilities  $\{ w_1, \ldots, w_{\L} \} $. This yields the approximation    
		
		\begin{equation*}
			\mathbb{E} \big[V(n+1,\mathcal{T}(n,x_i,\nu,\Ec_{n+1}) )\big] \approx  \sum_{l=1}^{\L} V(n+1,\mathcal{T}(n,x_i,\nu,\xi^*_l)) w_l.
		\end{equation*}

		On the right-hand side of the approximation, the value function is evaluated by linear interpolation  because  $\mathcal{T}(n,X_n^P,\nu,\xi_l)$ is not necessarily  a grid point in $ \mathcal{X}^D$.  Unfortunately, linear interpolation becomes  time-consuming  for scattered data in high dimension  due to the need for domain  triangulation   to identify the surrounding grid points of a given interpolant. In addition, after locating surrounding grid points, computing  the weight associated to each  of these grid points further adds to the computation time. This makes   the use of a finer grid discretization numerically intractable. To address this issue, one may consider some alternative to linear interpolation for high dimensional problems like  nearest neighbor interpolation  and radial basis function interpolation.  The nearest neighbor interpolation is known to be effective for a refined grid discretization, however, working with a finer grid is impractical due to the exponential computational burden as the state and control dimensions increase.  The radial basis functions (RBFs) interpolation have shown to achieve good accuracy for some high dimensional problems, see 
		\citet{jakobsson2009rational,doi:10.1080/0305215X.2013.765000}. RBFs use kernel functions to interpolate values based on radial distances from data points.
		
		In the implementation of our optimal control problem at hand, we aim to avoid  the linear interpolation in high dimensions as it presents two significant  challenges. First, it is time-consuming   because it requires  the triangularization of the multidimensional state space grid points, polytope  search to locate inquiry points and computation of barycentric coordinates. Second, the degeneracy of some polytope or conditioning issues  may render impossible the computation of some barycentric coordinates.  In order to overcome  these challenges, we consider a parametrization of the value function  that allows us to evaluate the parametrized expression instead of performing linear interpolation.  This idea is  similar to the least-squares Monte Carlo, except that we  use a discretization grid  and  the Monte Carlo simulations are replaced by the quantization.

		\subsection{Backward recursion with  quantization and value function regression}

		Several approaches designed to overcome the curse of dimensionality when solving optimal control problems are based on function approximation and are often referred to as approximate dynamic programming. These include approaches such as value iteration, policy iteration, and least-squares Monte Carlo. Some of these methods, especially the latter, consider a parametrization of the value function as a linear combination of some appropriate ansatz functions, and the coefficients of the parametrization are determined using a least-squares approach. The value function approximation is of the form 	$V(n,\cdot)  \approx \sum_{j=1}^{\M} \theta_j(n) \varphi_j(\cdot) $,  	
		where $(\varphi_j)_{j=1,\ldots,\M}$  is a basis of known  ansatz functions,  and $(\theta_j(n))_{j=1,\ldots,\M}$  the constant coefficients depending only  on  time $n$  to be determined.

		To choose a suitable basis of  ansatz functions, it is important to have some prior  knowledge about the behavior of the value function w.r.t. the components of the state variable. Common choices  for ansatz functions include multivariate polynomial bases, abstract Fourier series and  Gaussian kernel bases. In recent years, approximating the value function using neural network became very popular  and this approach is known as  fitted value iteration.

		Let us assume that we are given  the value function $V(n,x_i)$ at some time step $n$  and for all grid points $x_i, ~i \in \mathcal{I}$. We consider a parametrization of $V(n,\cdot)$  as a linear combination of the following form, 			 
		$V(n,\cdot)  \approx \sum_{j=1}^{\M} \theta_j(n) \varphi_j(\cdot) $ 				
		with known basis ansatz  functions  $(\varphi_j)_{j=1,\ldots,\M}$   and constant coefficients $(\theta_j(n))_{j=1,\ldots,\M}$ to be determined. Thus, since we know  the value function for all grid points $x_i, ~i \in \mathcal{I}$, the optimal  parameters $\theta^*(n) = (\theta^*_j(n))_{j=1,\ldots,\M}$  solve the least squares regression problem

		\begin{eqnarray*}
			\theta^*(n)= \argmin_{\theta } \sum_{i \in \mathcal{I}}\Vert V(n,x_i)- \sum_{j=1}^{\M} \theta_j(n) \varphi_j(x_i)\Vert^2 .
		\end{eqnarray*}
		
		The main difficulty  associated with this method is the selection   of an appropriate basis ansatz function. One heuristic based on  the classical linear quadratic problem,  is to consider that the value function inherits certain  properties of the running  and terminal  cost. Therefore, we assume that our ansatz functions are   piecewise  linear and  quadratic polynomial functions  that are involved  in the running and terminal  cost. The aforementioned  ansatz functions are then  scaled with optimal coefficients  $\theta^* (n)\in \R^{\M}_+$ to obtain the value function approximation.
		
		The complete implementation using  quantization and parametrization can be described as follows. For all time steps $n=N_t-1,\cdots,0$, we first  parametrize the  known  value function of the next time step 
		$V(n+1,\cdot)  \approx \sum_{j=1}^{\M} \theta^*_j(n+1) \varphi_j(\cdot) $		with $\theta^*_j (n+1)\in \R^{\M}_+$   that is  the minimizer of   the corresponding regression problem. 
		
		The value function  $V(n,\cdot)$ at time $n$,  is then approximated  at all grid points $(x_i)_{i \in \mathcal{I}}$ using the Bellman equation \eqref{BE}  where we replace the value function at the next time step $V(n+1,\cdot)$ by the regression ansatz $ \sum_{j=1}^{\M} \theta^*_j(n+1) \varphi_j(\cdot)$. Hence, the calculations details are as follows,

		\begin{align}		
			V(n,x_i)&\approx \inf_{\nu \in  \mathcal{U}}\Big\{\Psi(x_i,\nu)+ 
			\E \big[  \sum_{j=1}^{\M} \theta^*_j(n+1) \varphi_j(\mathcal{T}(n,x_i,\nu,\Ec_{n+1})) \big] \Big\}\\
			&\approx \inf_{\nu \in  \mathcal{U}}\Big\{\Psi(x_i,\nu)+ 
			\sum_{j=1}^{\M} \theta^*_j(n+1) \E \big[ \varphi_j(\mathcal{T}(n,x_i,\nu,\Ec_{n+1}))\big]\Big\}.
		\end{align}

		\begin{algorithm}
			\caption{Backward recursion with quantization and regression algorithm}\label{algBRP}
			\vspace{.3cm}		
			\KwResult{  Find the approximate  value function $\widehat{V}  $ and the approximate  optimal decision rule  $\widehat{u}^*$}
			\vspace{.3cm}
			\KwSty{	\textbf{Step 1:}}		 
			Fix the quantization size $\L$, and compute the quantizers and associated weights,\;
			$\Gamma^*_{\L}= \left(  \xi^*_1, \ldots, \xi^*_{\L} \right);$ 
			$\left( w_1,\ldots, w_{\L} \right) $\;		

			\vspace{.3cm}
			\KwSty{	\textbf{Step 2:}} {Compute for all $x_i, i \in \mathcal{I}$,
				$V(N_t,x_i) = \Phi(x_i)$}
			\vspace{.3cm}
			
			\KwSty{\textbf{Step 3:} }\;
			
			\Indp \For {$n=N_t- 1,\ldots,0 $,}{
				
				Find the parametrization 	$\theta^*(n+1)=(\theta^*_j(n+1))_{j=1,\ldots,\M}$ of $\widehat{V}(n+1,\cdot)$ by solving the  regression problem \;
				
				\begin{eqnarray*}
					\theta^*(n+1) = \argmin_{\theta \in \R^{\M}_+}	\sum_{i \in \mathcal{I}}\Vert \widehat{V}(n+1,x_i)- \sum_{j=1}^{\M} \theta_j(n+1) \varphi_j(x_i)\Vert^2 	.
				\end{eqnarray*}\

				Compute for all $x_i, i \in \mathcal{I}$
				
				$$\widehat{V}	(n,x_i)
				=\inf_{\nu \in  \mathcal{U}}\Big\{\Psi(x_i,\nu)+ 
				\sum_{j=1}^{\M} \theta^*_j(n+1) \sum_{l=1}^{\L} w_l        \varphi_j(\mathcal{T}(n,x_i,\nu,\xi_l)) \Big\}$$\;
				
				and the minimizer $\widehat{u}^{\;*}(n,x_i)$   given by:\;
				
				$$\widehat{u}^{\;*}(n,x_i)
				= \argmin_{\nu \in  \mathcal{U}}\Big\{\Psi(x_i,\nu)+ 
				\sum_{j=1}^{\M} \theta^*_j(n+1) \sum_{l=1}^{\L} w_l        \varphi_j(\mathcal{T}(n,x_i,\nu,\xi_l)) \Big\}.$$
			}
		\end{algorithm}
		
		Moreover, considering the approximation of the distribution of   $\Ec_{n+1}$ by quantization grid   $\Gamma^*_{\L}=\{ \xi^*_1, \ldots, \xi^*_{\L} \} $ 
		with associated  probability weights  $\{ w_1, \ldots, w_{\L} \} $, the unconditional expectation is substituted by a weighted average, and we have the approximation 
		
		\begin{align}		
			V(n,x_i)
			&\approx \inf_{\nu \in  \mathcal{U}}\Big\{\Psi(x_i,\nu)+ 
			\sum_{j=1}^{\M} \theta^*_j(n+1)  \sum_{l=1}^{\L} w_l       \varphi_j(\mathcal{T}(n,x_i,\nu,\xi_l)) \Big\}.
		\end{align}

		We are then left with a deterministic point-wise optimization w.r.t. $\nu \in  \mathcal{U}$ which is simple since we have discretized $ \mathcal{U}$. We therefore derive the approximation of the optimal decision rule as

		\begin{align}		
			\widetilde{u}^*(n,x)
			&\approx \argmin_{\nu \in  \mathcal{U}}\Big\{\Psi(x_i,\nu)+ 
			\sum_{j=1}^{\M} \theta^*_j(n+1) \sum_{l=1}^{\L} w_\ell       \varphi_j(\mathcal{T}(n,x_i,\nu,\xi_l)) \Big\}.
		\end{align}
		
		The entire procedure  is summarized in  Algorithm \eqref{algBRP}.

		\section{Numerical results}\label{sect4}
		
		In this section, we delve into the numerical experiment results of our optimal control problem. The SOCP with a partially observed state process is solved numerically using the backward recursion algorithm with quantization approximation coupled to linear interpolation or value function regression. Our focus is to understand the qualitative behavior of the optimal policy to be implemented and analyze the associated minimum aggregated costs or value function.    First, we will introduce the numerical settings by
		providing a comprehensive overview of the parameters, cost coefficients, and other constants relevant to our control problem.
		Next, a benchmark scenario is constructed from situations where the control variables are not activated (smallest value for each control variable). After that, numerical experiments based on the linear interpolation approach will be performed. In addition, we will investigate a faster  numerical method using regression ansatz functions that are aligned with the insights gained from the results of the linear interpolation approach. 
		

		\subsection{Parameter setting}
		In order to solve the optimal control problem numerically, it is necessary to specify the values of certain parameters, and  state process components bounds. Moreover, the coefficients associated with the different costs must be defined. In this setting, the time horizon $N_t$ is considered to be $120$ days. This corresponds to an  epidemic time frame of approximately four months. Such a temporal time frame allows analyzing the optimal control strategy to be  applied during the early stages of the epidemic. This time horizon can further be extended without any particular numerical challenges when dealing with constant parameters. Regarding model parameters,  the transition intensity between compartments is the inverse of the average sojourn time in the compartment, as explained in \citet{laredo2020statistical}. Accordingly, we consider $\beta=0.3$,   as in   \citet{landsgesell2020spread}  related to the early stage of the Covid-19 epidemic in Germany.  For  recovery rates,     we assign $\gamma^-=0.067$ for asymptomatic individuals, and   $\gamma^+=0.1$   for quarantined individuals. This translates to an average infection duration of 15 days for asymptomatic cases and 10 days for  detected infectious cases. The assumption is that  detected infectious individuals are aware of their condition and can take some actions to improve their health.  Additionally,  for the parameters $\eta^-$, $\eta^+$ and $\gamma^H$, we adopt values similar to those presented  in \citet{charpentier2020covid}. A detailed overview of these parameters is provided in the Table \ref{tabparam}.

		\begin{table} 
			\centering
			\begin{tabular}{|l||l|l|l|l|l|l|}
				\hline
				Parameters& $\beta$ & $\gamma^-$&$\gamma^+$& $\gamma^H$& $\eta^-$& $\eta^+$\\ \hline
				Values&0.25 &0.067&0.1& 0.09& 0.002&0.003\\ \hline
			\end{tabular}
			\vspace{.3cm}
			\caption{Model parameters}
			\label{tabparam}
		\end{table}

		\begin{table}
			\centering
			\begin{tabular}{ll|l|l|l}
				\hline
				Costs coefficients && Values&Thresholds& Values\\ \hline\hline
				Lockdown& $( \overline{a}_L, a_L,b_L)$& $(10000,80,0.8)$ &&\\ [.5ex]
				
				Test & $( \overline{a}_T, a_T,b_T)$& $(2000,1200,80)$& $\overline{x}^{Test}$ &50\\ [.5ex]
				
				Vaccination& $( \overline{a}_V, a_V,b_V)$& $(4000,1500,90)$&	$\overline{x}^{Vacc}$& 40\\ [.5ex]
				
				Hospital & $( \overline{a}_H, a_H,b_H)$& $(1000,2000,100)$& $\overline{x}^H$& 10\\ [.5ex]
				
				Infection $I^-$ & $( \overline{ a}_{I^-}, a_{I^-},b_{I^-})$& $(0,1500,150)$&  $\overline{x}^{I^-}$& 100\\ [.5ex]
				
				Infection $I^+$ & $( \overline{ a}_{I^+}, a_{I^+},b_{I^+})$& $(0,1000,100)$&$\overline{x}^{I^+}$& 150\\ [.5ex]
				
				Terminal cost  & $( \overline{ a}_{N_t,I^-}, a_{N_t,I^-},b_{N_t,I^-})$& $(0,15000,1500)$ &$\overline{x}^{I^-}_{N_t}$& $100$\\ [.5ex]
				& $( \overline{ a}_{N_t,I^+}, a_{N_t,I^+},b_{N_t,I^+})$&$(0,10000,1000)$ &$\overline{x}^{I^+}_{N_t} $& $150$\\ [.5ex]
				& $( \overline{ a}_{N_t,H}, a_{N_t,H},b_{N_t,H})$&$(10000,20000,1000)$ &$\overline{x}^{H}_{N_t}$& $10$\\ [.5ex]
				\hline 
				
			\end{tabular}
			\vspace{.5cm}
			\caption{Costs coefficients and thresholds values}
			\label{coefftresh}
		\end{table}
		
		Once the model parameters have been established, the next step is to assign cost coefficients and thresholds.  Assigning cost coefficients is equivalent to associating weights to each marginal cost. However, this process raises ethical questions since it is not that easy to determine how important economic concerns are, compared to human life. To avoid this issue, the weights considered here are purely for experimental purposes. The Table \ref{coefftresh} provides experimental values for cost coefficients and thresholds. To obtain terminal coefficient values, we simply scale the corresponding running coefficients by a factor of $10$.

		The boundary values of the state space are obtained by means of numerical simulations as described in the following benchmark scenario.
		
		\subsection{Benchmark scenario}
		
		In our benchmark scenario, we examine the epidemic's progression in the absence of any intervention by policymakers. We set the control variables to their respective minimum values. Specifically, the lock-down and vaccination control variables are both set to the natural minimum of $0$, while the testing control variable is assigned a small value, $u_{min}^T=0.001$. This value accounts for individual self-testing that is not financed by decision makers. By doing so, we allow  the conditional means of the hidden compartments to reach their maximum levels. Additionally, varying the control variables enables us to establish bounds for observable compartment, the diagonal entries of the covariance matrix and  the correlation coefficient. 
		
		In Figure \ref{bench}, we present sample trajectories of the observed compartments alongside the conditional means of the hidden compartments. Specifically, in Panel \subref{pa}, we depict the dynamics of all compartment sizes. The epidemic peak occurs approximately 25 days after the initial outbreak, and the estimated number of non-detected infected individuals rapidly declines, nearly vanishing by day 85.  Upon zooming into small compartment sizes, as illustrated in Panel \subref{pb}, the number of detected infected individuals decreases progressively. This is certainly due to the absence of a testing strategy. The hospital compartment size  fluctuates  and in the  first phase  increases up to exceeding the hospital threshold and in a second phase decreases toward  zero. 
		Additionally, estimated values of hidden states are subject to a certain degree of uncertainty quantified by the evolution of the entries of  the covariance matrix associated to hidden states on Panel \subref{pc}. The diagonal entries $\condvar^1$ and $\condvar^2$ grow considerably, before decreasing towards $0$ as the epidemic dynamic stabilizes. The interdependence between the estimates of both hidden states is expressed through the dynamic of the correlation coefficient.  The correlation is first positive as both estimates $\condmean^1$ and $\condmean^2$ are increasing, then it transitions to negative values as soon as $\condmean^1$ starts decreasing and $\condmean^2$ keeps increasing.
		
		Despite utilizing a stochastic model coupled with partial information, we note that the evolution of the epidemic is in accordance with the result in \citet{charpentier2020covid} wherein a purely deterministic ODE system is analyzed. This is not surprising since it is proven that for relative subpopulation sizes, the ODE system is the limiting case, as the total population size tends to infinity, of the associated stochastic differential equation (SDE) model that we have considered.

		\begin{figure}
			\centering
			\begin{subfigure}{.5 \textwidth}
				\centering
				\includegraphics[width=1\linewidth]{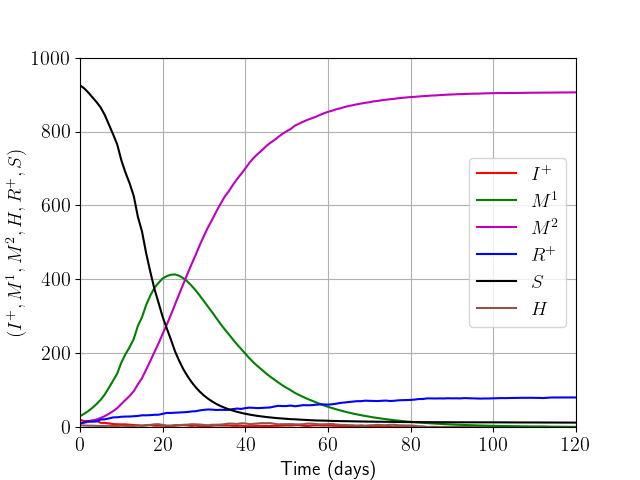}
				\caption{ All compartments}
				\label{pa}
			\end{subfigure}%
			\begin{subfigure}{.5 \textwidth}
				\centering
				\includegraphics[width=1 \linewidth]{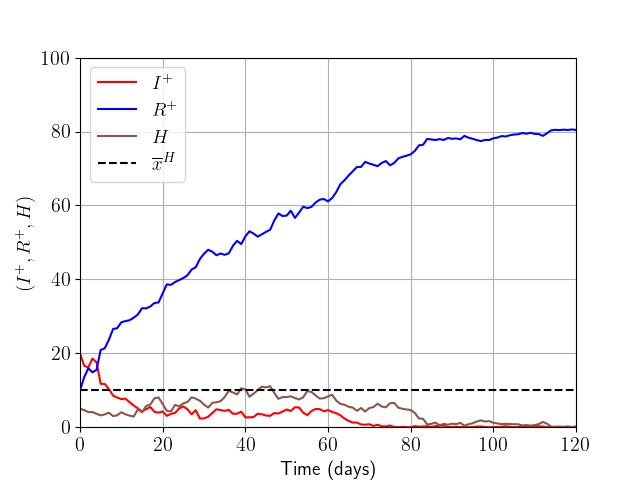}
				\caption{Zoom in for small compartments }
				\label{pb}
			\end{subfigure}
			\begin{subfigure}{.5 \textwidth}
				\centering
				\includegraphics[width=1\linewidth]{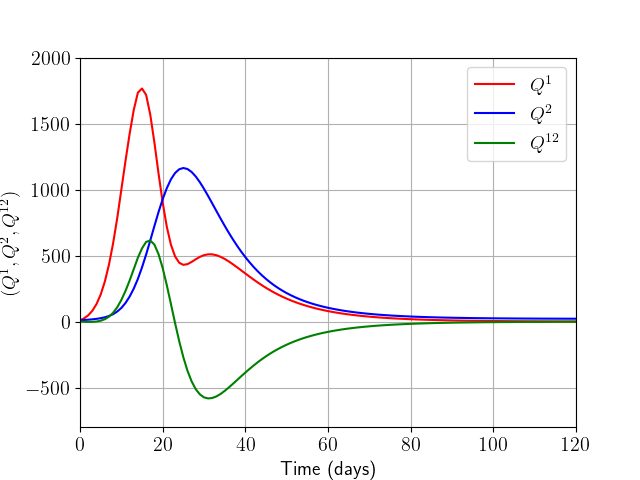}
				\caption{ Conditional covariance matrix entries}
				\label{pc}
			\end{subfigure}%
			\begin{subfigure}{.5 \textwidth}
				\centering
				\includegraphics[width=1\linewidth]{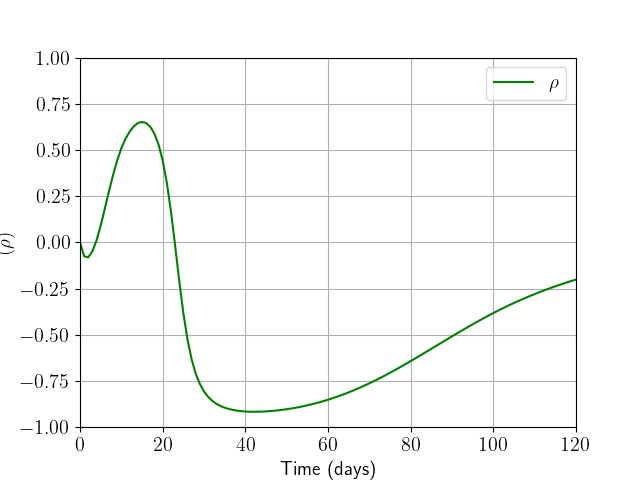}
				\caption{Correlation coefficient}
				\label{pd}
			\end{subfigure}
			\caption{Sample paths of state components  and correlation coefficient in the absence of control measures }
			\label{bench}
		\end{figure}
		
		To assess the impact of control measures, we simulate sample trajectories of the state dynamic for control measures  fixed at moderate values, namely $(u^L,u^T,u^V)=(0.2,0.03,0.015) $.  Figure  \ref{benchb} depicts panels that are similar to those in Figure  \ref{bench}. we observe that as expected, the infection maximum is smaller compare to the case of no control measure. Additionally, more infectious people are detected and can be put in quarantine.

		\begin{figure}
			\centering
			\begin{subfigure}{.5 \textwidth}
				\centering
				\includegraphics[width=1\linewidth]{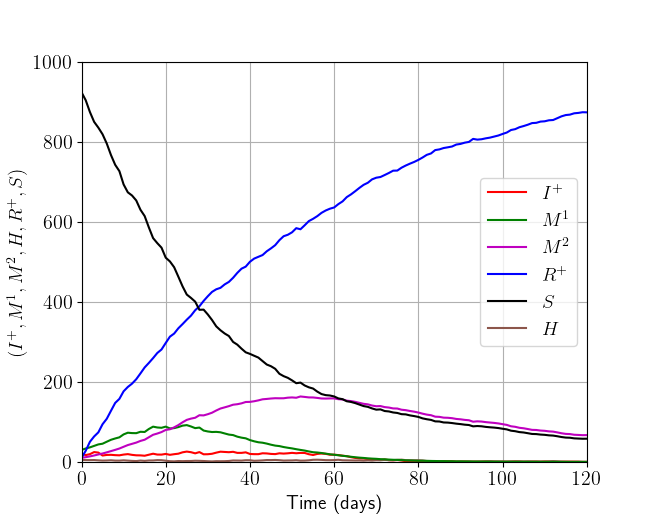}
				\caption{ All compartments}
				\label{pab}
			\end{subfigure}%
			\begin{subfigure}{.5 \textwidth}
				\centering
				\includegraphics[width=1 \linewidth]{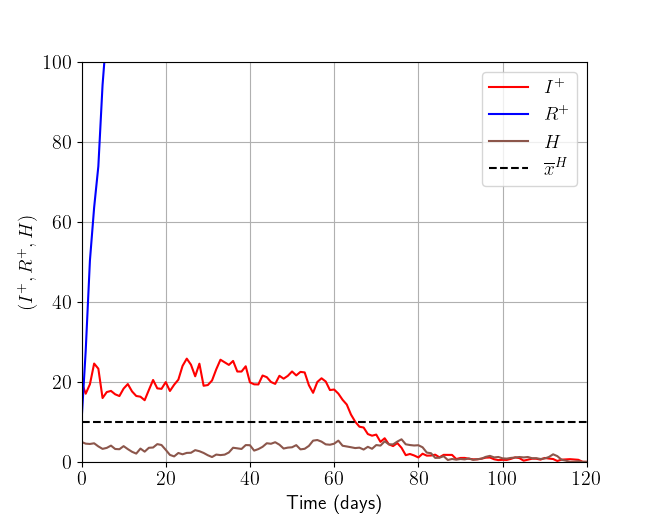}
				\caption{Zoom in for small compartments }
				\label{pbb}
			\end{subfigure}
			\begin{subfigure}{.5 \textwidth}
				\centering
				\includegraphics[width=1\linewidth]{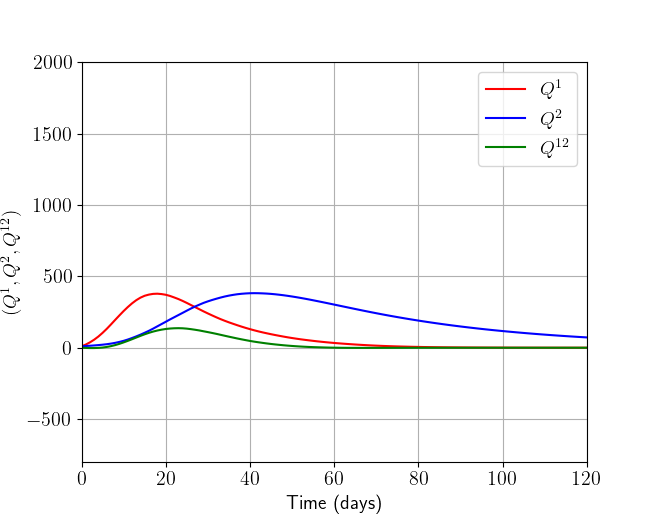}
				\caption{ Conditional covariance matrix entries}
				\label{pcb}
			\end{subfigure}%
			\begin{subfigure}{.5 \textwidth}
				\centering
				\includegraphics[width=1\linewidth]{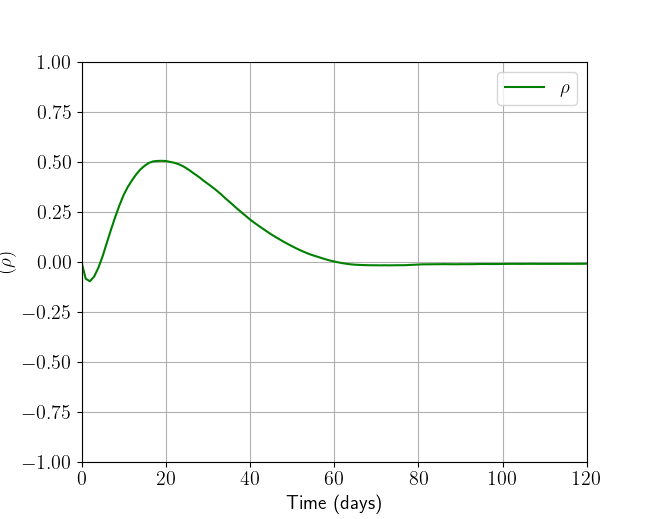}
				\caption{Correlation coefficient}
				\label{pdb}
			\end{subfigure}
			\caption{Sample paths of state components  and correlation coefficient moderate control measures, $(u^L,u^T,u^V)=(0.2,0.03,0.015) $ }
			\label{benchb}
		\end{figure}

		\subsection{Results using state discretization, quantization and linear interpolation}

		The first numerical solution to our optimal control problem employs the backward recursion algorithm, approximating the conditional expectation through quantization and linear interpolation. Due to the computationally demanding aspect of our algorithm, we utilize the values in Table  \ref{disctab}, for the number of  discretization grid points for the state and control space.	
		Even with this coarse discretization grid, the computational load is quite significant. Nevertheless, this grid yields already  informative results for the value function and the optimal decision rule. 
		
		\begin{table}
			\centering
			\begin{tabular}{|ll@{\hspace{1cm}}l|l|}
				\hline
				\textbf{	Discretization}& Range&&Values\\
				\hline \hline
				Time step& & $\Delta t$& 1 day\\
				Time horizon& &$T=N_t \Delta t$&120 days\\	
				Quantization level&& $N_l$& 125\\			
				State ~ $\condmean^1$&$ [1,450]$  & $\widetilde{N}_1$&6 \\
				\phantom{State ~} $\condmean^2$&$   [1,450]$  & $\widetilde{N}_2$&6 \\
				\phantom{State ~} $\condvar^1 $&$  [1,1200]$  & $\widetilde{N}_3$&6 \\
				\phantom{State ~} $\condvar^2 $&$   [1,1200]$  & $\widetilde{N}_4$&4 \\
				\phantom{State ~} $\rho$&$   [-0.5,0.5]$  & $\widetilde{N}_5$&3 \\
				\phantom{State ~} $Z^1$&$   [1,350]$  & $\widetilde{N}_6$&6 \\
				\phantom{State ~} $Z^2$&$   [1,990]$  & $\widetilde{N}_7$&6 \\
				\phantom{State ~} $Z^3 $&$  [1,50]$  & $\widetilde{N}_8$&6 \\
				Control $u^L$&$  [0,1]$  & $\widetilde{N}^L$&6 \\
				
				\phantom{Control} $u^T$&$   [0.001,0.07]$  & $\widetilde{N}^T$&3 \\	
				\phantom{Control} $u^V$&$   [0,0.1]$  & $\widetilde{N}^V$&3 \\
				\hline
			\end{tabular}
			\caption{Table of discretization values}
			\label{disctab}
		\end{table}
		
		\begin{remark}
			To apply the linear interpolation of the value function   in dimension $8$, it is  necessary to perform a Delaunay triangularization of the state space. Furthermore, this is followed by the computation of the location of the triangle in which each interpolant point belongs, as well as the calculation of the barycentric coordinates.   Upon computing the value function with fewer grid points and for the model without the hospital, it was observed that the value function does not exhibit a direct dependence on the correlation coefficient between the two hidden states and the second diagonal entry of the conditional covariance matrix. In other words, when all the other variables are held constant, the value function shows to be  independent on both  $\rho$ and $Q^2$. This allows the curse of dimensionality to be alleviated by performing the linear interpolation in a six-dimensional subspace generated by the remaining state components, namely,  $( \condmean^1, \condmean^2, \condvar^1,  Z^1,Z^2, Z^3).$		
			Another trick that we employed, due to the use of constant parameters, to reduce the computational time was to compute the search of triangles and the computation of barycentric coordinates for all the interpolants offline, prior to the backward recursion. 
			
		\end{remark}
		
		\subsubsection{Value function and optimal decision rule as function of  $(\condmean^1, Z^1)$}
		
		The value function of our optimization problem is a function of eight variables, which presents a significant challenge in terms of visual representation. In order to gain insight into the behavior of this multi-variable function, it is necessary to make projections by setting specific values for certain variables. Since we are interested in an infectious disease, the two most  important  variables are the infection variables:   the detected infected  represented by  $Z^1=I^+$ and the estimated number of undetected   infected $\condmean^1$. Hence, we visualize the value  function and the optimal  decision  rule w.r.t.  $\condmean^1, Z^1$ and   the remaining variable are held constant. The two scenarios  that we consider  are: $(i)$ the case  where all  the other  compartments are almost empty, except the susceptible $S$ that can be recovered by normalization,  and $(ii)$ the case of other  compartments with moderate sizes. 
		
		\paragraph{Small values of  $(  \condmean^2, \condvar^1,\condvar^2, \rho,  Z^2, Z^3)$ }
		
		This configuration may correspond to the early stage of the disease during which the majority of individuals have not yet recovered from the infection. It is therefore possible to represent the projected functions for a sample of time grid points. Typical behavior is observed at time steps  $t=118$ days and $t=100$  days.   In Figures \ref{LI118small}, and \ref{LI100small}, the remaining variables are set to small values as follows: 
		$  \m^2=1, \q^1=1,\q^2=1,\rho=-0.5,  \z^2=1, \z^3=1$. 
		
		\begin{figure}[p]
			\hspace{-2cm}	\includegraphics[width=1.2\linewidth]{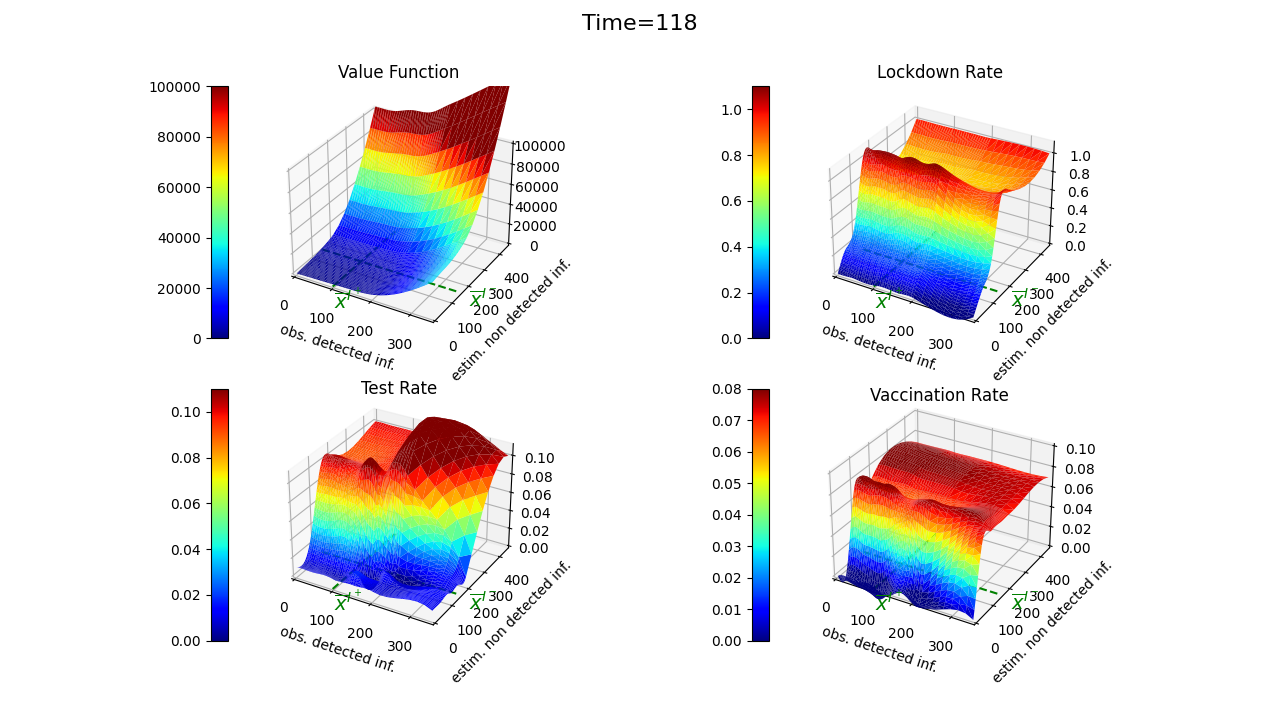}
			
			\caption{ \textbf{Linear interpolation method}:   value function and optimal decision rule   at time $118$ for  small values of  $(  \m^2, \q^1,\q^2, \rho,  \z^2, \z^3)$  }
			\label{LI118small}
		\end{figure}

		\begin{figure}[p]
			\hspace{-2cm}
			\includegraphics[width=1.2\linewidth]{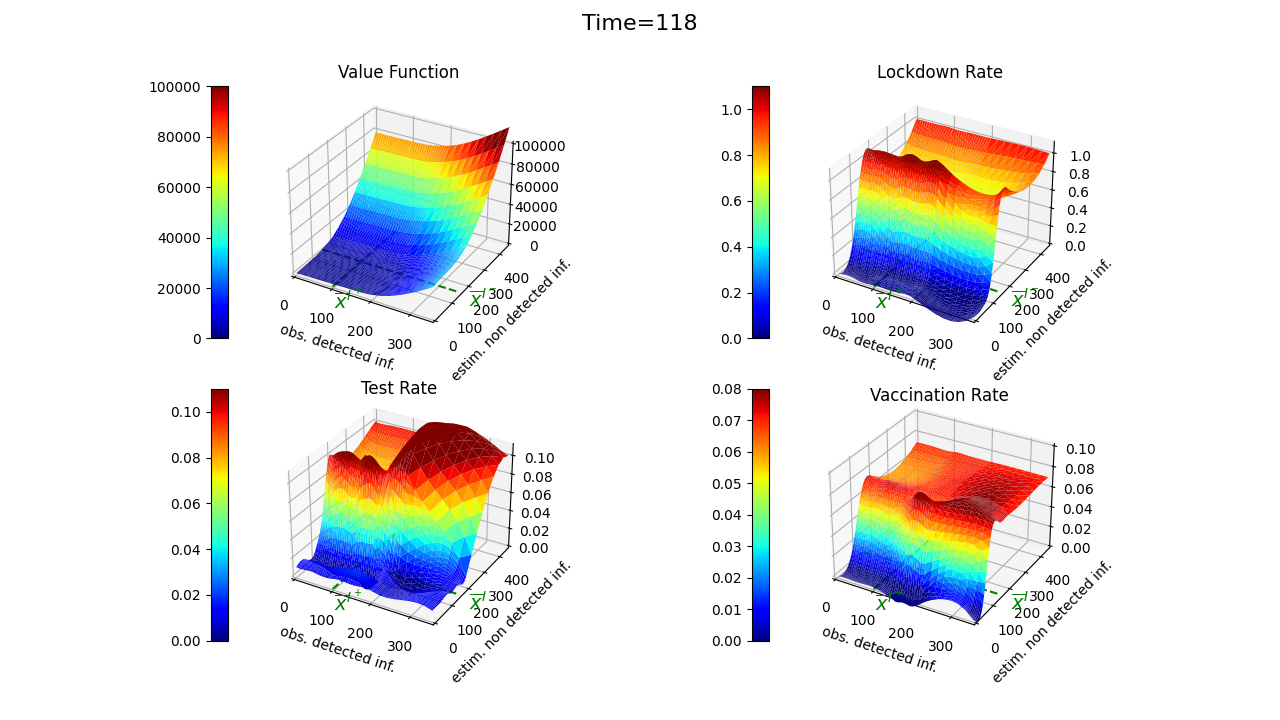}
			
			\caption{  \textbf{Regression method}: value function and optimal decision rule  at time $118$ for  small values of  $(  \m^2, \q^1,\q^2, \rho,  \z^2, \z^3)$  }
			\label{LI118smallfit}
		\end{figure}
		

		\begin{figure}[p]
			\hspace{-2cm}	\includegraphics[width=1.2\linewidth]{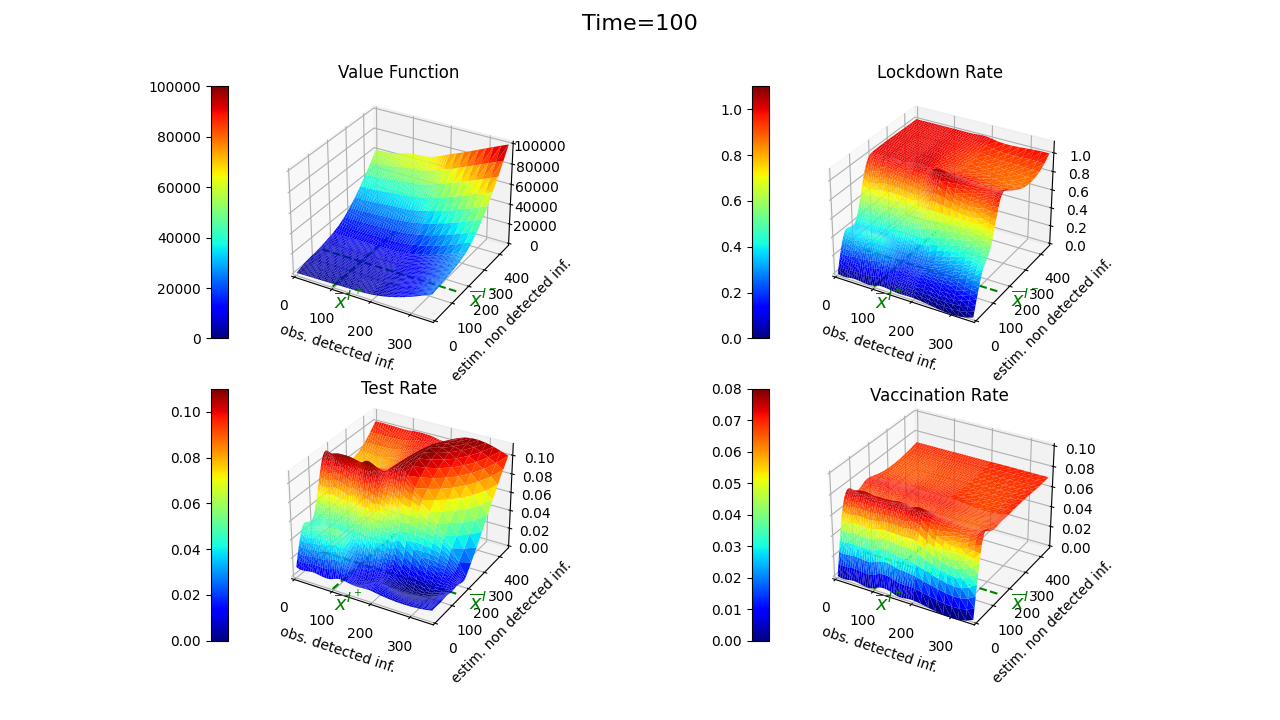}
			\caption{\textbf{Linear interpolation method}: value function and optimal decision rule  at time $100$ for  small values of  $(  \m^2, \q^1,\q^2, \rho,  \z^2, \z^3)$  }
			\label{LI100small}
		\end{figure}

		\begin{figure}[p]
			\hspace{-2cm}
			\includegraphics[width=1.2\linewidth]{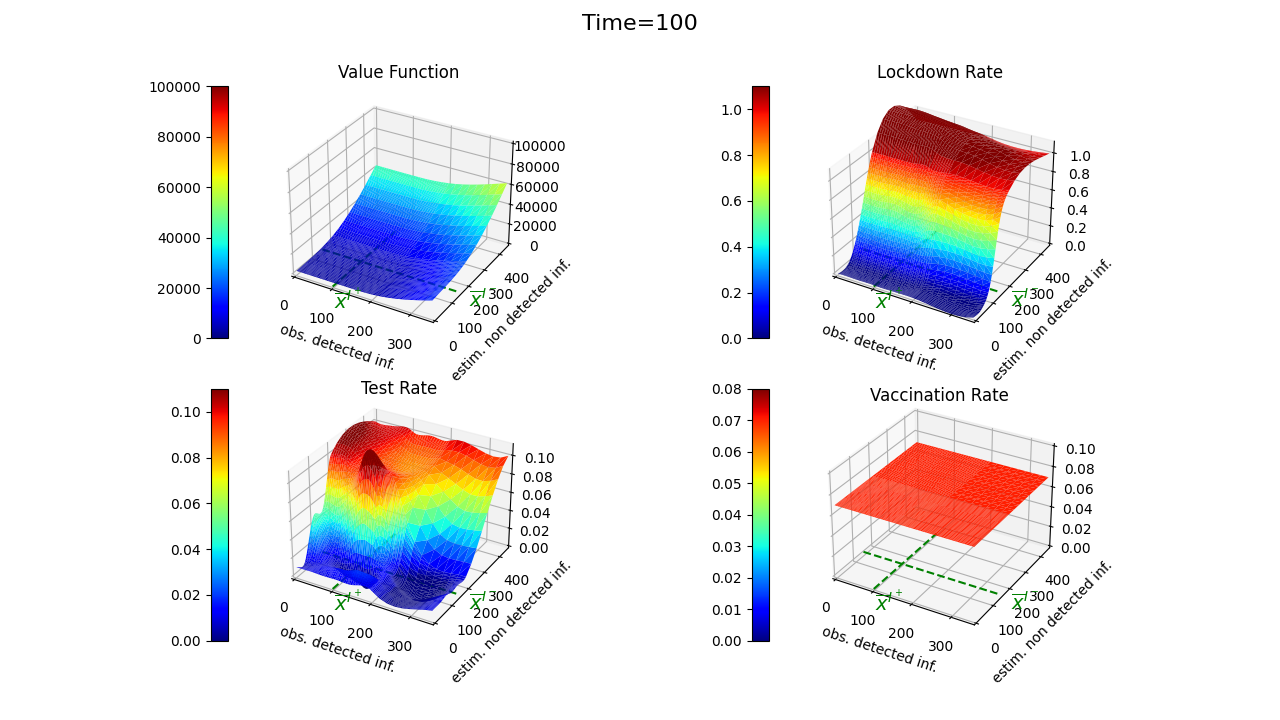}
			\caption{\textbf{Regression method}: value function and optimal decision rule   at time $100$ for  small values of  $(  \m^2, \q^1,\q^2, \rho,  \z^2, \z^3)$  }
			\label{LI100smallfit}
		\end{figure}

		We note that on one hand, the value function as a function of $Z^1$ and $\condmean^1$, inherits the linear and quadratic structure analogous to that observed in the running and terminal costs. The value function initially grows linearly for small values of $Z^1$ and the numbers of infected individuals increase and cross the respective thresholds, the growth becomes quadratic, which is a consequence of the quadratic structure of the penalty function. On the other hand, it is noteworthy that the optimal decision rule remains almost stationary up to 20 days before the time horizon, with only slight changes occurring as the terminal time approaches.  In particular, the rate of lockdown appears to be largely independent of the number of detected infected individuals, with a significant dependence on the estimated number of undetected infectious cases. As a function of $\condmean^1$, the lockdown rate evolves in a progressive manner, shifting from no social distancing when $\condmean^1$ is small to a pronounced lockdown as the threshold 
		$\overline{x}^{I^-}$ is exceeded.  As with the lockdown rate, the test and vaccination rates are primarily contingent on the number of estimated undetected infectious cases, with a lesser dependence on $Z^1=I^+$. In addition,  they both go from small detection and no vaccination rate when $\condmean^1$ is negligible to being applied at maximum rate as $\condmean^1$ approaches and crosses the threshold  $\overline{x}^{I^-}$.  In contrast to the gradual increase in the lock-down and testing rates, the vaccination rate rises rapidly to its maximum value. Consequently, the optimal decision rule is a combination of the three control variables, which are typically active at comparable levels. This indicates that, when necessary, intervention measures tend to combine the three controls to bring down the epidemic.
		This numerical result demonstrates that whenever the three control   measures are available, it is advantageous  to act simultaneously on all of them to mitigate the spread of the disease.

		\paragraph{Moderate values of  $(  \condmean^2, \condvar^1,\condvar^2, \rho,  Z^2, Z^3)$}
		
		To gain further insight into the value function and the optimal decision rule, we consider the remaining variables held fixed at moderate values: $\m^2 = 100$, $\q^1 = 500$, $\q^2 = 500$, $\rho = 0.5$, $\z^2 = 200$, and $Z^3 = 30$. This setup represents the middle stage of the epidemic, with an estimated number of recovered individuals at $300$.
		
		Figures \ref{LI118mod} and \ref{LI100mod} illustrate that the value function behaves similarly to the case with a fixed small compartment above, initially linear and then quadratic as it approaches the thresholds $\overline{x}^{I^-}$ and $\overline{x}^{I^+}$. The optimal decision rule primarily still depends on the estimated number of undetected infected individuals, $\condmean^1$, shifting progressively  from lower to higher control variable rates as $\condmean^1$ approaches and exceeds the threshold $\overline{x}^{I^-}$.
		
		Meanwhile, when both $I^+$ and $\condmean^1$ exceed their respective thresholds, the rates for lockdown and vaccination decline, while the testing rate remains high. This indicates that, in a scenario where a moderate number of individuals have already recovered (being in either compartment $R^+$ or $R^-$) and the sizes of infected compartments $I^-$ and $I^+$ have surpassed their thresholds, it becomes unnecessary to apply lockdown and vaccination measures since almost no one is susceptible anymore. 
		
		\begin{figure}[p]
			\hspace{-2cm}
			\includegraphics[width=1.2\linewidth]{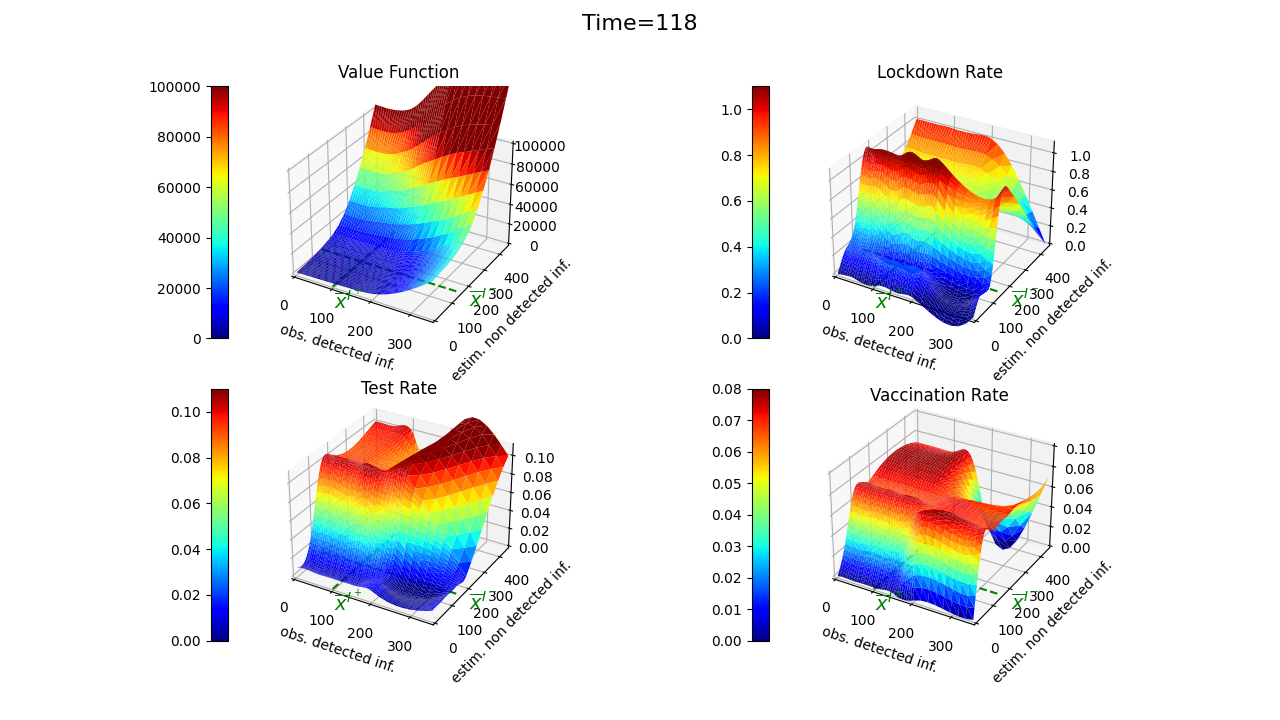}
			\caption{  \textbf{Linear interpolation method}: value function and optimal decision rule at time $118$ for  moderate values of $(  \m^2, \q^1,\q^2, \rho,  \z^2, \z^3)$  }
			\label{LI118mod}
		\end{figure}

		\begin{figure}[p]
			\hspace{-2cm}
			\includegraphics[width=1.2\linewidth]{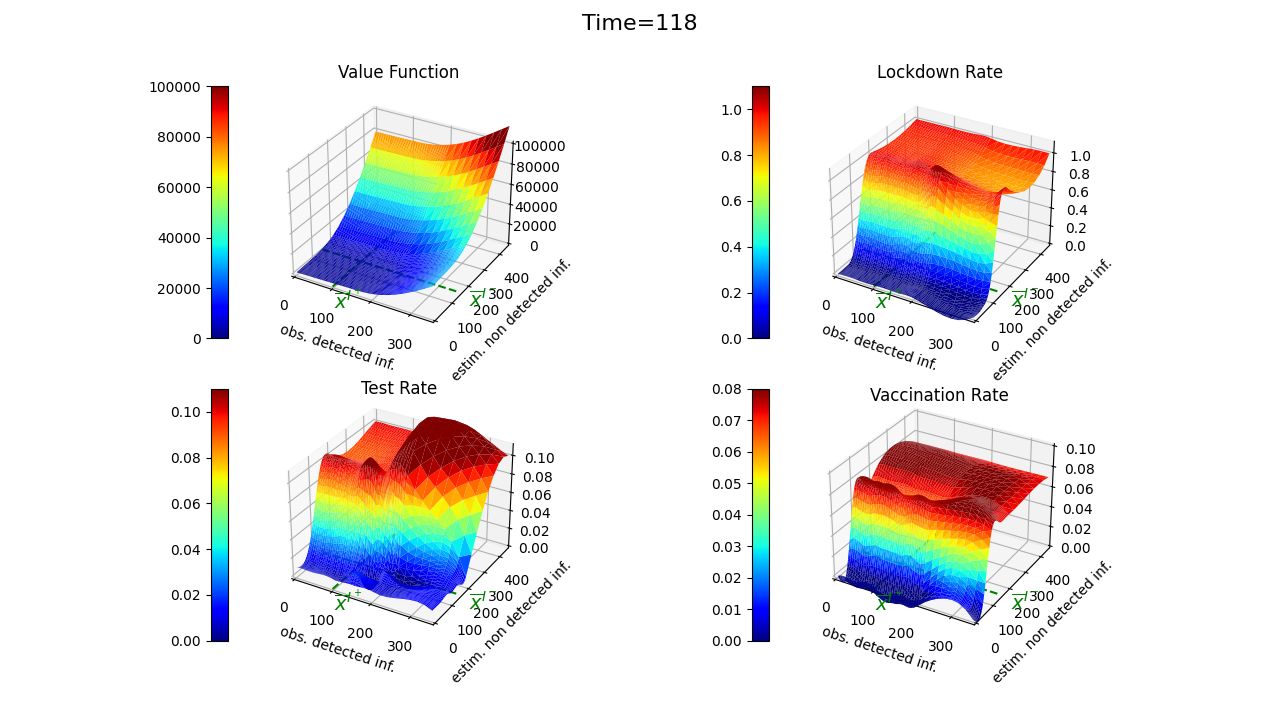}
			\caption{\textbf{Regression method}:  value function and optimal decision rule    at time $118$ for  moderate values of  $(  \m^2, \q^1,\q^2, \rho,  \z^2, \z^3)$  }
			\label{LI118modfit}
		\end{figure}

		\begin{figure}[p]
			\hspace{-2cm}
			\includegraphics[width=1.2\linewidth]{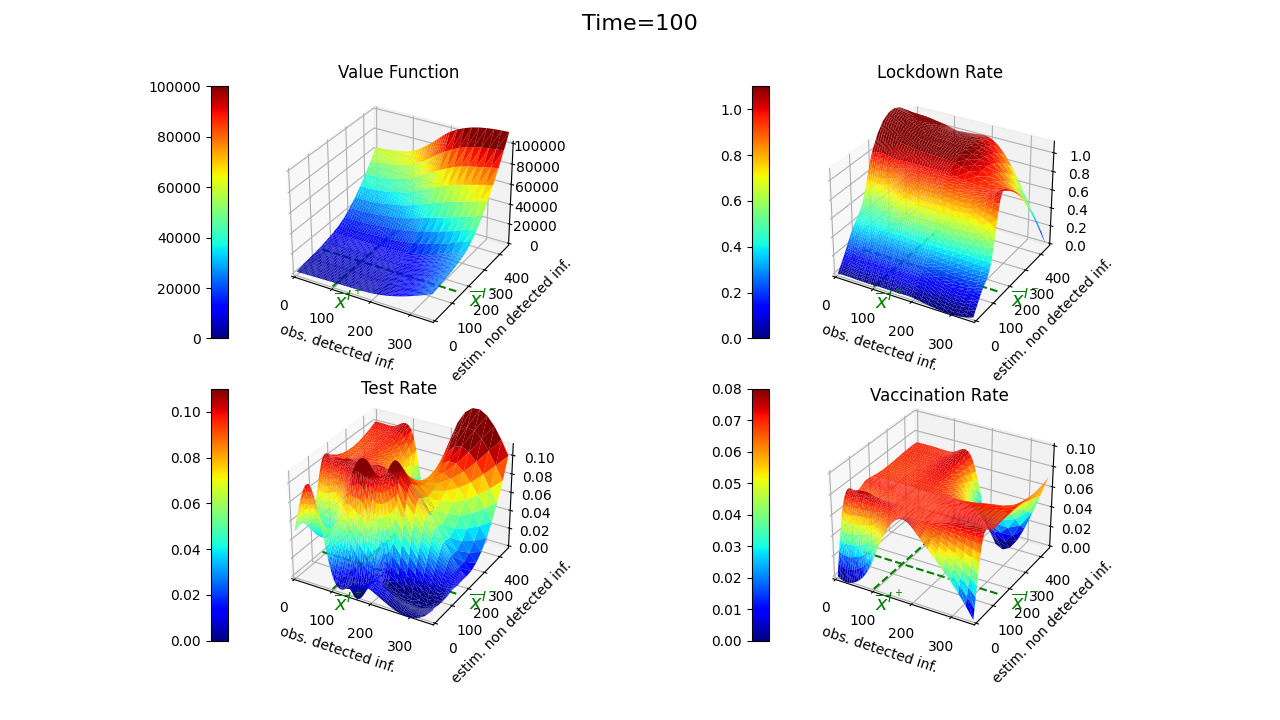}
			\caption{\textbf{Linear interpolation method}:  value function and optimal decision rule   at time $100$ for  moderate values of  $(  \m^2, \q^1,\q^2, \rho,  \z^2, \z^3)$  }
			\label{LI100mod}
		\end{figure}

		\begin{figure}[p]
			\hspace{-2cm}
			\includegraphics[width=1.2\linewidth]{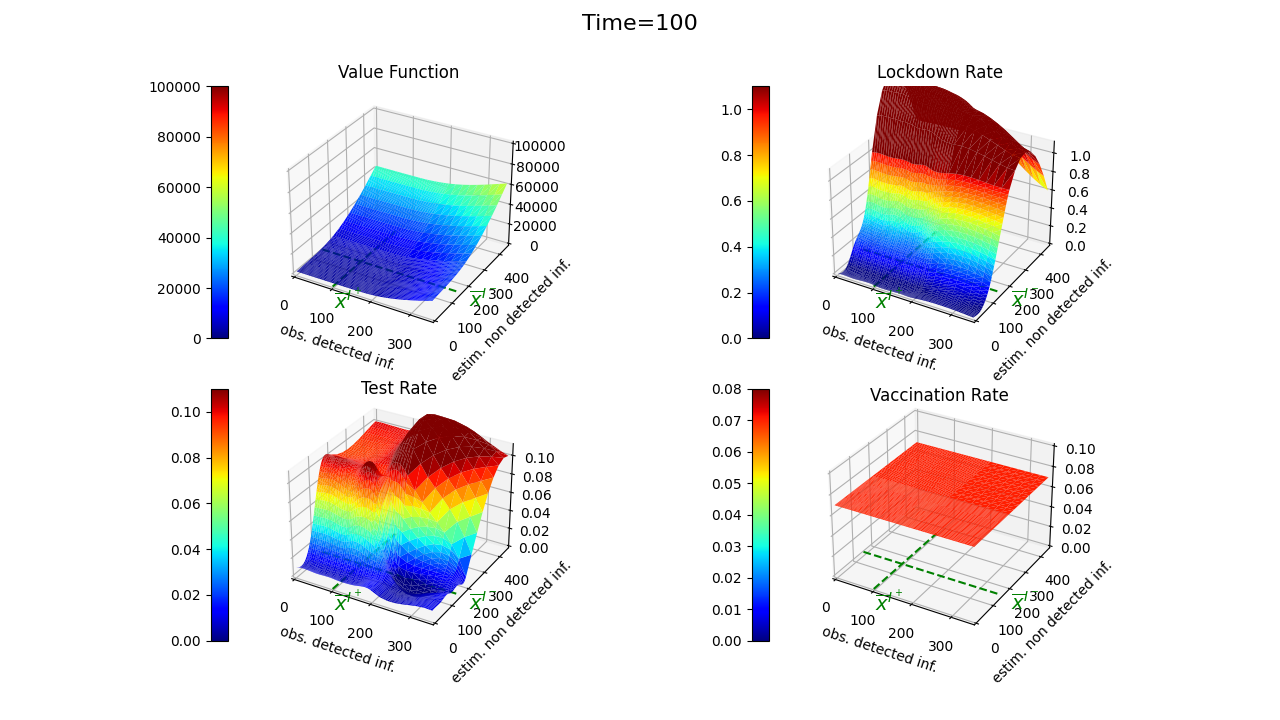}
			\caption{\textbf{Regression method}: value function and optimal decision rule at time $100$ for  moderate values of  $(  \m^2, \q^1,\q^2, \rho,  \z^2, \z^3)$  }
			\label{LI100modfit}
		\end{figure}

		\subsubsection{Optimal path  of the epidemic  state process}
		
		With the knowledge of the optimal decision rule at all grid points of the state space, we can simulate an optimal sample path of the state process using the optimal control of the nearest neighbor projection. Figures \ref{optpath} and \ref{optcontrpath} illustrate the dynamics of sample compartment sizes and the dynamics of the applied control measures.	
		We observe that the infection compartment sizes, $I^+$ and the estimate of $I^-$, are driven quickly toward zero in approximately 25 to 30 days. Meanwhile, the estimates of compartment sizes $S$ and $R^-$ decrease and stabilize once there are no more infected individuals in the population. The hospital compartment size fluctuates slightly but remains low and below the hospital threshold throughout the infection period, unlike the benchmark scenario where the threshold was exceeded.	
		Regarding the optimal control measures, we see that initially, a strong lockdown is applied along with significant levels of testing and vaccination. The lockdown is gradually eased after one week, while testing and vaccination continue. Approximately one week after the end of the lockdown, testing is stopped, and only vaccination continues for approximately two more weeks.   
		
		\begin{figure}
			\centering
			\includegraphics[width=0.8\linewidth]{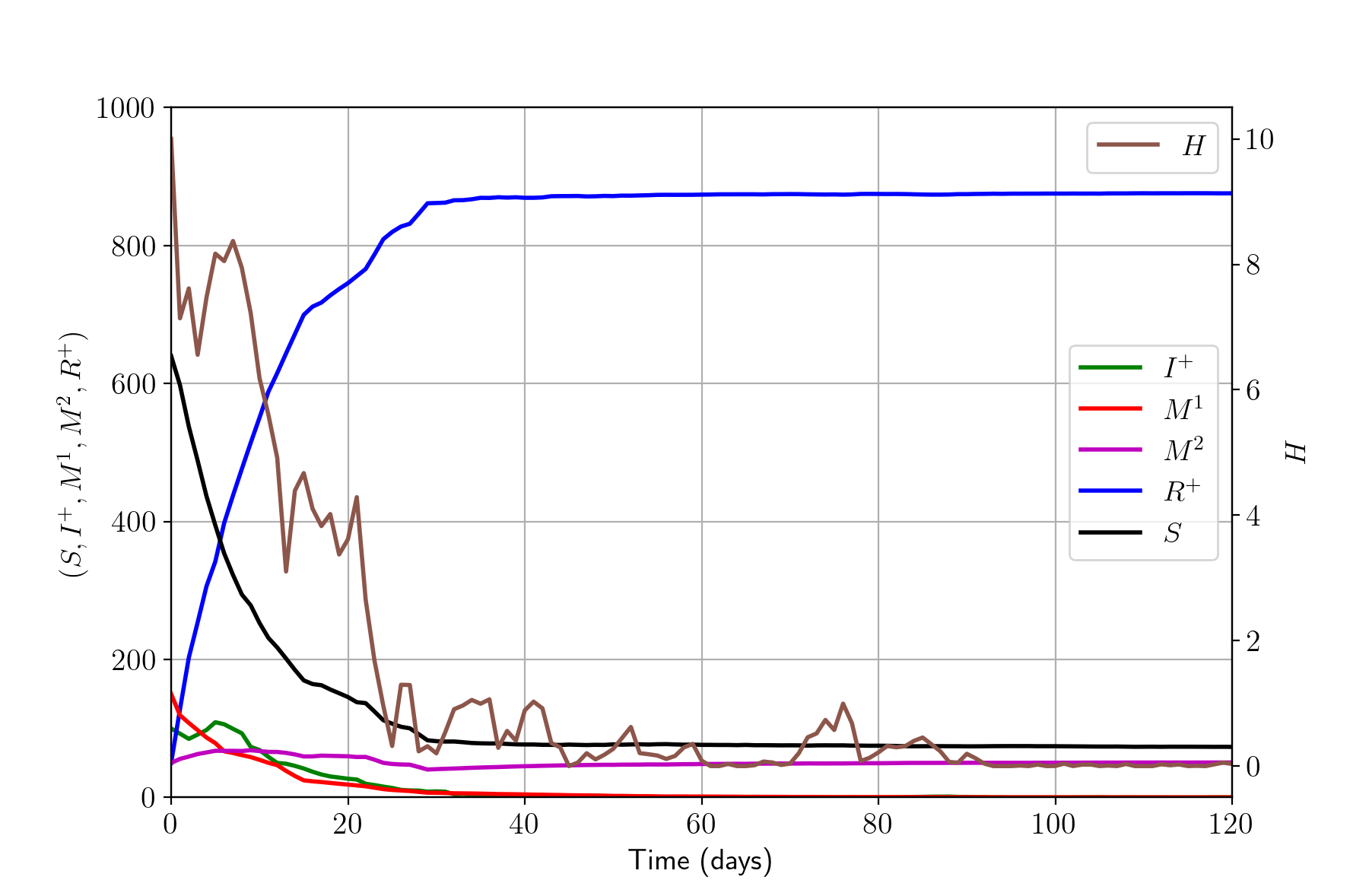}		
			\caption{\textbf{Linear interpolation method}: optimal trajectories of compartments size  ; Initial state  is $\m^1=150,  \m^2=50, \q^1=200,\q^2=200, \rho=0, \z^1=100,  \z^2=50, \z^3=10$; The hospital y-axis is on the right-hand side and the y-axis of the rest is on the left.}
			\label{optpath}
			
		\end{figure}

		\begin{figure}
			\centering
			\includegraphics[width=0.75\linewidth]{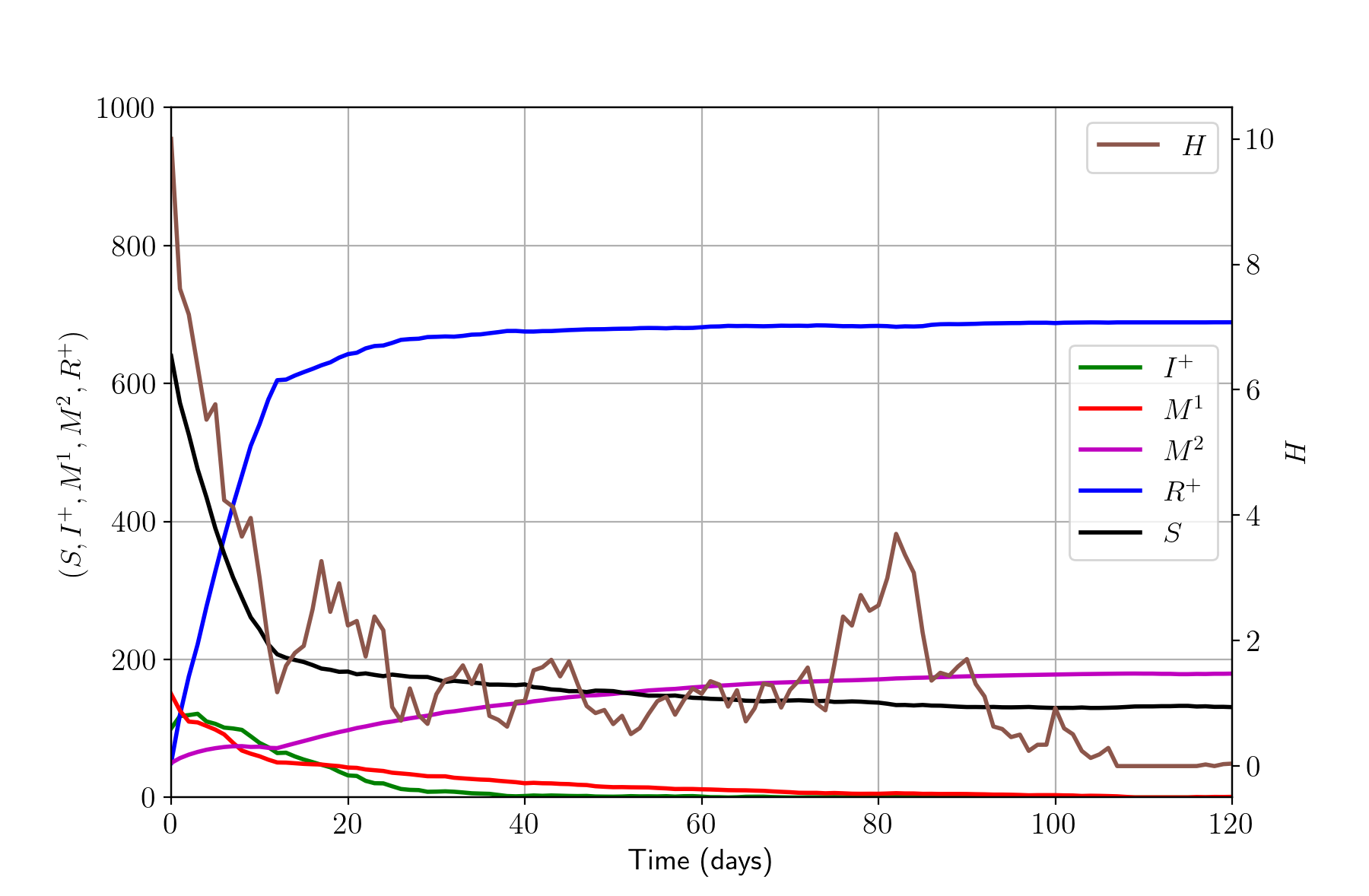}
			
			\caption{\textbf{Regression method}:  Optimal compartment  size  path   ; Initial state:   $(\m^1=150,  \m^2=50, \q^1=200,\q^2=200, \rho=0, \z^1=100,  \z^2=50, \z^3=10$); The hospital y-axis is on the right-hand side and the y-axis of the rest is on the right  }
			\label{optpathfit}
			
		\end{figure}

		\begin{figure}
			\centering
			\includegraphics[width=0.8\linewidth]{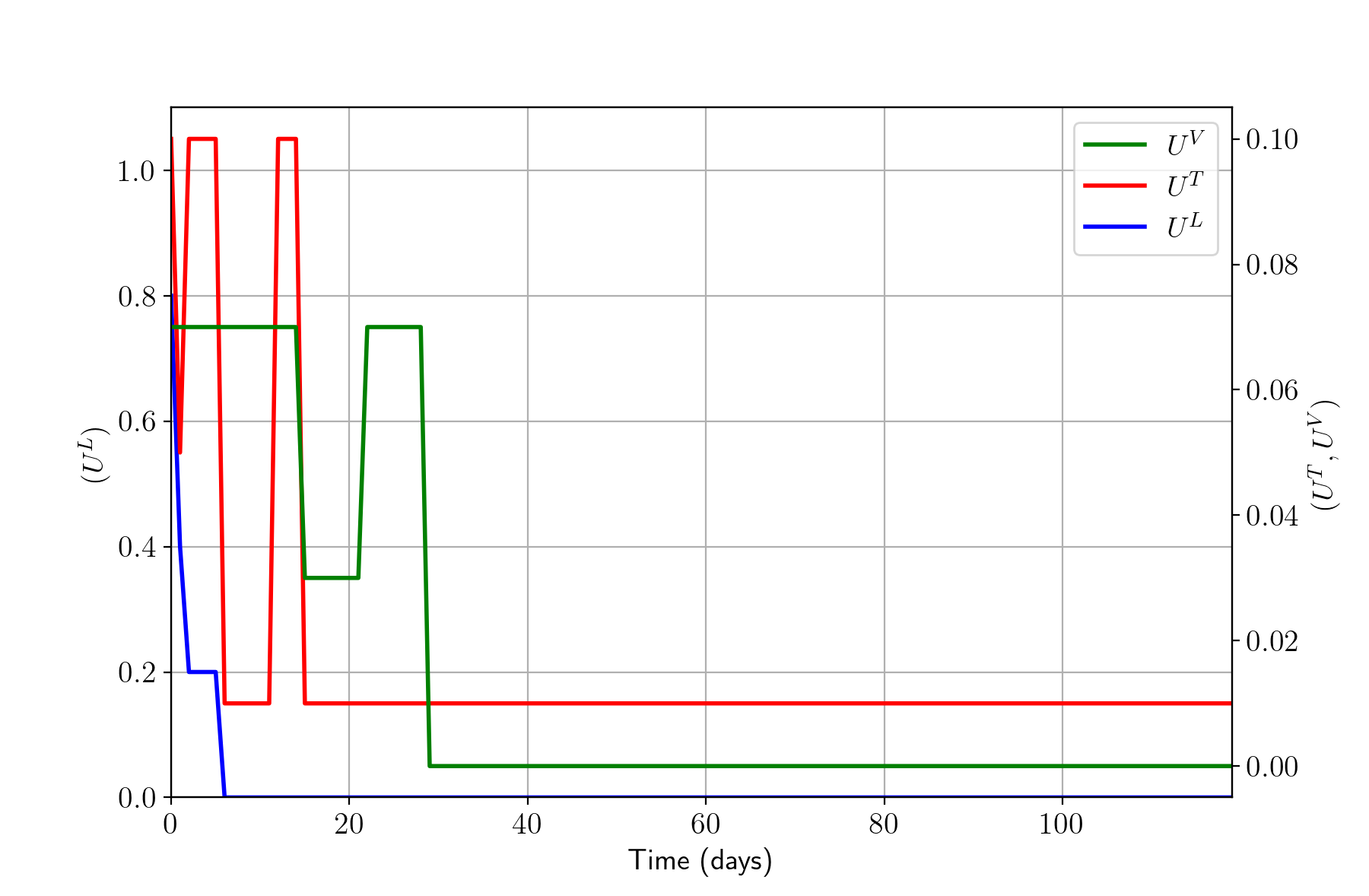}
			
			\caption{\textbf{Linear interpolation method}:  optimal policy ;  Initial state  is $\m^1=150,  \m^2=50, \q^1=200,\q^2=200, \rho=0, \z^1=100,  \z^2=50, \z^3=10$; The $u^L$ y-axis is on the left-hand side and the y-axis of $U^T,U^V$ is on the right. }
			\label{optcontrpath}
			
		\end{figure}
		
		\begin{figure}
			\centering
			\includegraphics[width=0.75\linewidth]{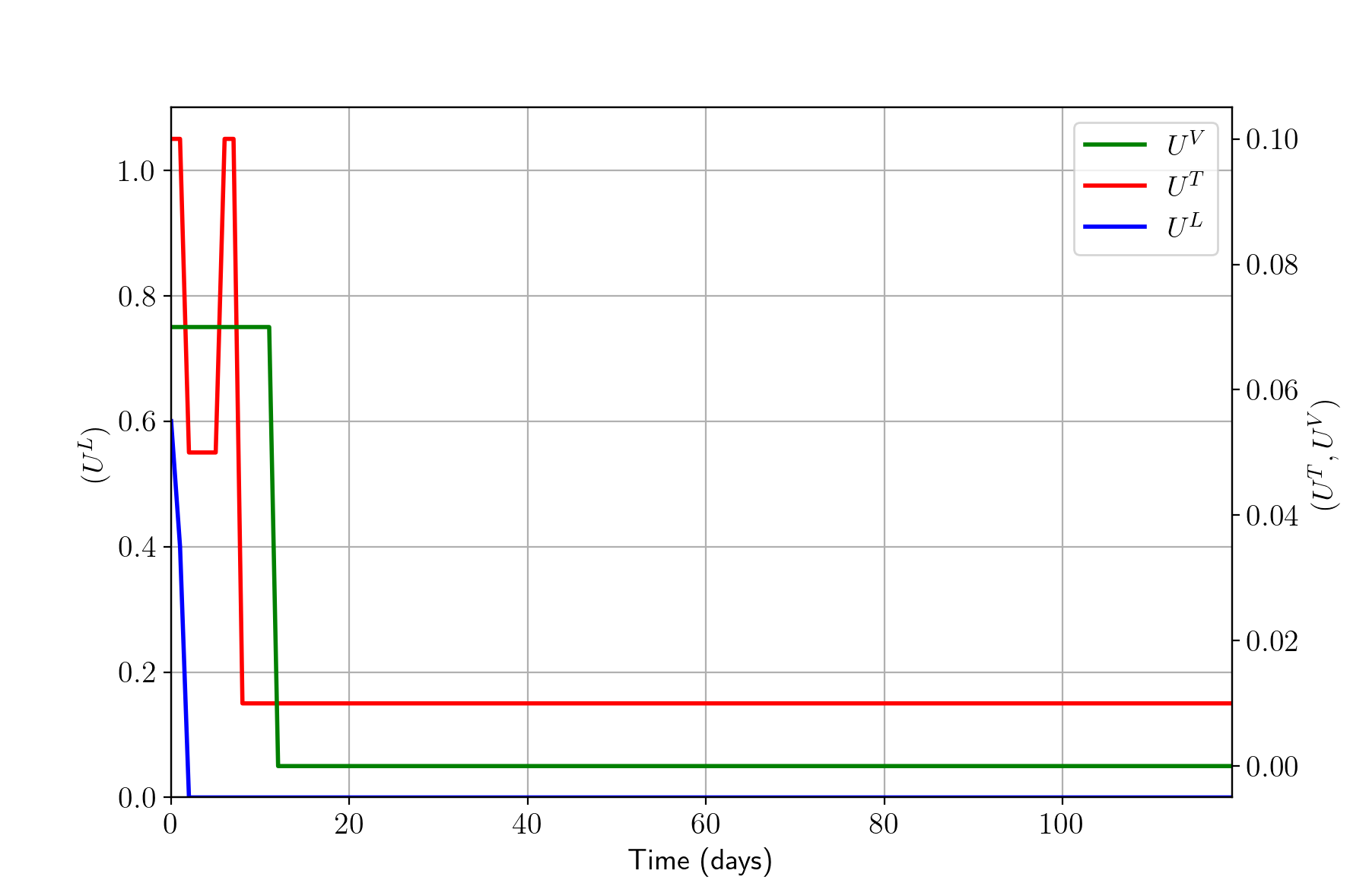}
			
			\caption{\textbf{Regression method}: optimal control path; Initial state:   $(\m^1=150,  \m^2=50, \q^1=200,\q^2=200, \rho=0, \z^1=100,  \z^2=50, \z^3=10$); The $u^L$ y-axis is on the left-hand side and the y-axis of $U^T,U^V$ is on the right.  }
			\label{optcontrpathfit}
			
		\end{figure}

		\subsection{Numerical results using state discretization, quantization and value function regression}
		
		In this section, we present the numerical results obtained through our second approach: value function regression. This method involves selecting a set of ansatz functions for the least squares regression. Leveraging our previous results obtained via linear interpolation, which provided insights into the behavior of the value function with respect to each component of the state process, we consider the following functions defined for $x = (\m^1, \m^2, \q^1, \q^2, \q^{12}, \z^1,\z^2, \z^3)$,
		\begin{itemize}
			\item constant function: ~~~~~~ $\varphi_0 (x)~~=c$;	
			\item Linear functions:	~~~~~~
			\begin{tabular}[t]{lcl@{\hspace{1cm}}lcl}
				$\varphi_1(x)$&$=$&$\z^1$;& $\varphi_4(x)$&$=$&$\m^1$;\\[1ex]	
				$\varphi_2(x)$&$=$&$\z^2$;& $\varphi_5(x)$&$=$&$\m^2$;\\[1ex]	
				$\varphi_3(x)$&$=$&$\z^3$;& $\varphi_6(x)$&$=$&$\q^1$.\\
			\end{tabular}	
			\item Quadratic functions:	
			\begin{tabular}[t]{lcl}
				$\varphi_7(x)$&$=$&$\max( \z^1-\overline{ 	\text{x}}^{I^+},0 )^2$;\\[1ex]	 $\varphi_8(x)$&$=$&$  \max(\m^1-\overline{ 	\text{x}}^{I^-},0)^2$;\\[1ex]	
				$\varphi_9(x)$&$=$&$   \max(N-\z^1-\z^2-\z^3- \overline{ \text{x}}^{\text{Test}},0)^2$ \\
				&$ =$&$ \max (\text{x}^{\text{Test}}-\overline{\text{x}}^{\text{Test}},0)^2$;\\[1ex]	
				$\varphi_{10}(x)$&$=$&$   \max(N-\z^1-\z^2-\z^3- \overline{ \text{x}}^{\text{Vacc}},0)^2$ \\
				&$ =$&$ \max (\text{x}^{\text{Test}}-\overline{\text{x}}^{\text{Vacc}},0)^2 $;\\[1ex]	
				$\varphi_{11}(x)$&$=$& $(N-\z^1)^2~ =~ (\text{x}^{\text{work}})^2$ ;\\[1ex]
				
				$\varphi_{12}(x)$&$=$&$  \max(\z^3-\overline{ 	\text{x}}^{H},0)^2$\\[1ex]					
			\end{tabular}
			\item Function appearing in the infection cost for non-detected infectious:
			
			$\varphi_{13}(x)$$=$ \small{  $\big((\m^1-\overline{ \text{x}}^{I^-})^2 +\q^1 \big) \Phi_{\mathcal{N}}\Big( \dfrac{\m^1-\overline{ \text{x}}^{I^-}}{\sqrt{\q^1}}\Big ) 
				+ (\m^1-\overline{ \text{x}}^{I^-}) \sqrt{\dfrac{\q^1}{2 \pi}}
				\exp \Big(- \frac{1}{2} \dfrac{(\m^1-\overline{ \text{x}}^{I^-})^2}{\q^1}\Big)  ;$}

			\noindent 	where  $\Phi_{\mathcal{N}}$  is the cumulative  distribution function of  the standard normal distribution. 
			
		\end{itemize}
		
		Each of these functions  is either  of  the form of a function appearing  in the running or terminal  cost,  or  follows the shape observed  from  the linear interpolation results. It is important to note that for the value function at terminal time, we know exactly  the structure of the value function.
		
		To ensure a fair   comparison between  results from both methods, we use the same set of parameters and discretization grid points   as in the first approach.

		\subsubsection{Value function and optimal decision rule w.r.t. $\condmean^1, Z^1$}  Similar to the first method we visualize the  value function  and optimal decision rule w.r.t. the two infection compartments. Once again, we consider  scenarios of remaining compartments held fixed to small  and moderate   sizes as previously. Figures \ref{LI118smallfit}, \ref{LI100smallfit}, \ref{LI118modfit}, and \ref{LI100modfit} display the  value function and optimal decision in   configurations identical to those used  in the linear interpolation approach. Overall, the results from both methods appear consistent, especially both optimal decision rules seem to exhibit a generally  similar  structure. However, it should firstly be  noted   that  the optimal decision rule is smoother  for the  approach with value function regression. This is certainly due to the constraint of the value function to belong to the vector space of functions generated by the chosen basis of ansatz functions. Secondly, the optimal decision rules obtained  with regression do not capture well all the nuances observed at the boundaries  in the linear interpolation method results. For instance, comparing the  optimal decision rules in Figures  \ref{LI118mod}, \ref{LI100mod} and  \ref{LI118modfit}, \ref{LI100modfit}, we observe in the result  with linear interpolation that  as both  the number of detected infected  $Z^1$ and  the estimated number of non-detected infected   $\condmean^1$  are  greater than their respective thresholds, the optimal decision rule  suggests  dropping  the lockdown  and vaccination measures  and maintaining only the maximum possible level of  testing. Conversely, the linear regression result  indicates  maintaining all the three control measures at high levels, with  only the lockdown rate that drop slightly at time step $t=100$ days.		
		Thirdly,  the value function obtained through regression seems to underestimate that derived from linear interpolation. This is illustrated in the value functions at time step $t=100$ days (Figures \ref{LI100small}, \ref{LI100mod} and  \ref{LI100smallfit}  \ref{LI100modfit}), where in the case of linear interpolation the maximum values is greater than $100,000$ monetary  unit,  while in the  regression case, the maximum value is around $60,000$ monetary unit.

		\subsubsection{Optimal path  of the epidemic state process}
		The optimal decision rule derived from value function regression can then be used to simulate the optimal dynamics of the epidemic. Figures \ref{optpathfit} and \ref{optcontrpathfit}  depict sample path of compartment sizes  and the optimal control measures to be applied, respectively. These figures show that the number of infected individuals decreases progressively to zero  within $80$  days. The implemented control strategy involves an early strong lockdown lasting only a few days, coupled with vaccination for nearly $15$  days  and less than $10$ days of testing. After $15$ days, all the control measures are lifted  and only the minimal possible testing continues. Compared to the optimal strategy derived from the linear interpolation method, the strategy obtained through regression takes a relatively longer period to reduce the number of infected individuals to zero. This is understandable, as restricting the form of the value function results in a suboptimal strategy. However, this suboptimal strategy effectively keeps the hospital compartment size quite small and below the threshold capacity throughout the infection period.

		\bigskip 
		\begin{footnotesize}
			%

			\smallskip\noindent\textbf{Acknowledgments~}	 
			The authors thank      Olivier Menoukeu Pamen (University of Liverpool), Gerd Wachsmuth,   Armin Fügenschuh, Markus Friedemann, Jesse Beisegel (BTU Cottbus--Senftenberg)	for insightful discussions and valuable suggestions that improved this paper.

			\smallskip\noindent
			\textbf{Funding~} 	
			The  authors gratefully acknowledge the  support by the Deutsche Forschungsgemeinschaft (DFG), award number 458468407,  and by the  German Academic Exchange Service (DAAD), award number 57417894.

			%
			%
			
		\end{footnotesize}

		\appendix
		\section*{Appendix}
		\addcontentsline{toc}{section}{Appendix}

		\section{Proofs}\label{proof}
		
		\begin{proof}[Proposition \ref{costprop}]
			Let us denote by $\mathcal{F}^Z_0 =\mathcal{F}^I_0 v \sigma\{Z_0\}$   the initial  information $\sigma$-algebra generated by the initial guess of $Y_0$   and the initial observation  $Z_0$. We  recall that the conditional distribution of $Y(0)$    follows a  Gaussian law  $\mathcal{N}(m_0, q_0)$.  
			
			Taking  the conditional  expectation of the performance criterion  with  respect to  the  initial information $\mathcal{F}^Z_0$, we obtain

			{\small \begin{eqnarray*}
					\E[\mathcal{J}^F(X_0,u)\vert  \mathcal{F}^Z_0 ]		&=& \E \Big[\sum_{n=0}^{N_t-1} \hspace*{1em} \Psi^F((Y_n, Z_n),u_n)  \hspace*{2em}    +\hspace*{1.5em}  \Phi^F(Y_{N_t}, Z_{N_t}) \Big\vert  \mathcal{F}^Z_0 	\Big] \\[.5ex] 
					&\D =
					&\D \E\Big[\sum_{n=0}^{N_t-1}   \E\Big[ \Psi^F((Y_n, Z_n),u_n) \Big\vert \mathcal{F}_n^Z \Big]     +  \E\Big[\Phi^F(Y_{N_t}, Z_{N_t})  \Big\vert \mathcal{F}_{N_t}^Z \Big] \Big\vert \mathcal{F}^Z_0	\Big] 
			\end{eqnarray*}}
			
			This last equality  is obtained  using the tower low  of  conditional  expectation since $\mathcal{F}_0^Z \subseteq \mathcal{F}_n^Z$ for all $n \geq 0$.  			
			Then,  because $\Psi^F$ and $ \Phi^F$  are considered linear and quadratic in $y$, the above conditional expectation $\E[\ldots\vert \mathcal{F}_n^Z ]$, $n=0,\cdots, N_t$  can be expressed in terms of the  extended Kalman filter $(\condmean,\condvar)$ approximation for the  hidden state $Y$. We  recall that in the case of conditionally Gaussian   sequences treated in \citet{liptser2013statistics}, the  conditional distribution of $Y_n$ is the Gauss distribution $\mathcal{N}(\condmean_n,\condvar_n)$.  When computing the conditional expectation of running and terminal  costs, expression that  are  fully observable,  meaning that they do not depend on  the  hidden state $Y$,  are not modified. In our case, it means that only the running  and terminal  infection  costs  for non-detected infectious are modified (since $X^{ \text{Work}}=N-Z^1$ and $X^{ \text{Test}}=N-(Z^1+Z^2+Z^3)$  allow to obtain fully observable expressions for economic, detection  and vaccination  costs). The running infection costs for non-detected infectious becomes

			\begin{eqnarray*}
				\mathbb{E}\Big [	C^-_I \left(Y^1_n,\overline{ \text{x}}^{I^-}\right) \Big\vert \mathcal{F}_{n}^Z  \Big]  &=&	 \mathbb{E}\Big [	a_{I^-} Y^1_n+ b_{I^-} (Y^1_n-\overline{ \text{x}}^{I^-})_+^2 \Big\vert \mathcal{F}_{n}^Z  \Big]  \\ [0.5ex] 
				&=&		a_{I^-} M^1_n+ b_{I^-}  \mathbb{E}\Big [(Y^1_n-\overline{ \text{x}}^{I^-})_+^2 \Big\vert \mathcal{F}_{n}^Z  \Big].
			\end{eqnarray*}
			
			Further, because the conditional law of $Y^1_n$ given $\mathcal{F}_{n}^Z$ is Gaussian with mean $\condmean^1_n$ and variance  $\condvar^1_n$, we have,  
			
			\begin{eqnarray*}
				\mathbb{E}\Big [	C^-_I \left(Y^1_n,\overline{ \text{x}}^{I^-}\right) \Big\vert \mathcal{F}_{n}^Z  \Big] 
				&=&		a_{I^-} M^1_n+ b_{I^-}  \int_{\overline{ \text{x}}^{I^-}}^{\infty} (y-\overline{ \text{x}}^{I^-})^2  \dfrac{1}{\sqrt{2 \pi \condvar^1_n}} \exp \Big(- \frac{1}{2} \dfrac{(y-\condmean^1_n)^2}{\condvar^1_n}\Big) dy .
			\end{eqnarray*}
			
			Similar expressions are derived  for the terminal infection  costs for non-detected infectious. 
			
			Hence, there exist measurable  functions $\Psi$ and $\Phi$ as given in the proposition, such that

			\begin{equation*}
				\E[\mathcal{J}^F(X_0,u)\vert  \mathcal{F}^Z_0 ]=\D  \E_{m_0, q_0 ,z}\Big[\sum_{n=0}^{N_t-1} \Psi(\condmean_n,\condvar_n, Z_n,u_n)      + \Phi(\condmean_{N_t},\condvar_{N_t}, Z_{N_t})  	\Big]\\
			\end{equation*}	\qed 
		\end{proof}
		\begin{proof}[Lemma \ref{lem_integral} ] \label{prooflem}
			The transformation of this integral is done by two successive changes of variable. First, let $\widetilde{y}=y-\overline{ \text{x}}^{I^-}$, thus 
			the integral becomes \\ $ \displaystyle \int_{0}^{\infty} \widetilde{y}^2   \dfrac{1}{\sqrt{2 \pi q^1}} \exp \Big(- \frac{1}{2} \dfrac{(\widetilde{y}+\overline{ \text{x}}^{I^-}-m^1)^2}{q^1}\Big) d\widetilde{y} $.   Second, let $\widehat{m}=m^1-\overline{ \text{x}}^{I^-}$  and $\widehat{y}= \dfrac{\widetilde{y}- \widehat{m}}{\sqrt{q^1}}$, we deduce that the integral reads
			
			$ \displaystyle \int_{-\dfrac{\widehat{m}}{\sqrt{q^1}}}^{\infty} (\widehat{m}+\sqrt{q^1}\widehat{y})^2   \dfrac{1}{\sqrt{2 \pi }} \exp \Big(- \frac{1}{2} \widehat{y}^2\Big) d\widehat{y} $.
			
			After  expansion of the quadratic term in the integral  and identification of the cumulative distribution function of the standard normal distribution\\ $\Phi_{\mathcal{N}}(t)=\displaystyle \int_{-\infty} ^t \dfrac{1}{\sqrt{2 \pi}} \exp(-\dfrac{1}{2} \widehat{y}^2) d\widehat{y}$,   we obtain the result\\
			$\big((m^1-\overline{ \text{x}}^{I^-})^2 +q^1 \big) \Phi_{\mathcal{N}}\Big( \dfrac{m^1-\overline{ \text{x}}^{I^-}}{\sqrt{q^1}}\Big ) 
			+ (m^1-\overline{ \text{x}}^{I^-}) \sqrt{\dfrac{q^1}{2 \pi}}
			\exp \Big(- \frac{1}{2} \dfrac{(m^1-\overline{ \text{x}}^{I^-})^2}{q^1}\Big).$
			
		\end{proof}
		
		\section{Transition operator functions} \label{transfunct}
		Expressions of functions involved in the transition operators are:
		\begin{align*}
			f_M(n,m,q,z, \nu) =&  m+f(n,m, z) , \\
			f_Q(n,m,q,z, \nu) =&   -\left[ g\ell^\top+
			f_1 q h_1^\top \right]  \left[\ell\ell^\top + h_1 q h_1^\top\right]^{+}  \left[ g\ell^\top+ f_1 q h_1^\top  \right]^\top 
			+f_1 q f_1^\top+ \sigma\sigma^\top , \\
			f_Z(n,m,q,z, \nu) =&  h_0+h_1 m ,\\
			g_M(n,m,q,z, \nu) =& \left[ g\ell^\top+f_1 q h_1^\top \right] \Big(\left[\ell\ell^\top + h_1 q h_1^\top\right]^{+}\Big)^{\frac{1}{2}},\\
			g_Z(n,m,q,z, \nu) =& \left[\ell\ell^\top + h_1 q h_1^\top \right]^{1/2}.   
		\end{align*}

		\section{Additional numerical  results}
		\subsection{Value function obtained with interpolation method  w.r.t. each component of the state process} 
		To better understand how the value function evolves with respect to each component of the state process, we freeze all variables except one and represent the value function in terms of that single variable. This approach is particularly helpful in determining the form of the ansatz function required for the second numerical approach, which involves fitting the value function. Therefore, the other components are fixed at either small or moderate values while the value function is plotted against the component of interest.
		Figures \ref{VFvsZY1}( Panels \subref{z1pa},\subref{y1pa},\subref{z3pa}),  \ref{VFvsZY1fit} ( Panels \subref{z1pb},\subref{y1pb},\subref{z3pb}),
		\ref{VFvsZY2}( Panels \subref{z2pa},\subref{y2pc},\subref{q1pa}),  and \ref{VFvsZY2fit} ( Panels \subref{z2pb},\subref{y2pd},\subref{q1pb})  illustrate the shape of the value function depending on these components. We observe that the value function has a linear plus quadratic shape with respect to components $Z^1$, $Z^3$, and $\condmean^1$. Additionally, it exhibits a decreasing shape as a function of $Z^2 = R^+$ and $\condmean^2$. Finally, the value function appears as a linearly increasing function with respect to the first diagonal entry of the conditional covariance $Q^1$. This  can be explained by the fact that as $Q^1$ increases, there is more uncertainty in the estimate of $\condmean^1$, suggesting that higher values of $\condmean^1$ are more likely and come with higher costs. 
		
		\newpage
		\begin{figure}
			\centering	
			\hspace{-1.2cm}\textbf{Interpolation} \hspace{5cm}  \textbf{Regression} \\		
			\begin{subfigure}{.5\textwidth}
				\centering			  
				\includegraphics[width=1.1\linewidth]{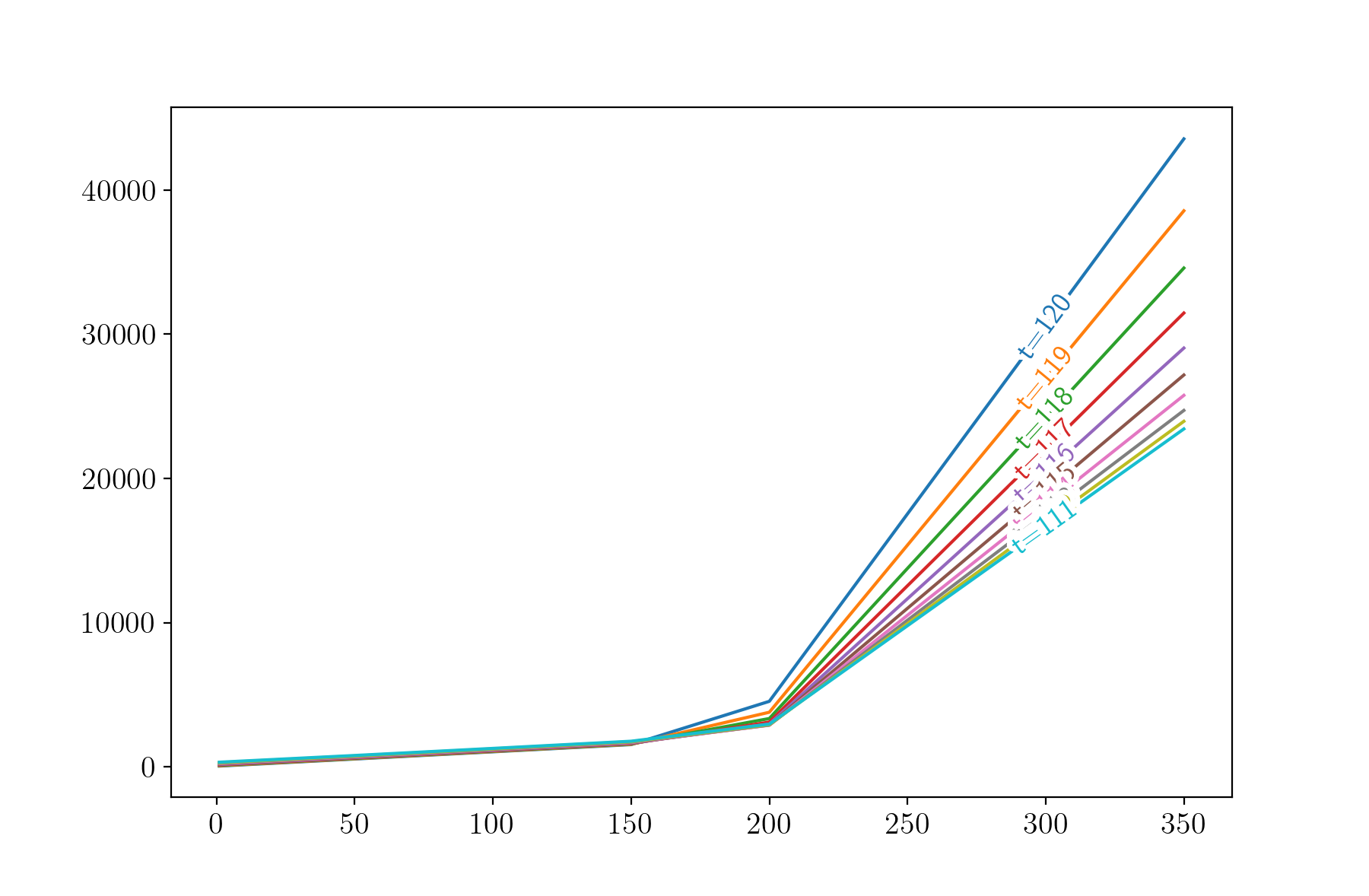}
				\caption{ Function of $\z^1$, other variables small }
				\label{z1pa}
			\end{subfigure}\begin{subfigure}{.5\textwidth}
				\centering%
				\includegraphics[width=1.1\linewidth]{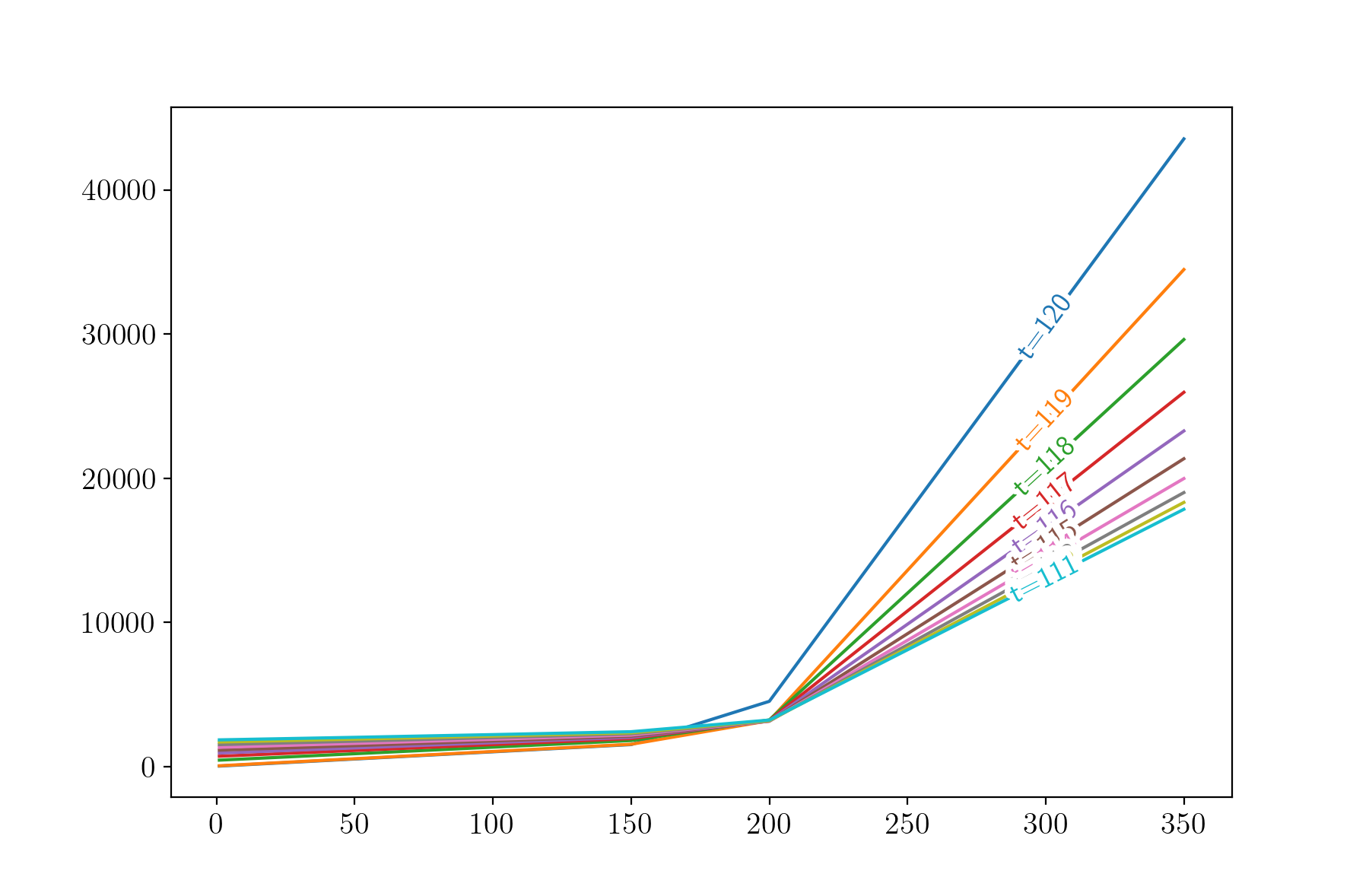}
				\caption{ Function of $\z^1$, other variables small }
				\label{z1pafit}
			\end{subfigure}\\
			\begin{subfigure}{.5\textwidth}
				\centering
				\includegraphics[width=1.1\linewidth]{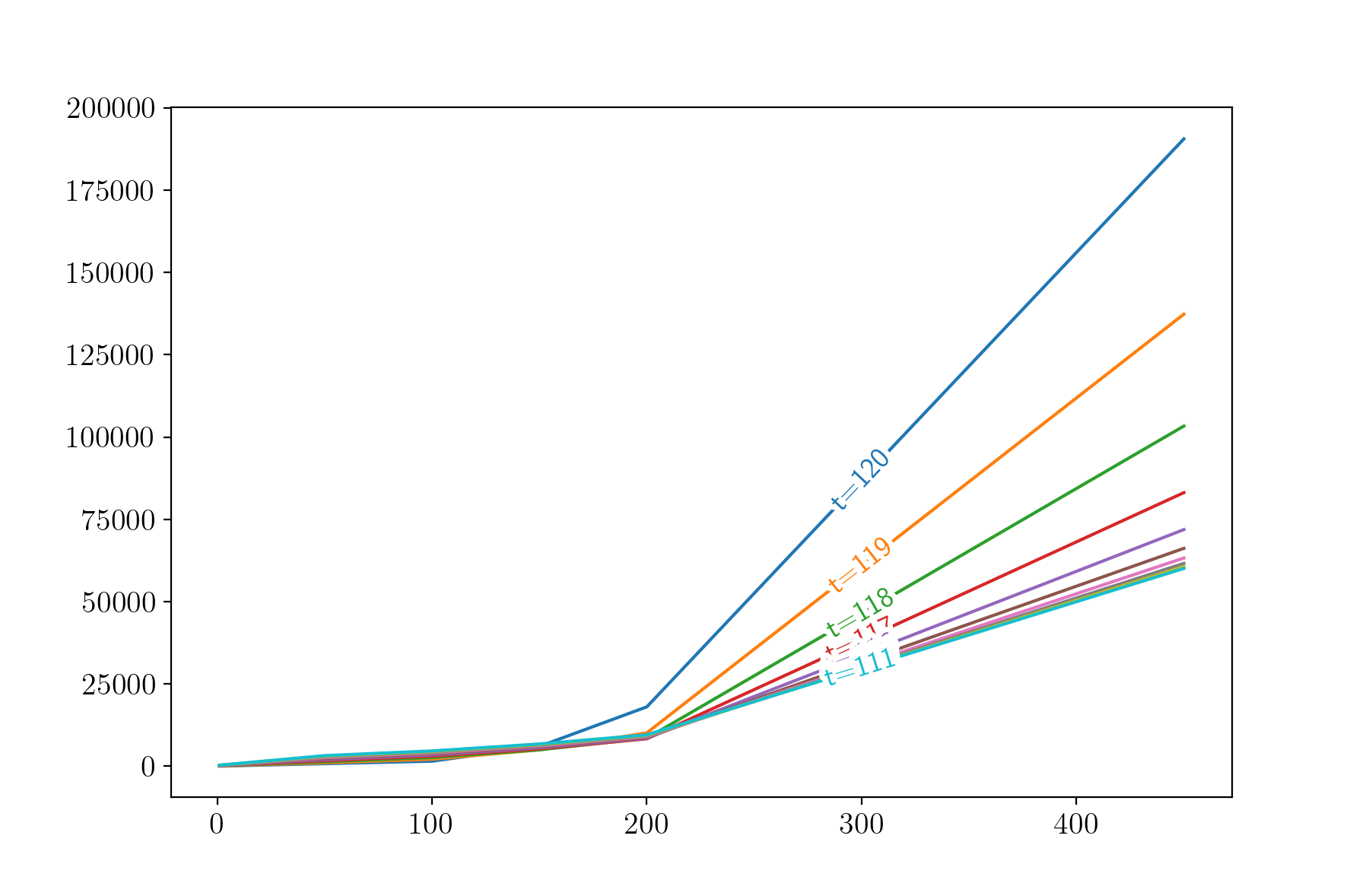}
				\caption{Function of $\m^1$, other variables small}
				\label{y1pa}
			\end{subfigure}\begin{subfigure}{.5\textwidth}
				\centering
				\includegraphics[width=1.1\linewidth]{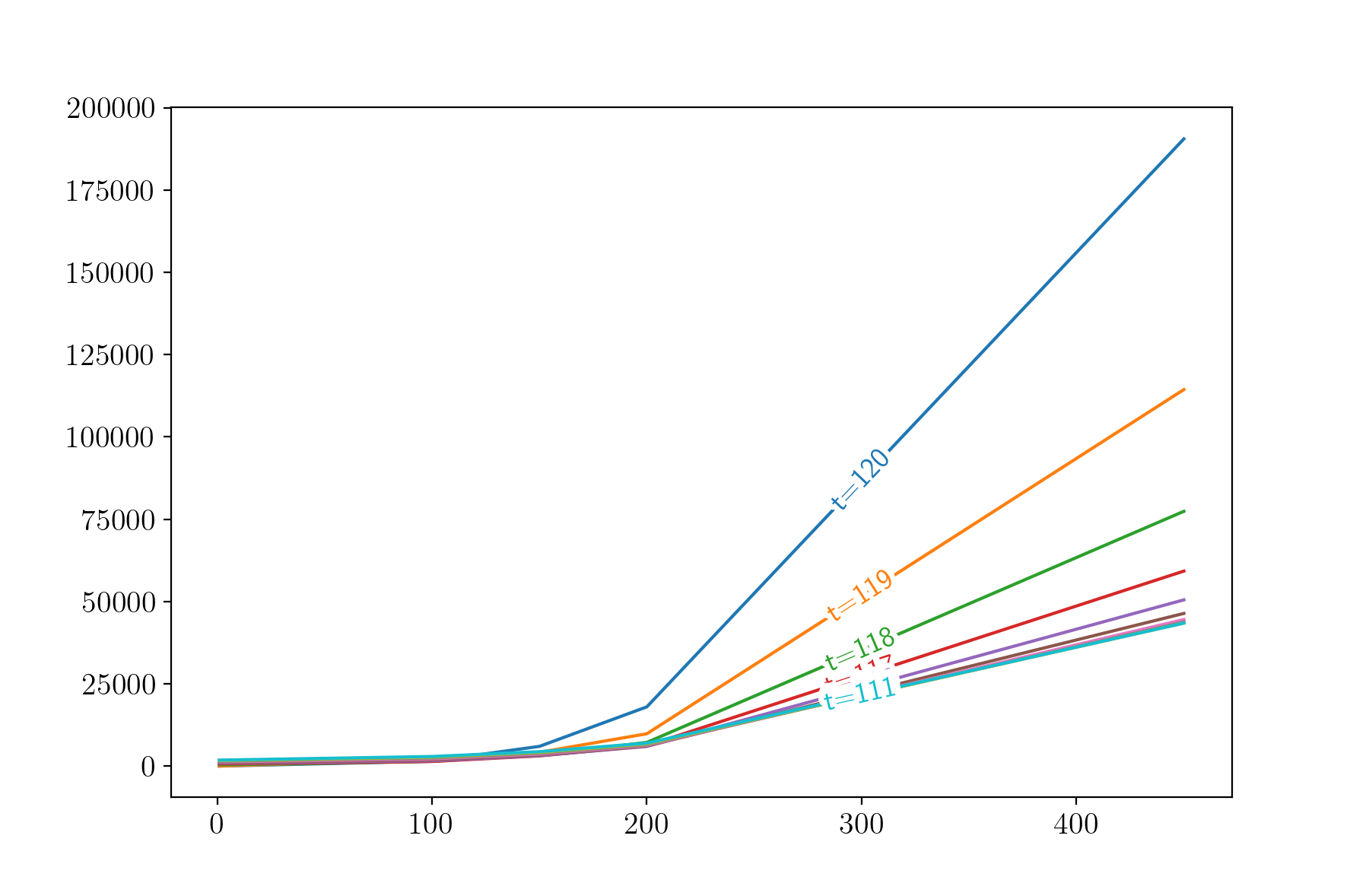}
				\caption{Function of $\m^1$, other variables small}
				\label{y1pafit}
			\end{subfigure}\\
			\begin{subfigure}{.5\textwidth}
				\centering
				\includegraphics[width=1.1\linewidth]{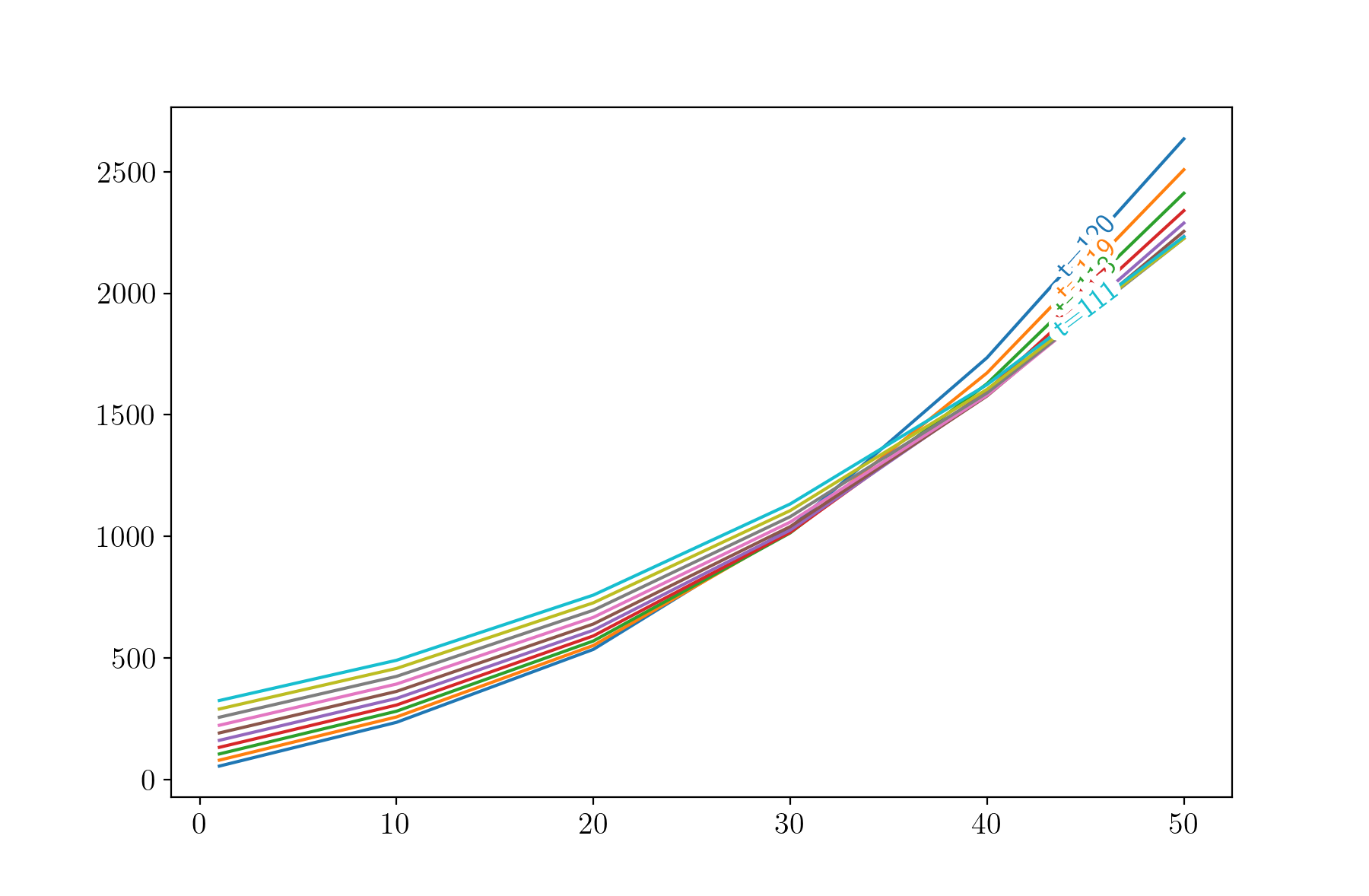}
				\caption{Function of $\z^3$, other variables small }
				\label{z3pa}
			\end{subfigure}\begin{subfigure}{.5\textwidth}
				\centering
				\includegraphics[width=1.1\linewidth]{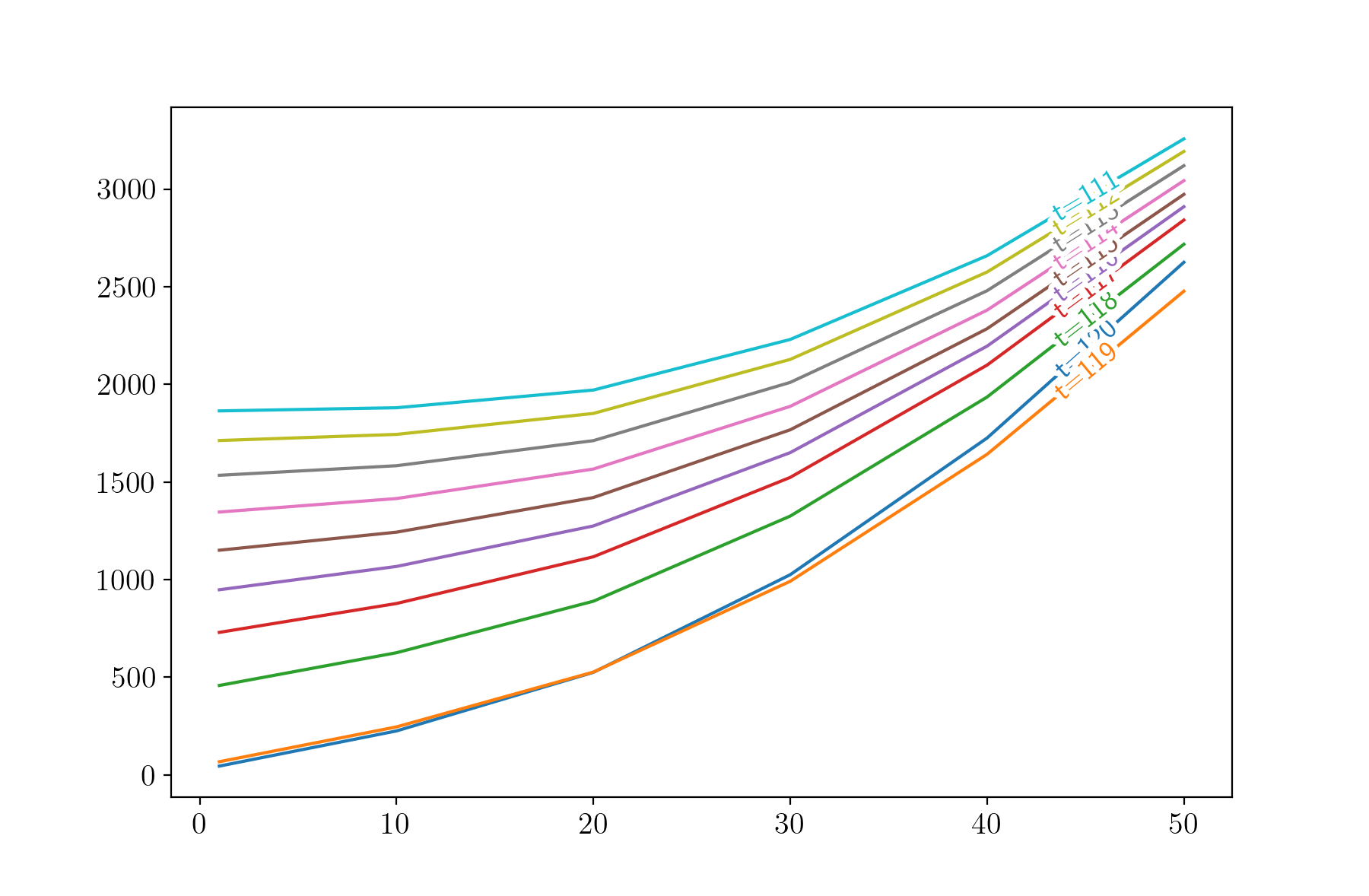}
				\caption{Function of $\z^3$, other variables small }
				\label{z3pafit}
			\end{subfigure}			\caption{Value function  from linear interpolation approach (left) and regression approach\\ (right) as function of $\z^1$, $\m^1$ and $\z^3$ }
			\label{VFvsZY1}
		\end{figure}

		
		\begin{figure}
			\centering	\hspace{-1.2cm}\textbf{Interpolation} \hspace{5cm}  \textbf{Regression} \\
			\begin{subfigure}{.5\textwidth}
				\centering
				\includegraphics[width=1.1\linewidth]{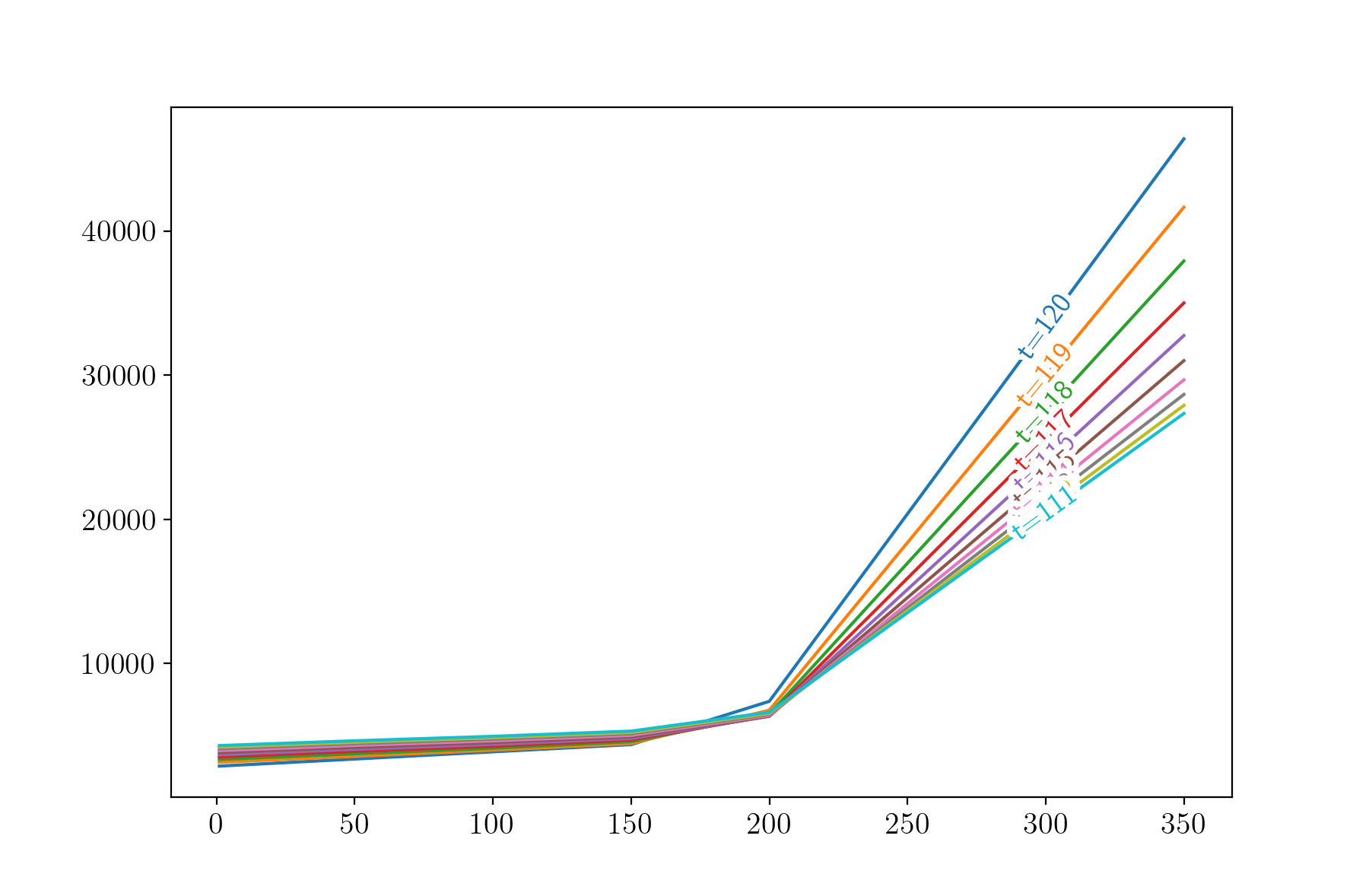}
				\caption{Function of $\z^1$, other variables moderate}
				\label{z1pb}
			\end{subfigure}\begin{subfigure}{.5\textwidth}
				\centering
				\includegraphics[width=1.1\linewidth]{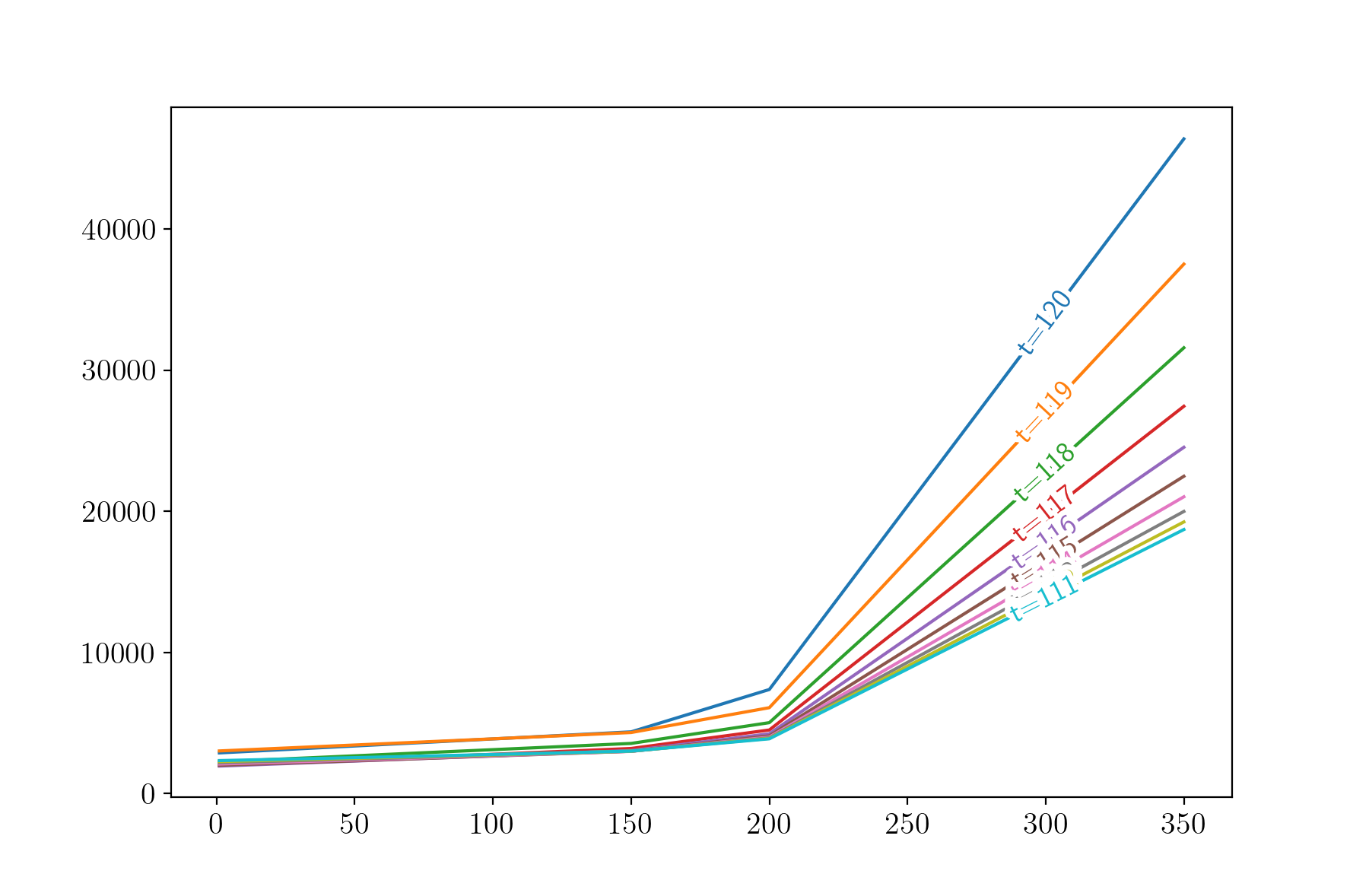}
				\caption{Function of $\z^1$, other variables moderate}
				\label{z1pbfit}
			\end{subfigure}\\
			\begin{subfigure}{.5\textwidth}
				\centering
				\includegraphics[width=1.1\linewidth]{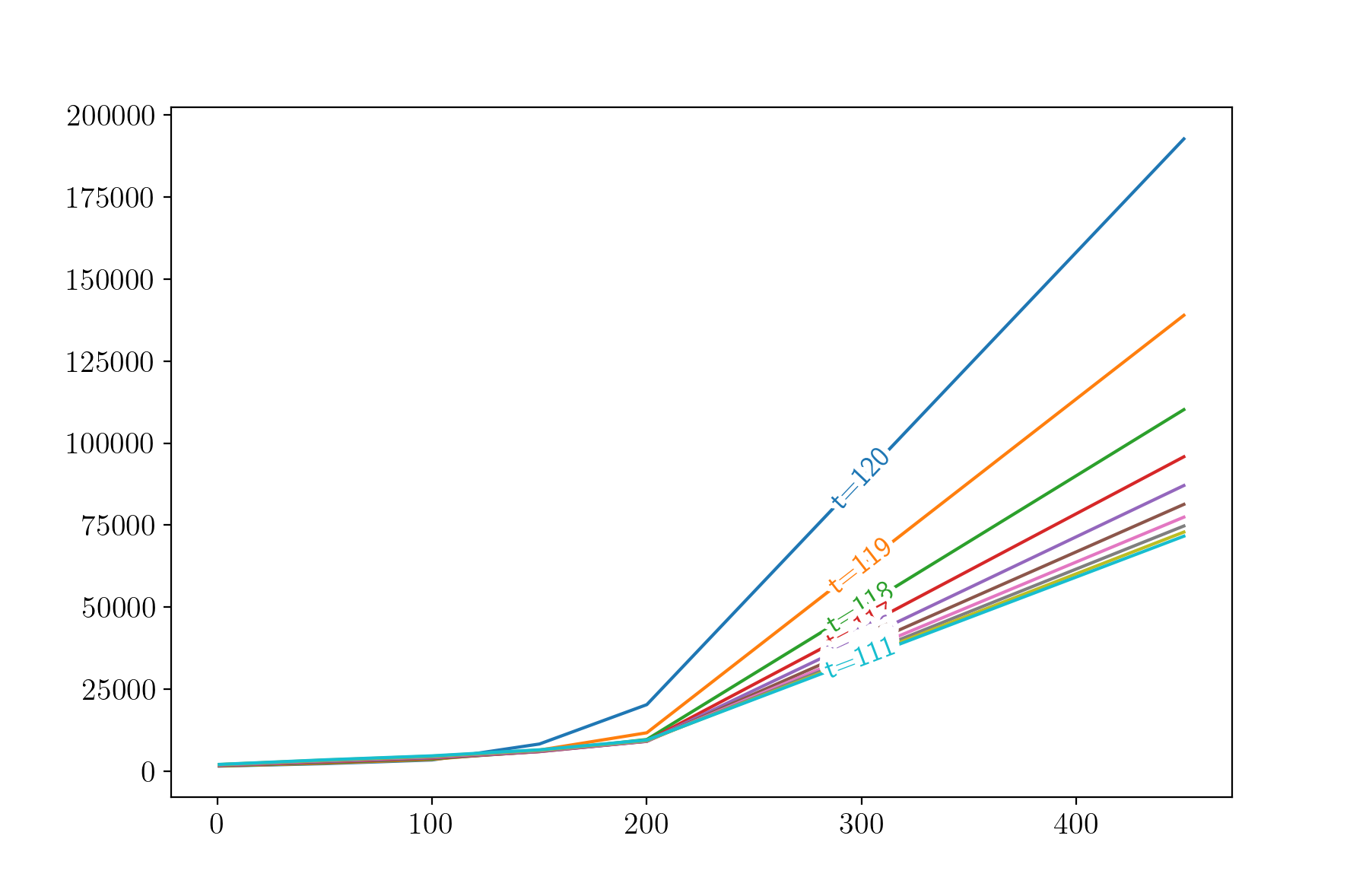}
				\caption{Function of $\m^1$, other variables moderate}
				\label{y1pb}
			\end{subfigure}\begin{subfigure}{.5\textwidth}
				\centering
				\includegraphics[width=1.1\linewidth]{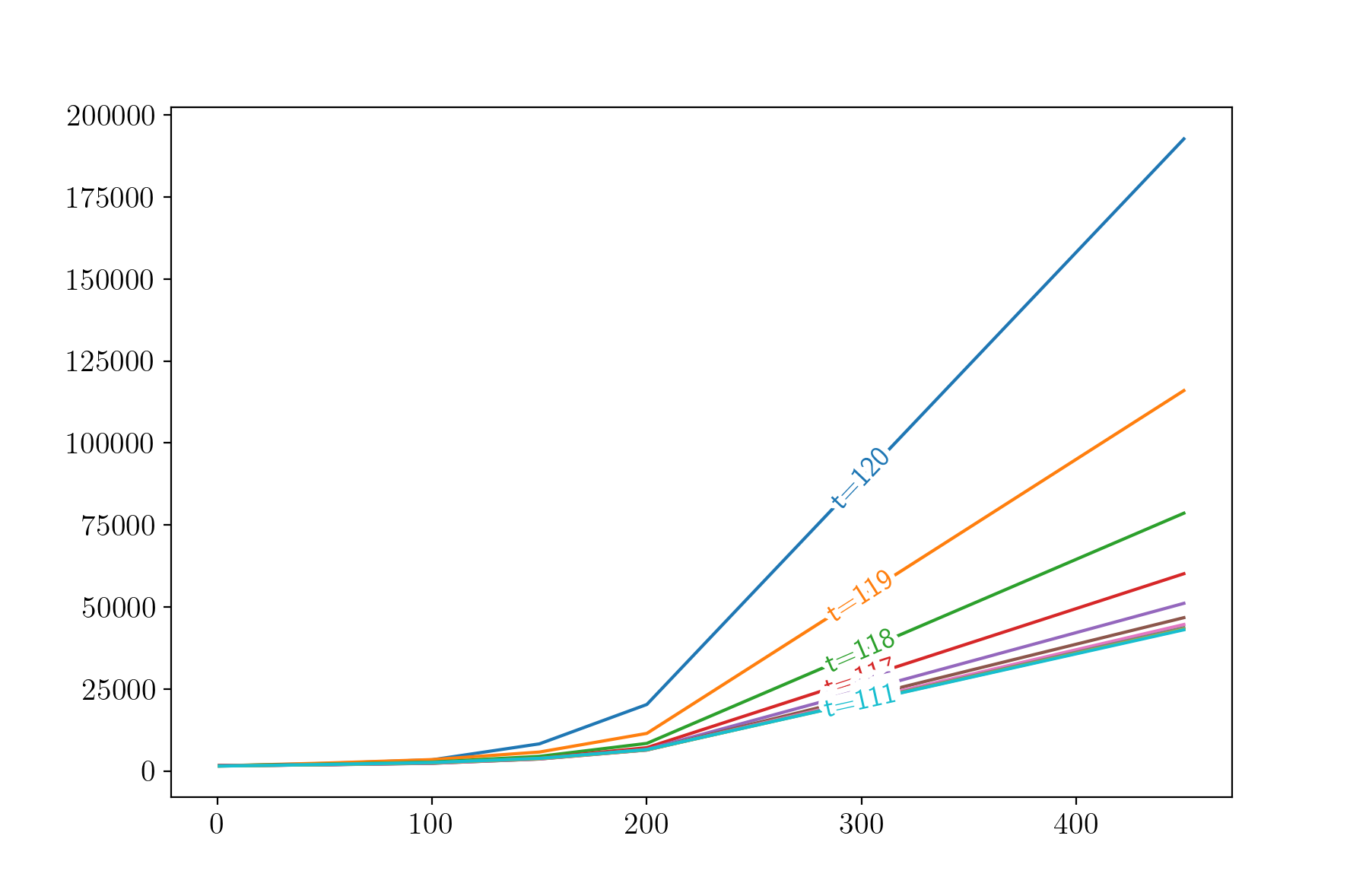}
				\caption{Function of $\m^1$, other variables moderate}
				\label{y1pbfit}
			\end{subfigure}\\
			\begin{subfigure}{.5\textwidth}
				\centering
				\includegraphics[width=1.1\linewidth]{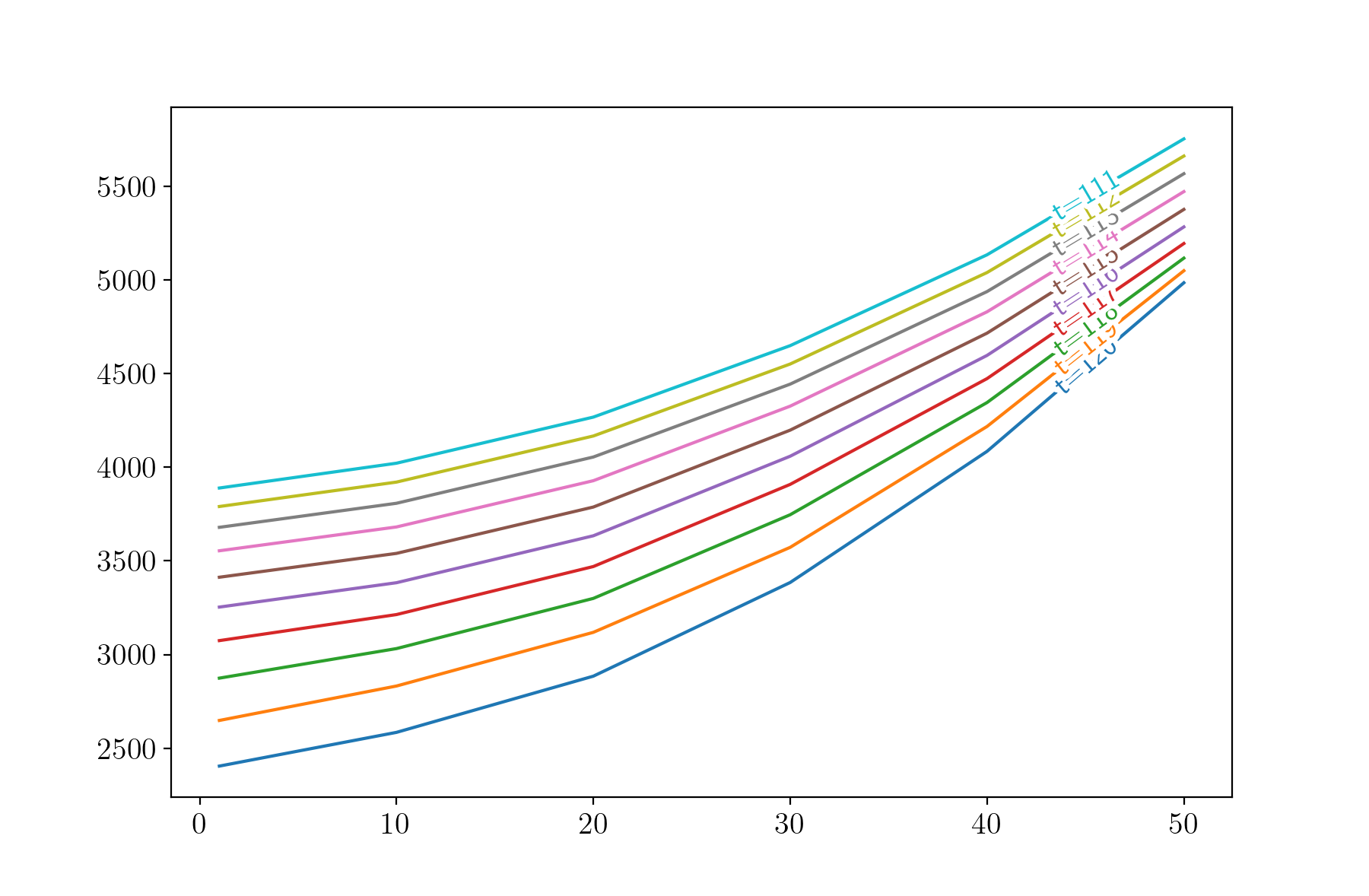}
				\caption{Function of $\z^3$, other variables moderate}
				\label{z3pb}
			\end{subfigure}\begin{subfigure}{.5\textwidth}
				\centering
				\includegraphics[width=1.1\linewidth]{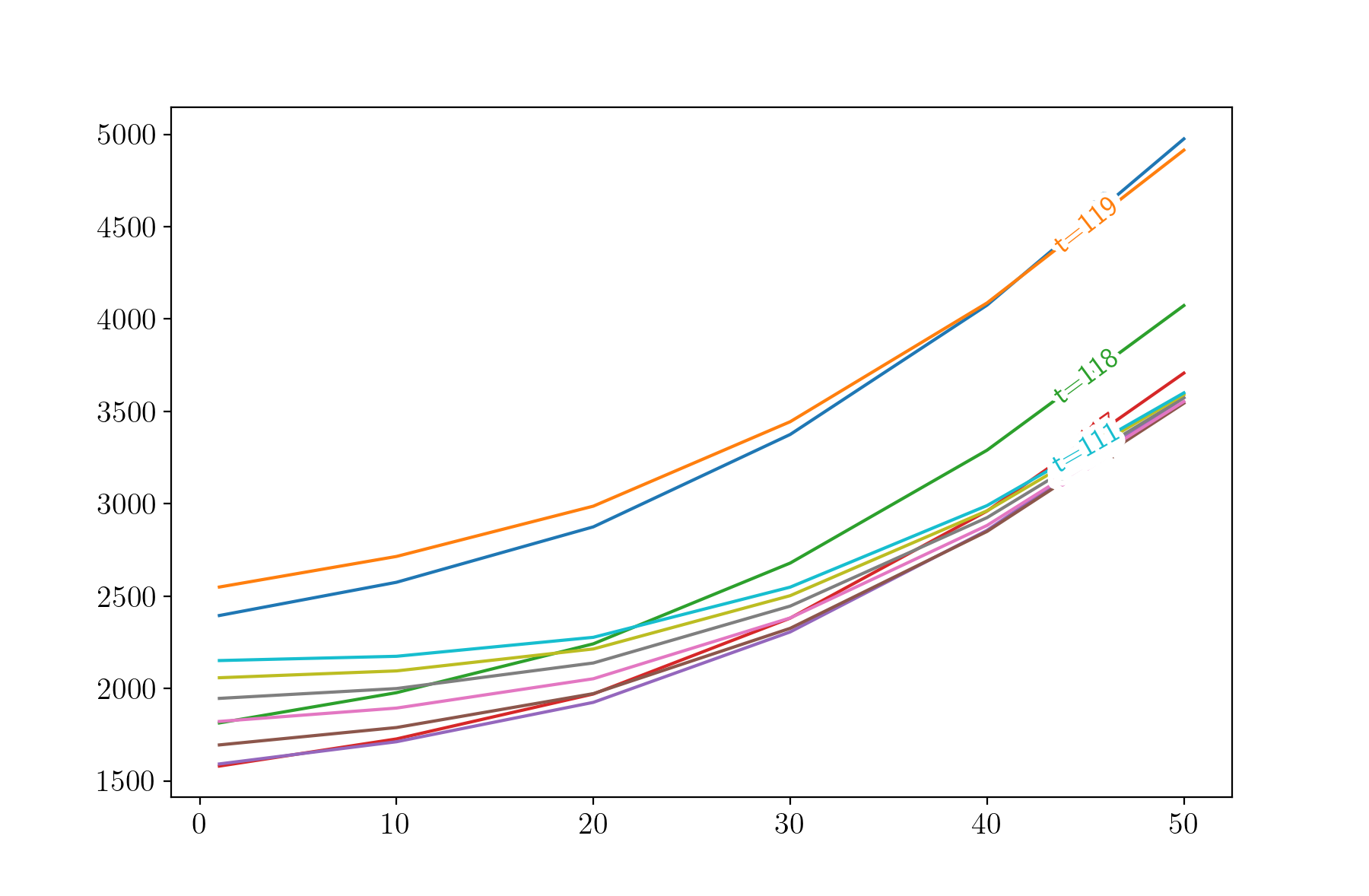}
				\caption{Function of $\z^3$, other variables moderate}
				\label{z3pbfit}
			\end{subfigure}			
			\caption{Value function from linear interpolation approach (left) and regression approach\\ (right)  as function of $\z^1$, $\m^1$ and $\z^3$ }
			\label{VFvsZY1fit}
		\end{figure}

		
		\begin{figure}
			\centering
			\hspace{-1.2cm}\textbf{Interpolation} \hspace{5cm}  \textbf{Regression} \\
			\begin{subfigure}{.5\textwidth}
				\centering
				\includegraphics[width=1.1\linewidth]{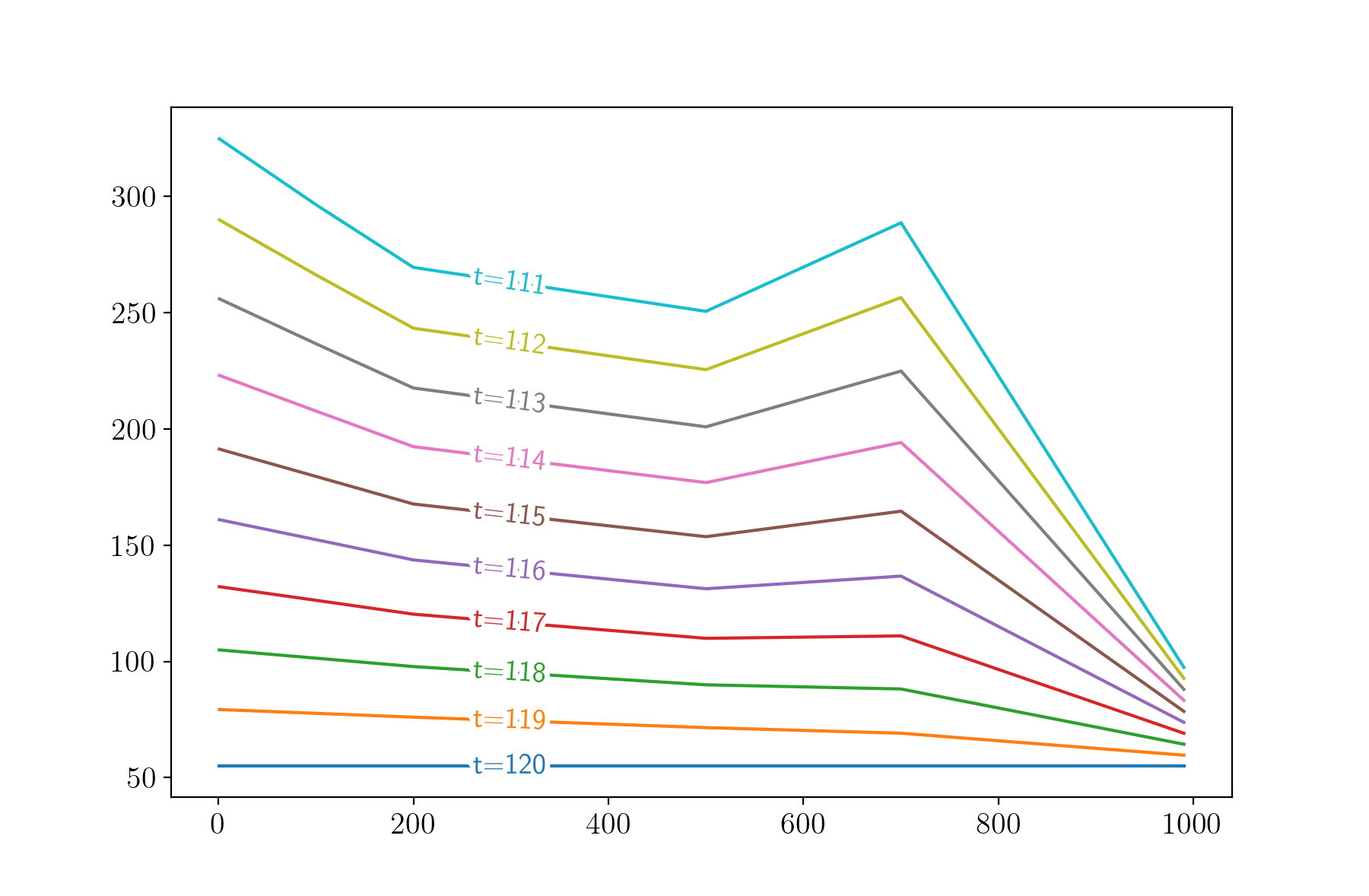}
				\caption{Function of $\z^2$,  other variables small }
				\label{z2pa}
			\end{subfigure}\begin{subfigure}{.5\textwidth}
				\centering
				\includegraphics[width=1.1\linewidth]{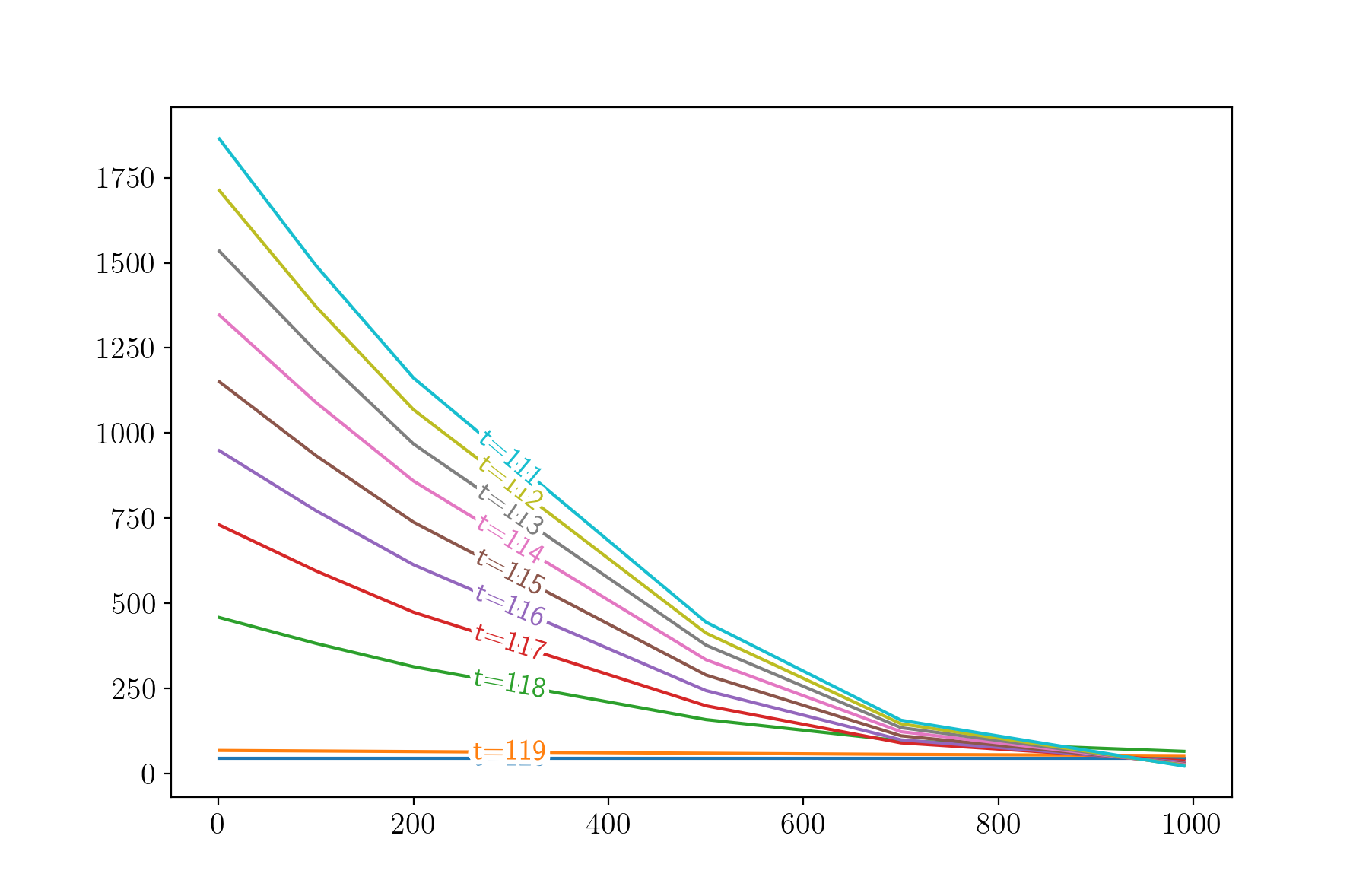}
				\caption{Function of $\z^2$,  other variables small }
				\label{z2pafit}
			\end{subfigure}\\			
			\begin{subfigure}{.5\textwidth}
				\centering
				\includegraphics[width=1.1\linewidth]{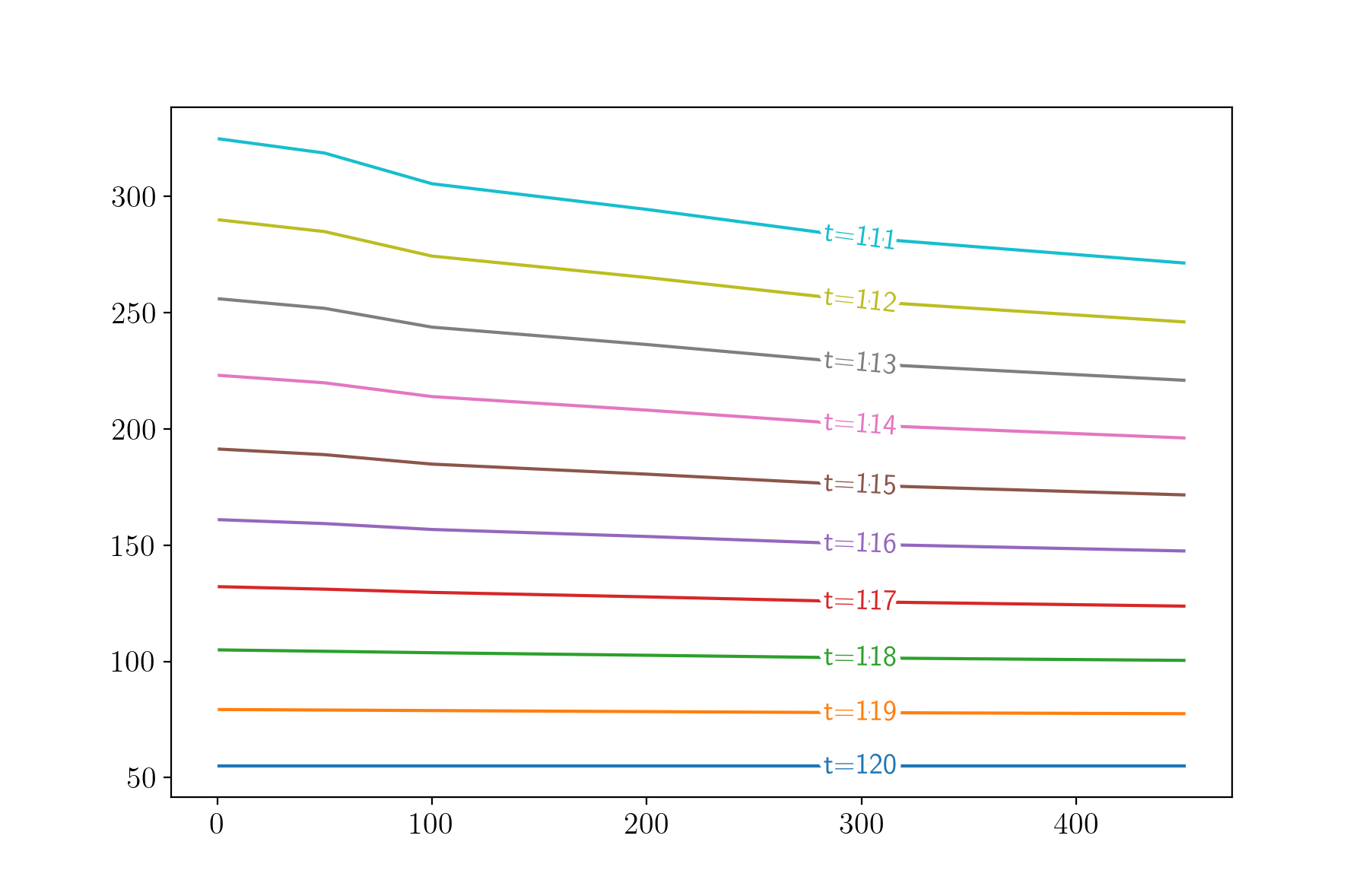}
				\caption{ Function of $\m ^2$, other variables small}
				\label{y2pc}
			\end{subfigure}\begin{subfigure}{.5\textwidth}
				\centering
				\includegraphics[width=1.1\linewidth]{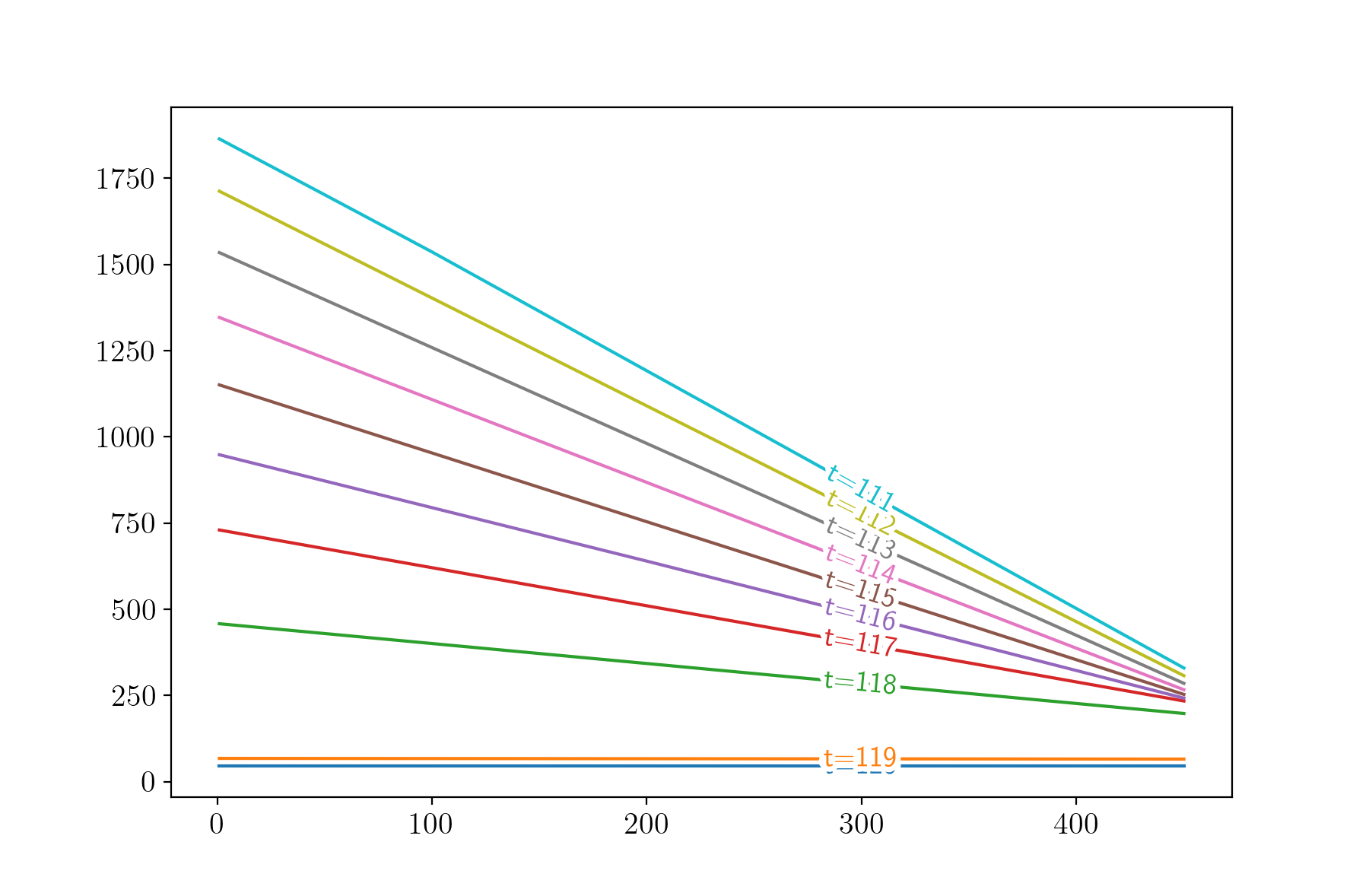}
				\caption{ Function of $\m ^2$, other variables small}
				\label{y2pcfit}
			\end{subfigure}\\			
			\begin{subfigure}{.5\textwidth}
				\centering
				\includegraphics[width=1.1\linewidth]{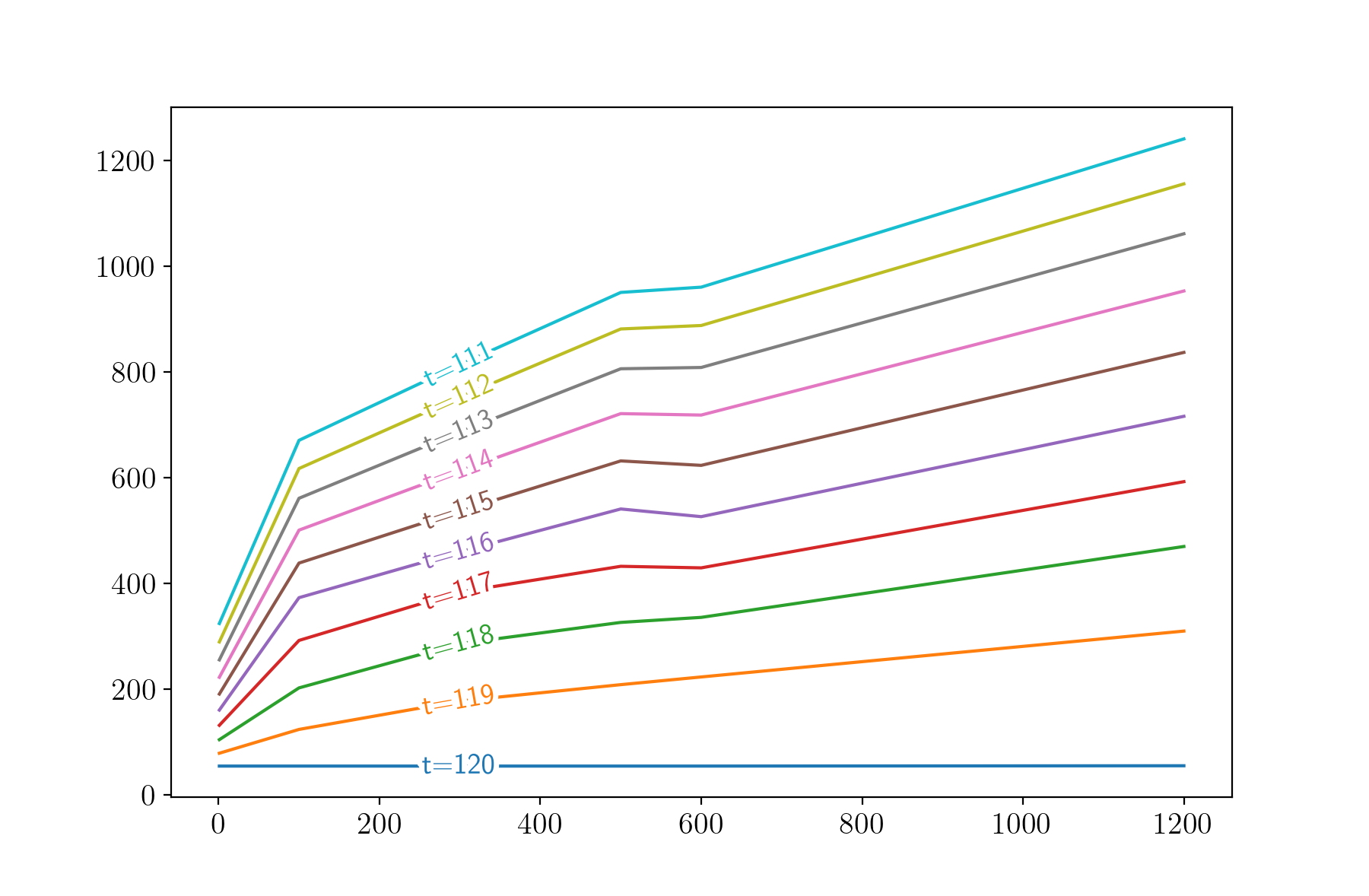}
				\caption{ Function of $\q ^1$, other variables small}
				\label{q1pa}
			\end{subfigure}\begin{subfigure}{.5\textwidth}
				\centering
				\includegraphics[width=1.1\linewidth]{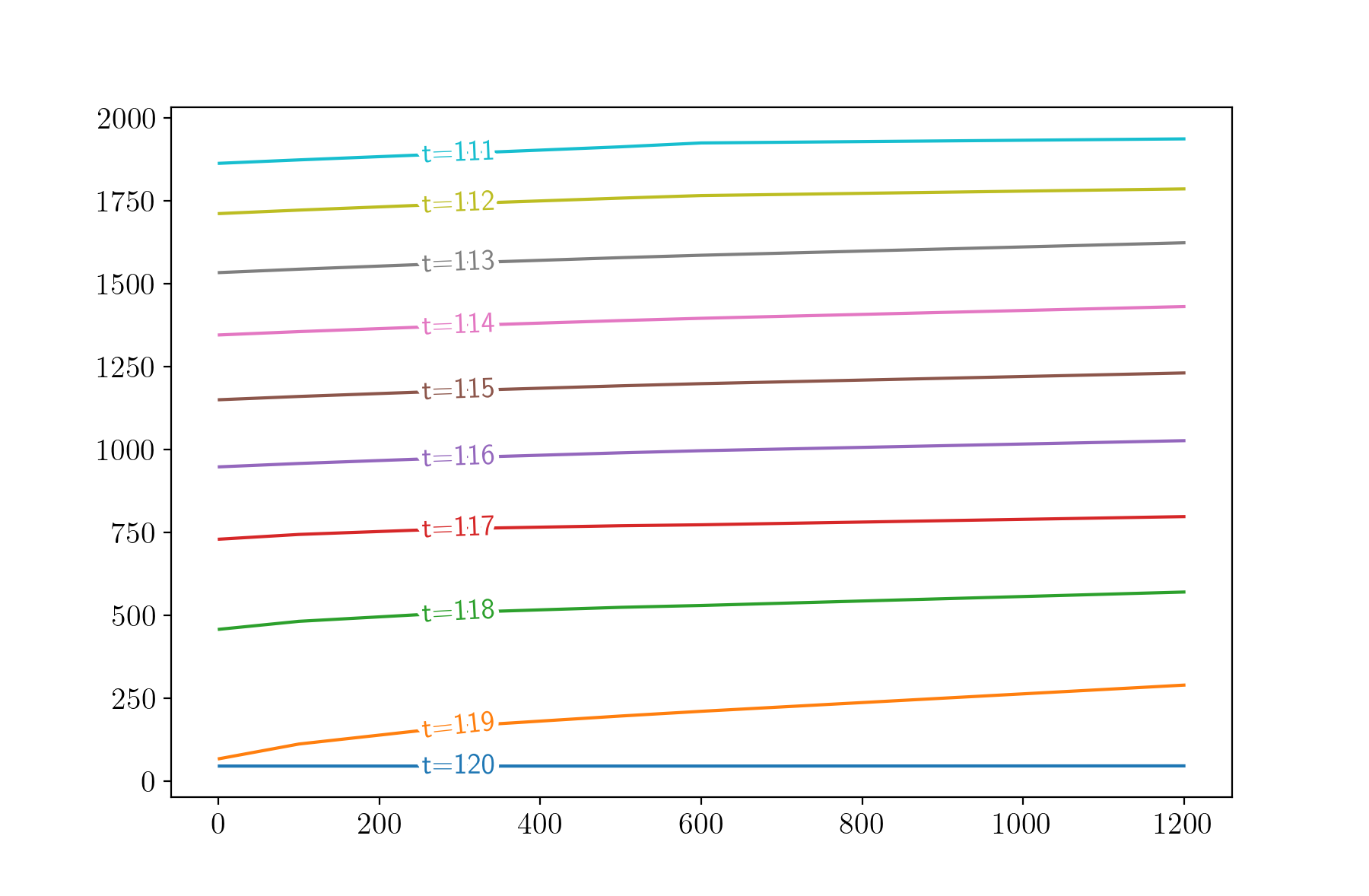}
				\caption{ Function of $\q ^1$, other variables small}
				\label{q1pafit}
			\end{subfigure}%
			\caption{Value function from linear interpolation approach (left) and regression approach (right)   as function of $\z^2$,$\m^2$ and $\q^1$ }
			\label{VFvsZY2}
		\end{figure}

		\begin{figure}
			\centering
			\hspace{-1.2cm}\textbf{Interpolation} \hspace{5cm}  \textbf{Regression} \\
			\begin{subfigure}{.5\textwidth}
				\centering
				\includegraphics[width=1.1\linewidth]{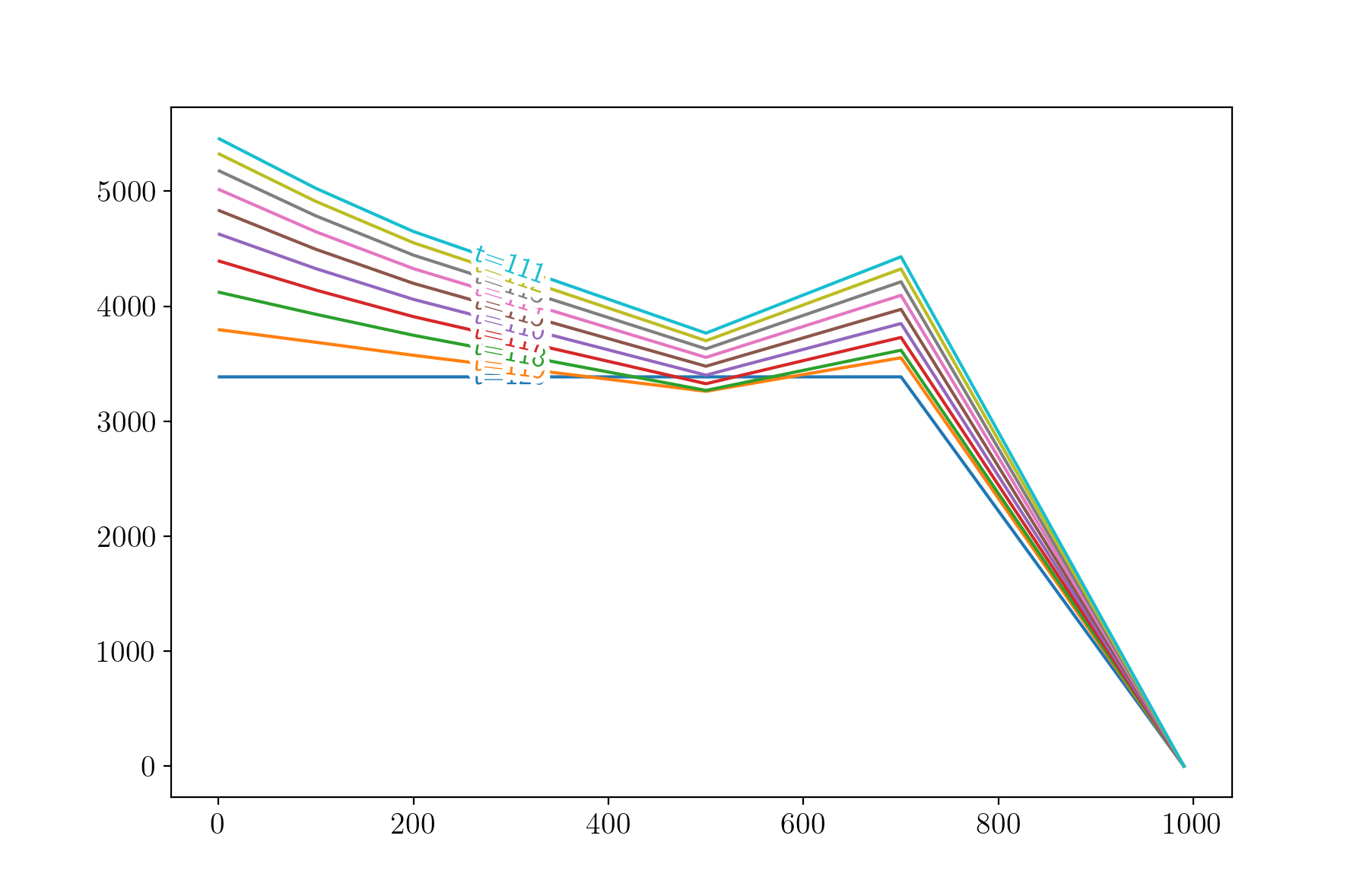}
				\caption{Function of $\z^2$, other variables moderate}
				\label{z2pb}
			\end{subfigure}\begin{subfigure}{.5\textwidth}
				\centering
				\includegraphics[width=1.1\linewidth]{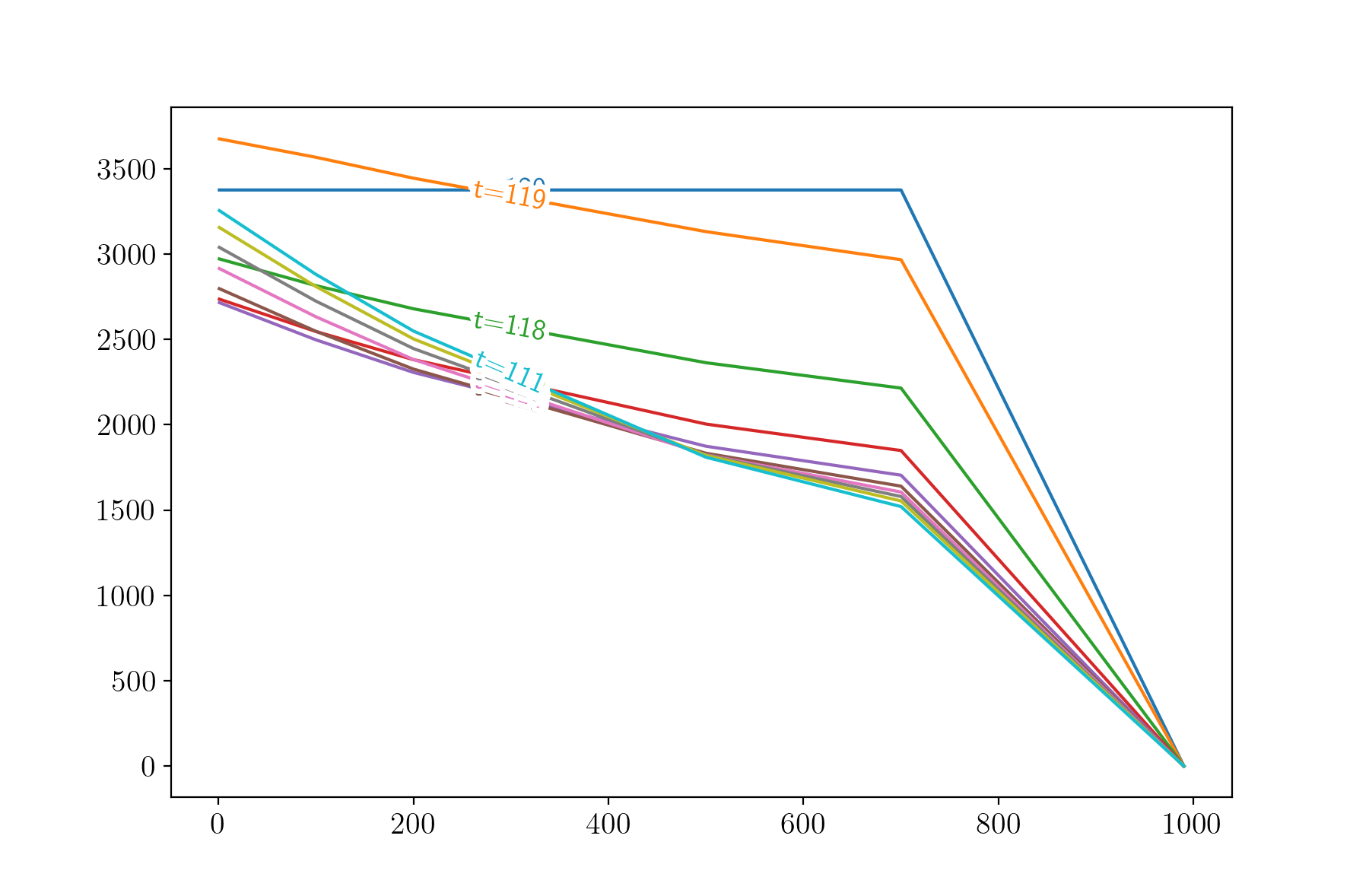}
				\caption{Function of $\z^2$, other variables moderate}
				\label{z2pbfit}
			\end{subfigure}	\\	
			\begin{subfigure}{.5\textwidth}
				\centering
				\includegraphics[width=1.1\linewidth]{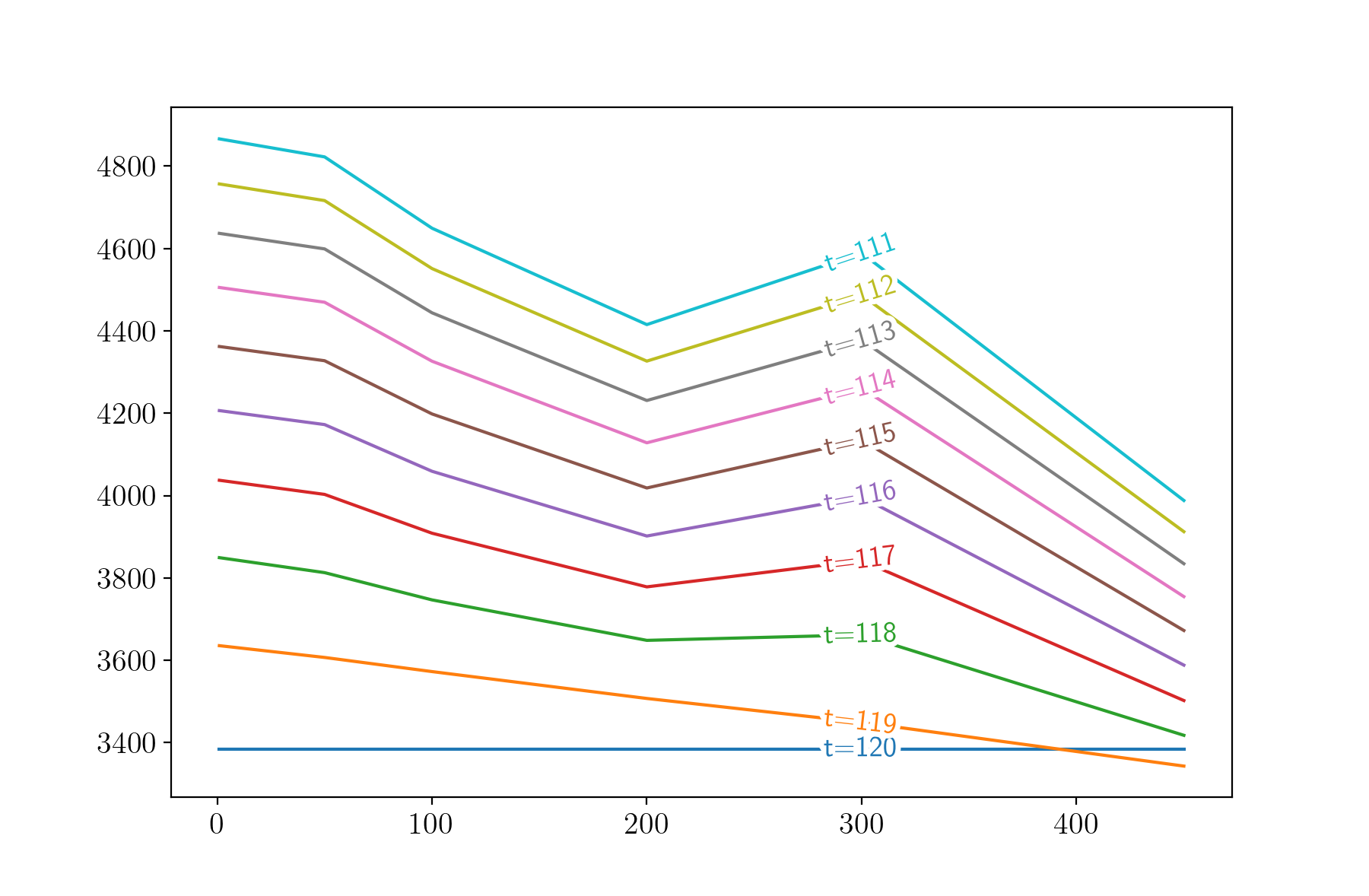}
				\caption{Function of $\m ^2$, other variables moderate}
				\label{y2pd}
			\end{subfigure}\begin{subfigure}{.5\textwidth}
				\centering
				\includegraphics[width=1.1\linewidth]{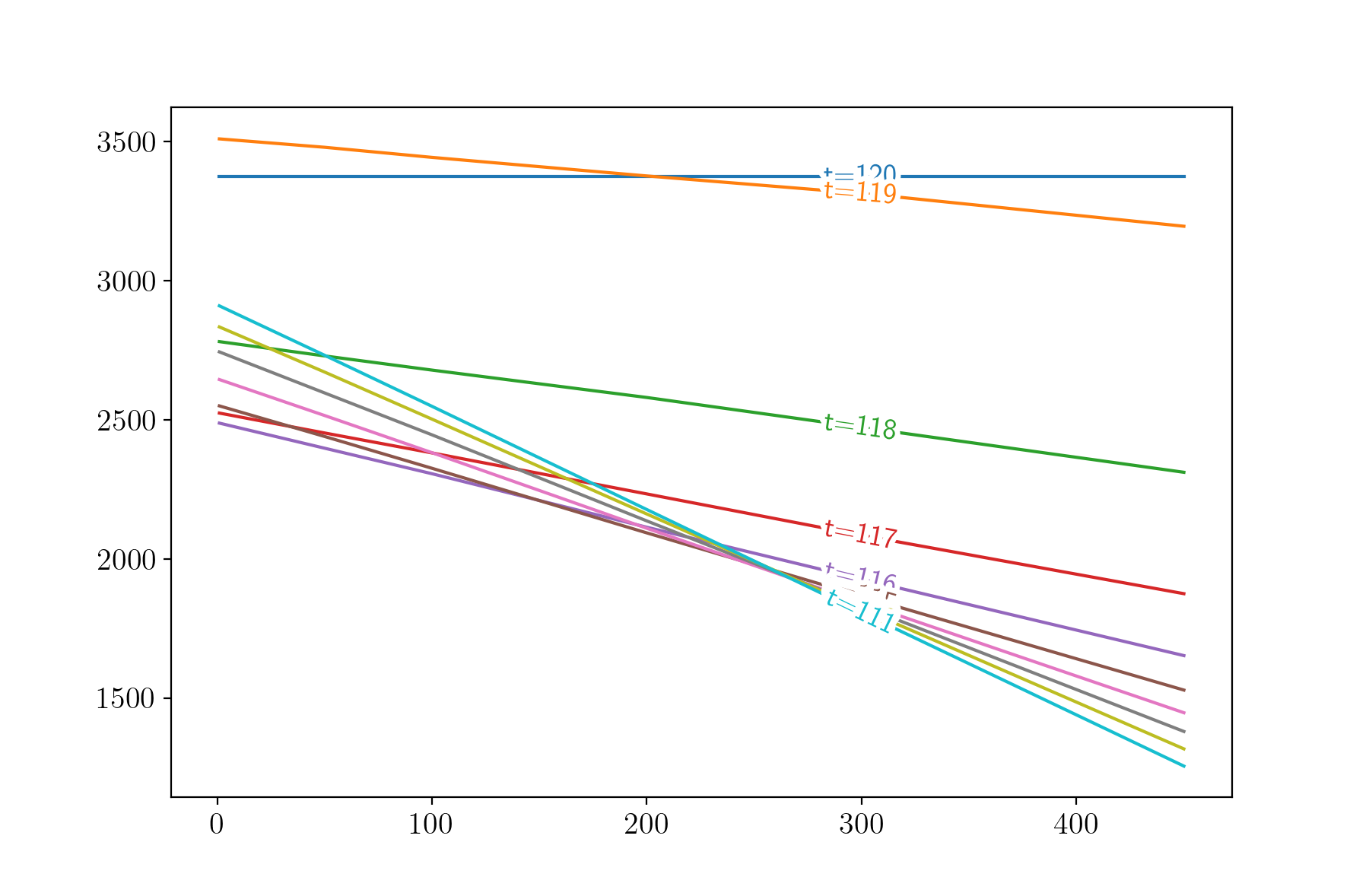}
				\caption{Function of $\m ^2$, other variables moderate}
				\label{y2pdfit}
			\end{subfigure}\\
			\begin{subfigure}{.5\textwidth}
				\centering
				\includegraphics[width=1.1\linewidth]{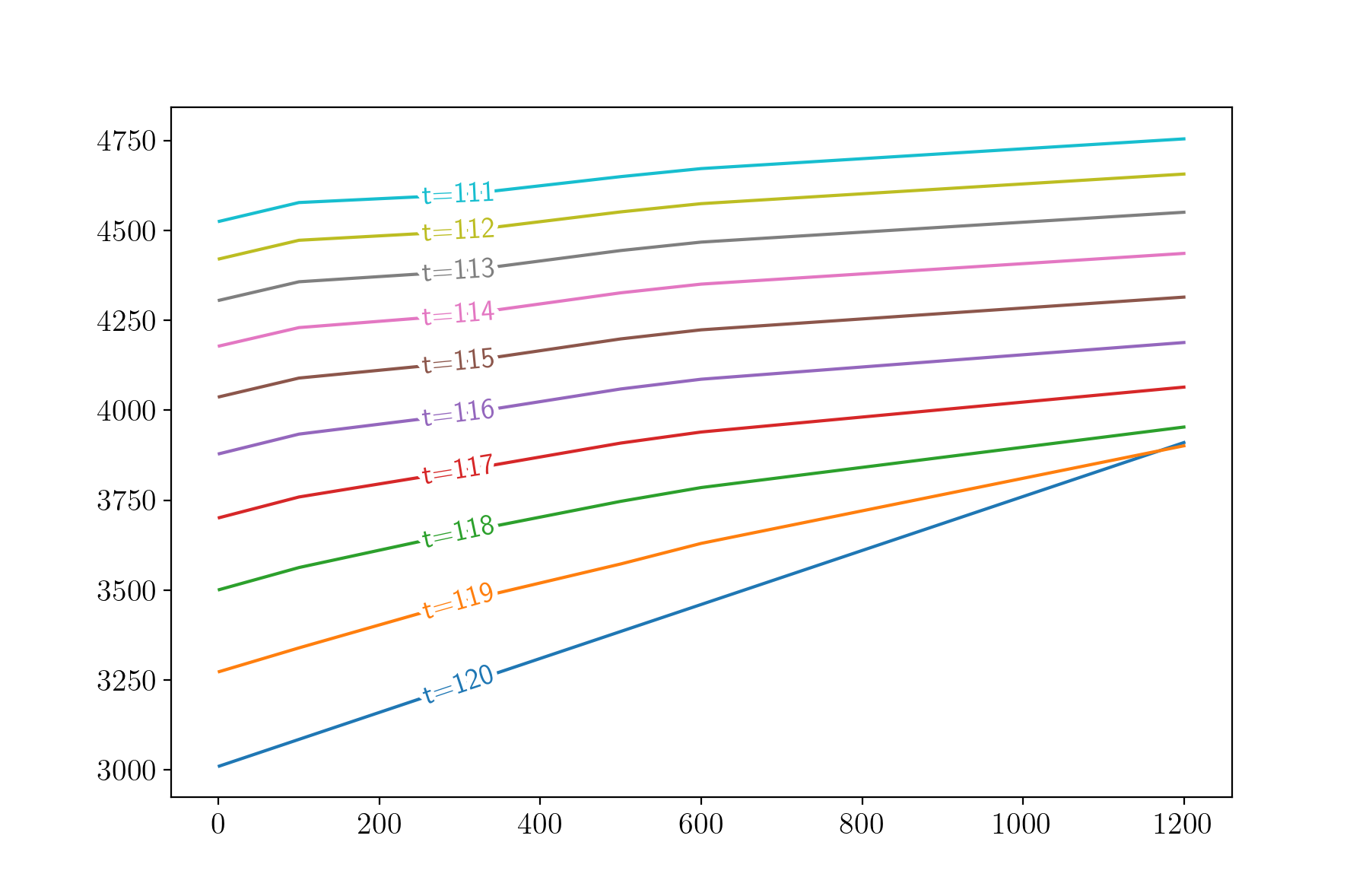}
				\caption{Function of $\q^1$, other variables moderate}
				\label{q1pb}
			\end{subfigure}\begin{subfigure}{.5\textwidth}
				\centering
				\includegraphics[width=1.1\linewidth]{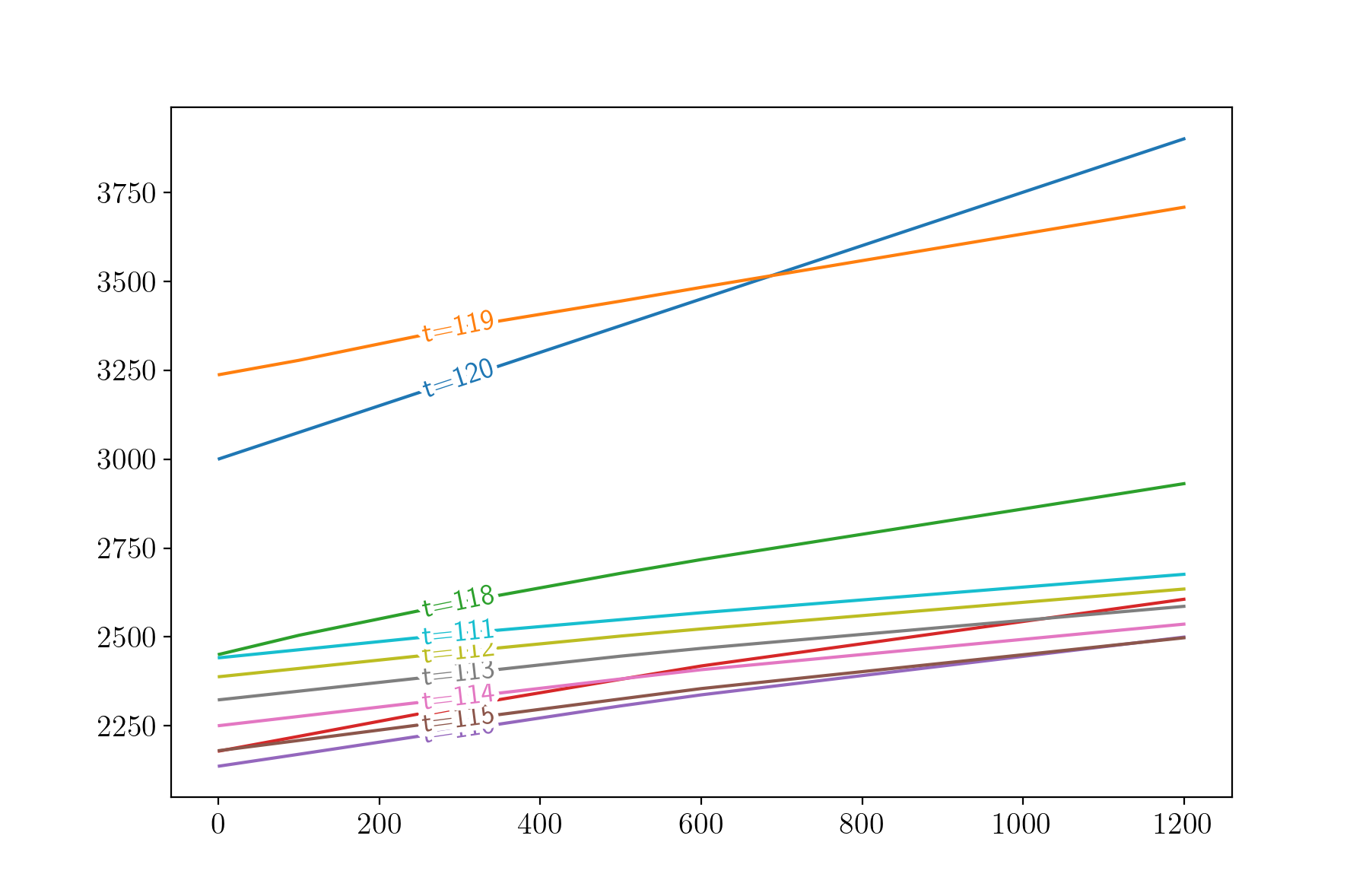}
				\caption{Function of $\q^1$, other variables moderate}
				\label{q1pbfit}
			\end{subfigure}			
			\caption{Value function from linear interpolation approach (left) and regression approach (right) as function of $\z^2$,$\m^2$ and $\q^1$ }
			\label{VFvsZY2fit}
		\end{figure}

		\subsection{Value function obtained with regression method w.r.t. each component of the state}
		
		The behavior of the value function with respect to each compartment can also be visualized using the second method, regression. Figure \ref{VFvsZY1}( Panels \subref{z1pafit},\subref{y1pafit},\subref{z3pafit}), Figure \ref{VFvsZY1fit} ( Panels \subref{z1pbfit},\subref{y1pbfit},\subref{z3pbfit}),
		Figure \ref{VFvsZY2}( Panels \subref{z2pafit},\subref{y2pcfit},\subref{q1pafit}),  and  Figure\ref{VFvsZY2fit} ( Panels \subref{z2pbfit},\subref{y2pdfit},\subref{q1pbfit})   illustrate how the value functions vary w.r.t.  each compartment size, while the other compartments are fixed at small and moderate values, similar to the first method results.		
		We observe that the overall shape of the value function for each compartment aligns with the basis of the ansatz functions. However, some structural changes in the value function over time can be noted. For instance, in Figures \ref{VFvsZY2} (Panels \subref{y2pcfit} and \subref{q1pafit}), the value functions increase as the time is running backward and when the other components are kept small. Conversely, they generally decrease when the other components are kept moderate \ref{VFvsZY2fit} (Panels \subref{y2pdfit} and \subref{q1pbfit}). This behavior differs from the results obtained with linear interpolation.

		\section{List of Abbreviations}
		
		\begin{longtable}{p{0.3\textwidth}p{0.68\textwidth}}
			SIR, SIRS&    Susceptible Infected Recovered with/without  lifelong immunity) \\
			$SI^{\pm}R^{\pm}H$ &  Suceptible Infected detected/non-detected Recovered detected/non-detected Hospitalized  \\
			Covid-19 &  Coronavirus disease 2019 \\
			ODE,SDE  &   Ordinary/Stochastic Differential Equation \\
			SOCP &  Stochastic optimal control problem \\
			MDP&  Markov decision process \\
			CTMC &   Continuous time Markov chain \\
			EKF &  Extended Kalman filter \\
			
			NPI& Non-Pharmaceutical Intervention	\\
			ICU& Intensive Care Unit\\
			HJB& Hamilton Jacobi Bellman		\\
			DPE& Dynamic programming equation\\
			CLVQ& Competitive learning vector quantization\\
			RBF& Radial basis functions\\			
		\end{longtable}
		
		\section{List of Notations}
		\begin{longtable}{p{0.3\textwidth}p{0.68\textwidth}}
			H&  hospitalized compartment \\
			I, $I^-$, $I^+$&  infected compartment, non-detected infected, detected infected compartment, \\
			$R$, $R^-$, $R^+$& Recovered, non-detected, detected recovered\\
			$S$ & Susceptible compartment\\
			$\rho$& correlation coefficient\\
			$I^+_{max}$& upper bound of number of detected infected \\
			$H_{max}$& upper bound of number of hospitalized  \\
			
			$\condmean^1_{max},~ \condmean^2_{max}$ & upper bound of  $\condmean^1,~ \condmean^2$\\ 
			$\condvar^1_{max},~ \condvar^2_{max}$ & upper bound of first and second diagonal entries  $\condvar^1,~ \condvar^2$\\ 
			$\condvar^{12}$ & non-diagonal entry of the covariance matrix\\
			$\mathcal{M}^1,~ \mathcal{M}^2$ & set of discretized values of $\condmean^1, ~ \condmean^2$\\
			$\mathcal{Q}^1,~ \mathcal{Q}^2$ & set of discretized values of $\condvar^1, ~ \condvar^2$\\$\mathcal{R}$ & set of discretized values of $\rho$\\
			$\mathcal{Z}^1,\mathcal{Z}^2,\mathcal{Z}^3$ & sets of discretized values of $Z^1,Z^2,Z^3$\\
			$\widetilde{M}^1$, $\widetilde{M}^2$ & discretization values of $\condmean^1$, $\condmean^2$\\
			$\widetilde{Q}^1$, $\widetilde{Q}^2$ & discretization values of $\condvar^1$, $\condvar^2$\\
			$\widetilde{Z}^1$, $\widetilde{Z}^2$, $\widetilde{Z}^3$& discretization values of $Z^1$, $Z^2$, $Z^3$\\
			$x_{\widetilde m}$ & generic grid point\\
			
			$u^L$& social distancing or lockdown rate\\
			$u^T $& Testing rate\\
			$u^V$& vaccination rate\\
			$u$ & generic control sequence\\ 
			$u^T_{min}$, $u^T_{max}$& minimum and maximum testing rate\\
			$u^V_{max}$ & maximum vaccination rate\\
			$\mathcal{A}$ & set of admissible control\\
			$\mathcal{U}$ & set of feasible control values\\
			
			$u^*$ & optimal control \\
			$\widetilde u^*$ & optimal decision rule \\ 
			$\mathcal{U}^D$ & set of discretized feasible control values\\
			$\beta$& infection rate\\
			$\alpha$& detection rate\\
			$\eta^-$, $\eta^+$& hospitalization rate for non-detected, detected infected\\
			$\gamma^-$, $\gamma^+$, $\gamma^H$& recovery rate for non-detected, detected, hospitalized infected\\
			$\condmean _n$& conditional mean  at time $n\times \Delta t$\\
			$\condvar _n$& conditional covariance  at time $n\times \Delta t$\\
			$n$& discrete time index\\
			$N$ & total population size\\
			$N_t$ & time index for terminal time $T=N_t \times \Delta t$\\
			$X^P$ & state process under partial information \\
			$\mathcal{F}^Z_n$  & observable $\sigma$-algebra at time $n$\\
			$\B_n$& Gaussian sequence\\
			$\Ec_n$ & innovation sequence \\
			$\overline{f}$ & drift coefficient of the hidden state in continuous time\\ $\overline{\sigma}$, 	$\overline{g}$ & first and second diffusion coefficient of the hidden state in continuous time\\
			$ \overline{h}_0$ $ \overline{h}_1$ & terms of   drift coefficient of the observation  in continuous time\\
			$ \overline{\ell}$ &  diffusion coefficient of the observation  in continuous time\\ 
			$f$ & drift coefficient of the hidden state in discrete time\\ 
			$\sigma$, $g$ & first and second   diffusion coefficient of the hidden state in discrete time\\
			$ h_0$ $ h_1$ & terms of   drift coefficient of the observation  in discrete time\\
			$ \ell$ &  diffusion coefficient of the observation  in discrete time\\ 
			$[A ]^+$& pseudoinverse of the matrix $A$\\
			$\mathcal{F}_0^I$ &  initial information $\sigma$-algebra\\
			$\mathbb{I}$&  identity matrix\\
			$m_0$,$q_0$ & initial values of conditional mean and conditional covariance \\
			$C_k$& generic running cost function\\
			$\overline{a}_k$, $a_k$, 	$b_k$  & generic running  cost coefficient\\
			$a_{I^-,N_t}$, $a_{I^+,N_t}$, $b_{I^-,N_t}$, $b_{I^+,N_t}$  & coefficient hidden, detected  infection cost at terminal time\\
			$\overline{x}$& threshold of the number of affected people\\
			
			$C_L$,  $C_T$, $C_V$, $C_H$& running lockdown, test, vaccination, hospitalization costs\\
			$C^{\pm}_I$ & running infection costs (detected/non-detected)\\
			$X^{\text{work}}$& labor force within the population\\
			$\overline{\text{x}}^{I^{\pm}}$ & threshold for the number of infected (detected/non-detected)\\
			$X^{\text{Test}}$ & number of people who ca be tested \\
			$\overline{ \text{x}}^{\text{Test}}$& maximum capacity for testing\\
			$\overline{ \text{x}}^{\text{Vacc}}$& maximum capacity for vaccination\\
			$\overline{ \text{x}}^{\text{H}}$& maximum hospital capacity \\
			$\Psi^F$, $\Phi^F$ & running  and terminal cost under full information\\
			
			$\mathcal{J}^F$  & expected aggregated cost under full information\\
			$\mathbb{F}$ & observable filtration\\
			$\mathcal{J}$  & expected aggregated cost under partial information\\
			$\Psi$,  	$\Phi$& running and terminal  cost under partial information\\
			$f_M,~g_M$ & drift and noise coefficients of the conditional mean \\
			$f_Z,~g_Z$ & drift and noise coefficients of the observation  \\
			$\mathcal{T}$ & transition operator for $X^P$ \\
			$\mathcal{T}^M$, 	$\mathcal{T}^Q$, 	$\mathcal{T}^Z$ & transition operator for conditional mean, conditional covariance, and observation \\
			$\Phi_{\mathcal{N}}$ & CDF of normal distribution\\		
			$\Gamma$ & set of quantization points\\
			$N_l$ & number of quantization points or quantization level\\
			$\xi_l$, 	$w_l$ & quantization points and weights\\
			$\varphi_j$ &  ansatz functions\\
			$\theta_j$ & ansatz functions regression coefficients\\
			$\mathcal{I}$ & set of grid points indexes\\
			$\nu$ & generic value of control\\
			$ W=(W^{1},W^{2})^{\top}$ &   $K$-dimensional standard Brownian motions\\ 
		\end{longtable}

		{\footnotesize

		} 
		
	\end{document}